\newif\ifonlineapp
\onlineapptrue
\ifonlineapp
\documentclass[11pt]{article}
\usepackage{fullpage}
\else
\documentclass[onefignum,onetabnum,siopt]{siamart190516}
\fi
\usepackage[utf8]{inputenc} % allow utf-8 input
\usepackage[T1]{fontenc}    % use 8-bit T1 fonts
\usepackage{booktabs}       % professional-quality tables

\ifonlineapp
\newcommand{\blue}{\color{black}}
\else
\newcommand{\blue}{\color{blue}}
\fi

\def\hx{\hat x}

\ifonlineapp
\usepackage[colorlinks=true,
linkcolor=blue,
urlcolor=blue,
citecolor=blue]{hyperref}       % hyperlinks
\usepackage{url}            % simple URL typesetting
\usepackage{amsthm}
\fi

\usepackage{enumitem,graphicx,bm,amsmath,amssymb,color,shortcuts_OPT,hyperref,cite,dsfont}
\usepackage{cleveref}
\usepackage{mathrsfs}

\ifonlineapp 
\usepackage{xcolor}

\newtheorem{Lemma}{Lemma} 

\newtheorem{Theorem}{Theorem} 

\newtheorem{Corollary}{Corollary}

\newtheorem{assumption}{Assumption}

\else
\newsiamthm{Fact}{Fact}
\newsiamthm{Lemma}{Lemma} 
\newsiamthm{Prop}{Proposition}
\newsiamthm{Theorem}{Theorem} 
\newsiamthm{Def}{Definition}
\newsiamthm{Corollary}{Corollary} 
\newsiamthm{Property}{Property}
\newsiamthm{Observation}{Observation}
\newsiamthm{Exa}{Example}
\newsiamthm{assumption}{Assumption}
\newsiamthm{Remark}{Remark}
\fi

\crefdefaultlabelformat{#2#1#3}

\title{A Two-Timescale Stochastic Algorithm Framework for Bilevel Optimization: Complexity Analysis and Application to Actor-Critic}
\ifonlineapp
\author{%
	Mingyi  Hong\thanks{Authors listed in alphabetical order.}\thanks{University of Minnesota,
		\texttt{email:mhong@umn.edu}} \qquad   \quad  
	Hoi-To  Wai
	\thanks{The Chinese University of Hong Kong,
		\texttt{email:htwai@se.cuhk.edu.hk}} \qquad  \quad  
	Zhaoran Wang\thanks{Northwestern University,  \texttt{email:zhaoranwang@gmail.com}}\qquad   \quad     
	Zhuoran Yang\thanks{Princeton University, \texttt{email:zy6@princeton.edu}}
}
\else
\author{Mingyi Hong, Hoi-To  Wai, Zhaoran Wang and Zhuoran Yang}
\fi
\newcommand{\acm}{\textsf{TT-NAC}}

\newcommand{\beqq}{\begin{equation}}
\newcommand{\eeqq}{\end{equation}}

\let\hat\widehat
\let\tilde\widetilde
\let\check\widecheck

    \makeatletter
    \def\multilimits@{\bgroup
  \Let@
  \restore@math@cr
  \default@tag
 \baselineskip\fontdimen10 \scriptfont\tw@
 \advance\baselineskip\fontdimen12 \scriptfont\tw@
 \lineskip\thr@@\fontdimen8 \scriptfont\thr@@
 \lineskiplimit\lineskip
 \vbox\bgroup\ialign\bgroup\hfil$\m@th\scriptstyle{##}$\hfil\crcr}
    \def\Sb{_\multilimits@}
    \def\endSb{\crcr\egroup\egroup\egroup}
\makeatother

\makeatletter
\DeclareRobustCommand*\cal{\@fontswitch\relax\mathcal}
\makeatother

\begin{document}
	
	\maketitle
	
	\begin{abstract}
		This paper analyzes a two-timescale stochastic algorithm framework for bilevel optimization. 
		Bilevel optimization is a class of problems which exhibits a two-level structure, and its goal is to minimize an outer objective function with variables which are constrained to be the optimal solution to an (inner) optimization problem.
		We consider the case when the inner problem is unconstrained and strongly convex, while the outer problem is constrained and has a smooth objective function. We propose a two-timescale stochastic approximation ({\sf TTSA}) algorithm for tackling such a bilevel  problem. In the algorithm, a stochastic gradient update with a larger step size is used for the inner problem, while a projected stochastic gradient  update with a smaller step size is used for the outer problem. 
		We analyze the convergence rates for the {\sf TTSA} algorithm under various settings: when the outer problem is strongly convex (resp.~weakly convex), the {\sf TTSA} algorithm finds an $\mathcal{O}(K_{\max}^{-2/3})$-optimal  (resp.~$\mathcal{O}(K_{\max}^{-2/5})$-stationary) solution, where $K_{\max}$ is the total iteration number. 
		As an application, we show that a two-timescale natural actor-critic proximal policy optimization algorithm can be viewed as a special case of our {\sf TTSA} framework. Importantly, the natural actor-critic algorithm is shown to converge at a rate of $\mathcal{O}(K_{\max}^{-1/4})$ in terms of the gap in expected discounted reward compared to a global optimal policy. 
	\end{abstract}

\ifonlineapp
\else	
\everydisplay{\small}
\fi
	
\setlength\abovedisplayskip{6pt}
\setlength\belowdisplayskip{6pt}
	
\section{Introduction}\vspace{-.1cm}
Consider bilevel optimization problems of the form:
\beq \label{eq:bilevel}
\min_{x \in X \subseteq \RR^{d_1}} \ell(x) := f(x, y^\star(x)) \quad \mbox{subject to} \quad y^\star(x) \in \argmin_{y \in  \RR^{d_2}} ~g(x,y),
\eeq
where $d_1, d_2 \geq 1$ are integers; $X$ is a closed and convex subset of $\RR^{d_1}$, $f: X \times \RR^{d_2} \rightarrow \RR$ and $g : X \times \RR^{d_2} \rightarrow \RR$ are continuously differentiable functions with respect to (w.r.t.) $x,y$. 
Problem \eqref{eq:bilevel} involves two optimization problems following a two-level structure. We refer to $\min_{y \in \RR^{d_2}} g(x,y)$ as the \emph{inner problem} whose solution depends on $x$, and $g(x,y)$ is called the \emph{inner objective function}; $\min_{x \in X} \ell(x)$ is referred as the \emph{outer problem}, which represents the outer objective function that we wish to minimize and $\ell(x) \equiv f(x, y^\star(x))$ is called the \emph{outer objective function}. 
Moreover, both $f,g$ can be stochastic functions whose gradient may be difficult to compute. Despite being a non-convex stochastic problem in general, \eqref{eq:bilevel} has a wide range of applications, { e.g., reinforcement learning \cite{konda2000actor}, hyperparameter optimization \cite{franceschi2018bilevel}, game theory \cite{stackelberg}, etc..}

Tackling \eqref{eq:bilevel} is challenging as it involves solving the inner and outer optimization problems  \emph{simultaneously}. Even in the simplest case when $\ell(x)$ and  $g(x,y)$ are strongly convex in $x$, $y$, respectively, solving \eqref{eq:bilevel} is difficult. 
For instance, if we aim to minimize $\ell(x) $ via a gradient method,  
at any iterate $x^{\rm cur} \in \RR^{d_1}$ -- applying the gradient method for \eqref{eq:bilevel} involves a double-loop algorithm that {\sf (a)} solves the inner optimization problem $y^\star (x^{\rm cur}) = \argmin_{y \in \RR^{d_2}} g( x^{\rm cur}, y )$ and then {\sf (b)} evaluates the gradient as $\grd \ell(x^{\rm cur})$ based on the solution $y^\star (x^{\rm cur})$. Depending on the application, step {\sf (a)} is usually accomplished by applying yet another gradient method for solving the inner problem (unless a closed form solution for $y^\star (x^{\rm cur})$ exists). In this way, the resulting algorithm necessitates a double-loop structure. 
	
To this end, \cite{ghadimi2018approximation} and the references therein proposed a stochastic algorithm for \eqref{eq:bilevel} involving a \emph{double-loop} update. During the iterations, the inner problem $\min_{y \in \RR^{d_2}} g( x^{\rm cur}, y )$ is solved using a stochastic gradient (SGD) method, with the solution denoted by $\widehat{y}^\star( x^{\rm cur})$. Then, the outer problem is optimized with an SGD update using estimates of $\grd f(x^{\rm cur}, \widehat{y}^\star( x^{\rm cur}) )$. 
Such a double-loop algorithm is proven to converge to a stationary solution, yet a practical issues lingers: \emph{
	% How to select step size and termination criterion for the inner loop?
	What if the (stochastic) gradients of the inner and outer problems are only revealed sequentially? For example, when these problems are required to be updated at the same time such as in a sequential game.} 
	\begin{table}[t] 
		\caption{Summary of the main results. SC stands for strongly convex, WC for weakly convex, C for convex;  $k$ is the iteration counter, $K_{\max}$ is the total number of iterations.}
	\vspace{-0.2cm}	\label{fig:table_compare}
		\footnotesize
		\begin{center}
			\begin{sc}	\begin{tabular}{lcccc}
					\toprule 
					$\ell(x)$ &  Constraint & Step Size ($\alpha_k$, $\beta_k$) & Rate (outer) & Rate (Inner) \\
					\midrule
					SC  & $X \subseteq \mathbb{R}^{d_1}$& ${\cal O}(k^{-1})$,  ${\cal O}(k^{-2/3})$ & $\mathcal{O}(K_{\max}^{-2/3})^\dagger$ & $\mathcal{O}(K_{\max}^{-2/3})^\star$\\
					\midrule
					{C} &  {$X\subseteq \mathbb{R}^{d_1}$} & ${\cal O}(K_{\max}^{-3/4})$,  ${\cal O}(K_{\max}^{-1/2})$ & $\mathcal{O}(K_{\max}^{-1/4})^{\mathparagraph}$ & $\mathcal{O}(K_{\max}^{-1/2})^\star$\\
					\midrule
					WC &  $X\subseteq \mathbb{R}^{d_1}$ & ${\cal O}(K_{\max}^{-3/5})$,  ${\cal O}(K_{\max}^{-2/5})$ & $\mathcal{O}(K_{\max}^{-2/5})^{\#}$ & $\mathcal{O}(K_{\max}^{-2/5})^\star$\\
					\bottomrule
				\end{tabular}\vspace{.2cm}
			\end{sc}
		\end{center}
		$^\dagger$in terms of $\| x^{K_{\max}} - x^\star\|^2$, where $x^\star$ is the optimal solution; $^\star$in terms of $\|y^{K_{\max}} - y^\star(x^{K_{\max}-1})\|^2$, where $y^\star(x^{K_{\max}-1})$ is the optimal inner solution for fixed $x^{K_{\max}-1}$;  
		% $^{\ddagger}$measured using squared gradient norm; 
		$^{\mathparagraph}$measured using $\ell(x)-\ell(x^{\star})$; $^{\#}$measured using distance to a fixed point with the Moreau proximal map $\hat{x}(\cdot)$; see \eqref{eq:moreaum}. 	\vspace{-0.5cm}
	\end{table}
	
	To address the above issues, this paper investigates a \emph{single-loop} stochastic algorithm for \eqref{eq:bilevel}. 
	Focusing on a class of the bilevel optimization problem \eqref{eq:bilevel} where the inner problem is unconstrained and strongly convex, and the outer objective function is smooth, our contributions are three-fold:
	\begin{itemize}[leftmargin=5mm]
	\vspace{-0.2cm}
		\item We study a two-timescale stochastic approximation ({\sf TTSA}) algorithm \cite{borkar1997stochastic} for the concerned class of bilevel optimization. 
		The {\sf TTSA} algorithm updates both outer and inner solutions simultaneously, by using some cheap estimates of stochastic gradients of both inner and outer objectives. 
		The algorithm guarantees convergence by improving the inner (resp., outer) solution with a larger (resp., smaller) step size, also known as using a faster (resp., slower) timescale. 
		\item We analyze the expected convergence rates of the {\sf TTSA} algorithm. Our results are summarized in \Cref{fig:table_compare}.
		Our analysis is accomplished by building a set of \emph{coupled inequalities} for the one-step update in  {\sf TTSA}. For strongly convex outer function, we show inequalities that couple between the outer and inner optimality gaps. 
		For convex or weakly convex outer functions, we establish inequalities coupling between  the difference of outer iterates, the optimality of function values, and the inner optimality gap. 
		We also provide new and generic results for solving coupled inequalities.
		The distinction of timescales between  step sizes of the inner and outer updates plays a crucial role in our convergence analysis.
		\item Finally, we illustrate the application of our analysis results on a two-timescale natural actor-critic policy optimization algorithm with linear function approximation \cite{kakade2002natural, peters2008natural}. The natural  actor-critic algorithm converges at the rate ${\cal O}(K^{-1/4})$ to an optimal policy, which is comparable to the state-of-the-art results. 
	\end{itemize}
	The rest of this paper is organized as follows. 
	\S\ref{sec:ttsa} formally describes the problem setting of bilevel optimization and specify the problem class of interest. In addition, the {\sf TTSA} algorithm is introduced and some application examples are discussed.
	\S\ref{sec:main} presents the main convergence results for the generic bilevel optimization problem \eqref{eq:bilevel}. The convergence analysis is also presented where we highlight the main proof techniques used.
	Lastly, \S\ref{zwsec:RL} discusses the application to reinforcement learning.
	Notice that some technical details of the proof have been relegated to the online appendix \cite{Hong-TTSA-2020}.
	
	\vspace{-0.2cm}
	\subsection{Related Works}
	The study of bilevel optimization problem  \eqref{eq:bilevel} can be traced to that of Stackelberg games \cite{stackelberg}, where the outer (resp.~inner) problem optimizes the action taken by a leader (resp.~the follower). In the optimization literature, bilevel optimization was introduced in \cite{Bracken73} for resource allocation problems, and later studied in  \cite{Bracken74}. Furthermore, bilevel optimization is a special case of the broader class problem of Mathematical Programming with Equilibrium Constraints  \cite{luo_pang_ralph_1996}. 
	
	Many related algorithms have been proposed for bilevel optimization. This includes approximate descent methods \cite{Falk93,Vicente94}, and penalty-based method \cite{White93,Ishizuka92}. The approximate descent methods deal with a subclass of problem where  the outer problem possesses certain (local) differentiability property, while the penalty-based methods approximate the inner problems and/or the outer problems with an appropriate penalty functions. It is noted in  \cite{colson2007overview} that descent based { methods have} relatively strong assumptions about the inner problem (such as non-degeneracy), while the penalty based methods are typically slow. Moreover, these works typically focus on asymptotic convergence analysis, without characterizing the convergence rates;  see  \cite{colson2007overview} for a comprehensive survey.

	In \cite{Ji_ProvablyFastBilevel_Arxiv_2020,ghadimi2018approximation,couellan2016convergence}, the authors considered bilevel problems in the (stochastic) unconstrained setting, when the outer problem is non-convex and the inner problem is  strongly (or strictly) convex. These works are more related to the algorithms and results to be developed in the current paper. In this case, the (stochastic) gradient of the outer problem may be computed using the chain rule. However, to obtain an accurate estimate, one has to either use {\it double loop} structure where the inner loop solves the inner sub-problem to a high accuracy \cite{ghadimi2018approximation,couellan2016convergence}, or use a large batch-size (e.g., $\mathcal{O}(1/\epsilon)$) \cite{Ji_ProvablyFastBilevel_Arxiv_2020}.
	Both of these methods could be difficult to implement in practice { as the batch-size selection or number of required inner loop iterations are difficult to adjust}. In reference \cite{Sabach17}, the authors analyzed a special bilevel problem where there is a single optimization variable in both outer and inner levels. The authors proposed a Sequential Averaging Method (SAM) algorithm which can provably solve a problem with strongly convex outer problem, and convex inner problems. Building upon the SAM, \cite{liu2020generic,Li_ImprovedBilevel_Arxiv_2020} developed first-order algorithms for bilevel problem, without requiring that for each fixed outer-level variable, the inner-level solution must be a singleton. 
	
	In a different line of recent works, references \cite{shaban2019truncated, likhosherstov2020ufoblo} proposed and analyzed different versions of the so-called truncated back-propagation approach for approximating the (stochastic) gradient of the outer-problem, and
	established  convergence for the respective algorithms. The idea is to use a dynamical system to model an optimization algorithm
	that solves the inner problem, and then replace the optimal solution of the inner problem by unrolling a few iterations of the updates. 
	However, computing the {(hyper-)gradient of the objective function $\ell(x)$} requires using back-propagation through the optimization algorithm, and can be computationally very expensive. 
	It is important to note that none of the methods discussed above have considered {\it single-loop} stochastic algorithms, in which a small batch of samples are used to approximate the  inner and outer gradients at each iteration. Later we will see that the ability of being able to update using a small number of samples for both outer and inner problems is critical in a number of applications, { and it is also beneficial numerically.}
	
	In contrast to the above mentioned works, this paper considers  a {\sf TTSA} algorithm for stochastic bilevel optimization, which is a single-loop algorithm employing cheap stochastic estimates of the gradient.
	Notice that {\sf TTSA} \cite{borkar1997stochastic} is a class of algorithms designed to solve coupled system of (nonlinear) equations. While its asymptotic convergence property has been well understood, e.g., \cite{borkar1997stochastic,karmakar2018two,borkar2018concentration}, the convergence rate analysis have been focused on \emph{linear} cases, e.g., \cite{konda2004convergence,dalal2019tale,kaledin2020finite}. 
	In general, the bilevel optimization problem \eqref{eq:bilevel} requires a nonlinear {\sf TTSA} algorithm. For this case, an asymptotic convergence rate is analyzed in \cite{mokkadem2006convergence} under a restricted form of nonlinearity. For convergence rate analysis,  \cite{Sabach17} considered a single-loop algorithm for deterministic bilevel optimization with only one variable, and \cite{doan2020nonlinear} studied the convergence rate when the expected updates are strongly monotone. 
	
	Finally, it is worthwhile mentioning that
	various forms of {\sf TTSA} have been applied to tackle compositional stochastic optimization \cite{wang2017stochastic}, policy evaluation methods \cite{bhatnagar2009convergent, sutton2009fast}, and actor-critic methods \cite{konda2000actor,bhatnagar2008incremental,maei2010toward}. 
	Notice that some of these optimization problems can be cast as a bilevel optimization, as we will demonstrate next.
	\vspace{-0.2cm}
\paragraph{Notations} Unless otherwise noted, $\| \cdot \|$ is the  Euclidean norm on finite dimensional Euclidean space. For a twice differentiable function $f: X \times Y \rightarrow \RR$, $\grd_x f(x,y)$ (resp.~$\grd_y f(x,y)$) denotes its partial gradient taken w.r.t.~$x$ (resp.~$y$), and $\grd_{yx}^2 f(x,y)$ (resp.~$\grd_{xy}^2 f(x,y)$) denotes the Jacobian of $\grd_y f(x,y)$ at $y$ (resp.~$\grd_x f(x)$ at $x$). A function {$\ell(\cdot)$ is said to be weakly convex with modulus $\mu_\ell \in \RR$} if
	\begin{equation}\label{eq:weakly:convex}
	\ell(w) \ge \ell(v) + \langle \nabla \ell(v), w-v\rangle + \mu_\ell \|w-v\|^2, \; \forall~w, v\in X.
	\end{equation}
Notice that if $\mu_\ell \geq 0$ (resp.~$\mu_\ell>0$), then $\ell(\cdot)$ is convex (resp.~strongly convex).
	% \vspace{-.2cm}
	
	\section{Two-Timescale Stochastic Approximation Algorithm for \eqref{eq:bilevel}}
	\label{sec:ttsa}
	To formally define the problem class of interest, we state the following conditions on the bilevel optimization problem \eqref{eq:bilevel}. 
	\begin{assumption}\label{ass:f} The outer functions $f(x,y)$ and $\ell(x) \eqdef f(x, y^\star(x))$ satisfy:
		\begin{enumerate}[leftmargin=8mm]
			\item For any $x \in \RR^{d_1}$, $\nabla_x f(x, \cdot)$ and $\nabla_y f(x, \cdot)$  are Lipschitz continuous with respect to (w.r.t.) $y\in \mathbb{R}^{d_2}$, and with constants $L_{fx}$ and $L_{fy}$, respectively.
			\item For any $y\in \RR^{d_2}$, $\nabla_y f(\cdot, y)$ is Lipschitz continuous w.r.t.~$x \in X$, and with constant $\bar{L}_{fy}$.
			\item For any $x\in X, y\in \RR^{d_2}$, we have $\|\nabla_y f(x,y)\|\le C_{f_y}$, for some $C_{f_y}>0$.
			% \item The function {$\ell(\cdot)$ is weakly convex with modulus $\mu_\ell$}, i.e.,
			% \begin{align}\label{eq:weakly:convex}
			% \ell(w) \ge \ell(v) + \langle \nabla \ell(v), w-v\rangle + \mu_\ell \|w-v\|^2, \; \forall~w, v\in X.
			% \end{align}
		\end{enumerate}
	\end{assumption}
	
	\begin{assumption}\label{ass:g} The inner function $g(x,y)$ satisfies:
		\begin{enumerate}[leftmargin=8mm]
			\item For any $x\in X$ and $y\in \RR^{d_2}$, $g(x,y)$ is twice continuously differentiable in $(x,y)$;
			\item For any $x\in X$, $\nabla_y g(x, \cdot)$ is Lipschitz continuous w.r.t. $y \in \RR^{d_2}$, and with constant $L_g$.
			\item For any $x\in X$, $g(x,\cdot)$ is strongly convex in $y$, and with modulus $\mu_g>0$. 
			\item For any ${x}\in X$, $\nabla^2_{xy}g(x,\cdot)$ and $\nabla^2_{yy}g(x,\cdot)$ are Lipschitz continuous w.r.t. $y \in \RR^{d_2}$, and with constants $L_{gxy}>0$ and $L_{gyy}>0$, respectively. 
			\item For any $y\in\mathbb{R}^m$, $\nabla^2_{xy}g(\cdot,y)$ and $\nabla^2_{yy}g(\cdot,y)$ are Lipschitz continuous w.r.t. $x \in X$, and with constants $\bar{L}_{gxy}>0$ and $\bar{L}_{gyy}>0$, respectively.
			\item For any $x\in X$ and $y\in \mathbb{R}^{d_2}$, we have $\|\nabla^2_{xy}g(x,y)\|\le C_{gxy}$ for some $C_{gxy}>0$. 
		\end{enumerate}
	\end{assumption}
	\vspace{-0.2cm}
% 	In \Cref{ass:f}-4, if $\mu_\ell \geq 0$, then problem \eqref{eq:bilevel} will be convex. 
	Basically, \Cref{ass:f}, \ref{ass:g} require that the inner and outer functions $f,g$ are well-behaved. In particular, $\nabla_x f$, $\nabla_y f$, $\nabla^2_{xy}g$, and $\nabla^2 _{yy} g$ are Lipschitz continuous w.r.t.~$x$ when $y$ is fixed, and Lipschitz continuous w.r.t.~$y$ when $x$ is fixed. {These assumptions are satisfied by common problems in machine learning and optimization, e.g., the application examples discussed in Sec.~\ref{sec:app}.}

%	Lastly, we remark that \Cref{ass:g}.(3) can be relaxed to requiring one-point strong convexity in our analysis. 

	Our first endeavor is to develop a single-loop stochastic algorithm for tackling \eqref{eq:bilevel}. Focusing on solutions which satisfy the first-order stationary condition of \eqref{eq:bilevel}, we aim at finding a pair of solution $(x^\star, y^\star)$ such that 
	\beq \label{eq:opt_cond}
	\grd_y g(x^\star,y^\star) = 0, \quad \langle \grd \ell(x^\star), x - x^\star \rangle \geq 0,~\forall~x \in X.
	\eeq
	Given $x^\star$, a solution $y^\star$ satisfying the first condition in \eqref{eq:opt_cond} may be found by a cheap stochastic gradient recursion such as $y \leftarrow y - \beta h_g$ with $\EE[ h_g ] = \grd_y g(x^\star,y)$. 
	On the other hand, given $y^\star(x)$ and suppose that we can obtain a cheap stochastic gradient estimate $h_f$ with $\EE[h_f] = \grd \ell(x) = \Bgrd f(x,y^\star(x))$, where $\Bgrd f(x,y)$ is a surrogate for $\grd \ell(x)$ (to be described later), then the second condition can be satisfied by a simple projected stochastic gradient recursion as $x \leftarrow P_X( x - \alpha h_f )$, where $P_X(\cdot)$ denotes the Euclidean projection onto $X$.
	
A challenge in designing a \emph{single-loop} algorithm for satisfying \eqref{eq:opt_cond} is to ensure that the outer function's gradient $\Bgrd f(x,y)$ is evaluated at an inner solution $y$ that is close to $y^\star(x)$.
This led us to develop a \emph{two-timescale stochastic approximation} ({\sf TTSA}) \cite{borkar1997stochastic} framework, as summarized in Algorithm 1. An important feature is that the algorithm utilizes two step sizes $\alpha_k$, $\beta_k$ for the outer ($x^k$), inner ($y^k$) solution, respectively, designed with different timescales as $\alpha_k / \beta_k \rightarrow 0$.
As a larger step size is taken to optimize $y^k$, the latter shall stay close to $y^\star(x^k)$. 
Using this strategy, it is expected that $y^k$ will converge to $y^\star (x^k)$ asymptotically. 

	\ifonlineapp 
	\else
	\capstartfalse
	\fi
	\begin{figure}[t]
	\begin{center}
	\vspace{-0cm}
	\setlength\fboxrule{0.0pt}
	\noindent\fcolorbox{black}[rgb]{0.95,0.95,0.95}{\begin{minipage}{0.98\columnwidth}
	\begin{center}
	{\bf Algorithm 1. Two-Timescale Stochastic Approximation ({\sf TTSA})}
	\end{center}
	{\bf S0)} 	Initialize the variable $(x^0, y^0) \in X \times \RR^{d_2}$ and the step size sequence $\{ \alpha_k, \beta_k \}_{k \geq 0}$; \\
	{\bf S1)}	For iteration $k=0,...,K$,
	\begin{subequations} \label{eq:ttsgd}
	\begin{align}
	y^{k+1} & = y^k - \beta_k \cdot h^k_g, \label{eq:y:update}\\
	x^{k+1} & = P_X( x^k - \alpha_k \cdot h^k_f ),
	% \argmin_{x \in X} \big\{ \pscal{h_f^k}{x - x^k} + 1 / (2 \alpha_k) \cdot  \|x - x^k\|^2\big\} 
	\label{eq:x:update}
	\end{align}
	\end{subequations}
	where $h_g^k$, $h^k_f$ are  stochastic estimates of $\nabla_y g(x^k, y^{k})$, $\Bgrd f(x^k,y^{k+1})$ [cf.~\eqref{eq:bar:gradient}], respectively, satisfying \Cref{ass:stoc} given below. Moreover, $P_X(\cdot)$ is the Euclidean projection operator onto the convex set $X$.
	\end{minipage}}
	\vspace{-0.4cm}
	\end{center} 
	\end{figure}
	\ifonlineapp
	\else
	\capstarttrue
	\fi
	
Inspired by \cite{ghadimi2018approximation}, we provide a method for computing a surrogate of $\grd \ell(x)$ given $y$ with general objective functions satisfying \Cref{ass:f}, \ref{ass:g}. Given $y^\star(x)$, we observe that using chain rule, the gradient of $\ell(x)$ can be derived as 
\beq
\nabla \ell(x) = \nabla_x f \bigl (x, y^\star (x) \bigr ) - \nabla^2_{xy} g\bigl (x,y^\star (x) \bigr )\big [\nabla^2_{yy}g\big (x,y^\star (x)  \big )\big]^{-1} \nabla_y f\big (x,y^\star (x) \bigr ). 
\eeq
We note that the computation of the above gradient critically depends on the fact that the inner problem is strongly convex and unconstrained, so that the inverse function theorem can be applied when computing $\nabla y^*(x)$.

We may now define $\Bgrd f(x,y)$ as a surrogate of  $\nabla \ell(x)$  by replacing $y^\star (x)$ with $y \in \RR^{d_2}$:
\begin{align}\label{eq:bar:gradient}
\Bgrd f(x,y) & := \nabla_x f(x,y) -  \nabla^2_{xy} g(x,y)[\nabla^2_{yy}g(x,y)]^{-1} \nabla_y f(x,y).
\end{align} 
Notice that $\grd \ell(x) = \Bgrd f(x, y^\star(x))$. Eq.~\eqref{eq:bar:gradient} is a surrogate for $\grd \ell(x)$ that may be used in {\sf TTSA}. We emphasize that \eqref{eq:bar:gradient} is not the only construction and the {\sf TTSA} can accommodate other forms of gradient surrogate. For example, see \eqref{eq:other_surrogate} in the application of our results to actor-critic.
% 	A natural idea to satisfy the second condition in \eqref{eq:opt_cond} is to apply a simple projected stochastic gradient recursion, e.g., $x \leftarrow P_X( x - \alpha h_f )$ with $\EE[h_f] = \Bgrd f(x,y)$, where $P_X(\cdot)$ denotes the Euclidean projection onto $X$. 

% 	Note that $h_g^k$, $h^k_f$ are stochastic estimates of $\nabla_y g(x^k, y^{k})$, $\Bgrd f(x^k,y^{k+1})$, respectively. 
Let ${\cal F}_k \eqdef \sigma\{ y^0, x^0, ..., y^k, x^k \}$, ${\cal F}_k' \eqdef \sigma\{ y^0, x^0..., y^k, x^k, y^{k+1} \}$ be the filtration of the random variables up to iteration $k$, where $\sigma\{ \cdot \}$ denotes the $\sigma$-algebra generated by the random variables. We consider the following assumption regarding $h_f^k, h_g^k$: 
\begin{assumption} \label{ass:stoc}
For any $k \geq 0$, there exist constants $\sigma_g, \sigma_f$, and a nonincreasing sequence $\{b_k\}_{k \geq 0}$ such that: 
\begin{subequations}\label{eq:property:h}
\begin{align}
~& \mathbb{E} [h^k_g | {\cal F}_k ] = \grd_y g(x^k, y^{k}),~~\mathbb{E}[h^k_f | {\cal F}_k' ] = \Bgrd f(x^k,y^{k+1}) + B_k,~~\| B_k \| \leq b_k ,\label{eq:property:hg1} \\
~& \mathbb{E}[\|h^k_g - \grd_y g(x^k, y^{k})\|^2 | {\cal F}_k ]\le \sigma^2_g \cdot \{ 1 + \| \grd_y g(x^k,y^k)\|^2 \}, \label{eq:property:hg}\\
~& \mathbb{E}[\|h^k_f - B_k - \Bgrd f(x^k,y^{k+1})\|^2 | {\cal F}_k' ]\le \sigma^2_f.\label{eq:property:hf:bound}
\vspace{-.4cm}
\end{align}
\end{subequations}
\end{assumption}
Notice that the conditions on $h_g^k$ are standard when the latter is taken as a stochastic gradient of $g(x^k,y^k)$, while $h_f^k$ is a potentially biased estimate of $\Bgrd f(x^k,y^{k+1})$. As we will see in our convergence analysis, the bias shall decay polynomially to zero. 
	
In light of \eqref{eq:bar:gradient} and as inspired by \cite{ghadimi2018approximation}, we suggest to construct a stochastic estimate of $\Bgrd f(x^k,y^{k+1})$ 
% {\red[why use $\approx$ here?]} 
as follows.
% We declare $x \equiv x^k, y \equiv y^{k+1}$. 
% Assume that we can draw i.i.d.~samples from the distributions $\mu^{(1)}, \mu^{(2)}$. 
% The challenge in constructing an unbiased sample of $\Bgrd f(x^k,y^{k+1})$ lies in the Hessian inverse involved in its calculation, as it is known that $\EE_{\mu^{(2)}} \{  [ \grd_{yy}^2 g(x^k,y^{k+1}; \xi^{(2)})] ^{-1} \} \neq  [ \grd_{yy}^2 g(x^k,y^{k+1}) ] ^{-1} $. 
% Inspired by \cite{ghadimi2018approximation}, 
Let $\tmax(k) \geq 1$ be an integer, ${\rm c}_h \in (0,1]$ be a scalar parameter, and denote $x \equiv x^k, y \equiv y^{k+1}$ for brevity. Consider:
\begin{enumerate}[leftmargin=*, topsep = 0mm]
\item Select ${\sf p} \in \{0, \dots, \tmax(k)-1\}$ uniformly at random and draw $2+{\sf p}$ {independent} samples as $\xi^{(1)} \sim \mu^{(1)}$, $\xi_0^{(2)}, \dots, \xi_{\sf p}^{(2)} \sim \mu^{(2)}$. 
\item Construct the gradient estimator $h_f^k$ as
% \beq \notag
\begin{align*}
& h_f^k = \grd_x f(x,y ; \xi^{(1)}) -  \\
& \grd_{xy}^2 g(x,y; \xi_0^{(2)}) \left [  \frac{ \tmax(k) \, {\rm c}_h }{L_g} \prod_{i=1}^{\sf p} \Big (  I - \frac{ {\rm c}_h }{ L_g } \grd_{yy}^2 g(x,y ; \xi_i^{(2)} )  \Big )  \right ]   \grd_y f(x,y; \xi^{(1)} ),
\end{align*}
% \eeq
where  as a convention, we set $\prod_{i=1}^0 \big( I - \frac{{\rm c}_h}{L_g} \grd_{yy}^2 g(x,y ; \xi_i^{(2)} \big) = I$. 
\end{enumerate}
In the above, the distributions $\mu^{(1)}, \mu^{(2)}$ are defined such that they yield unbiased estimate of the gradients/Jacobians/Hessians as:
\begin{align}
& \grd_x f(x,y) = \EE_{\mu^{(1)}} [ \grd_x f(x,y ; \xi^{(1)}) ], \quad \grd_y f(x,y) = \EE_{\mu^{(1)}} [ \grd_y f(x,y ; \xi^{(1)}) ] ,\label{eq:add_hfe_main} \\
& \grd_{xy}^2 g(x,y) = \EE_{\mu^{(2)}} [ \grd_{xy}^2 g(x,y; \xi^{(2)}) ], \quad \grd_{yy}^2 g(x,y) = \EE_{\mu^{(2)}} [ \grd_{yy}^2 g(x,y; \xi^{(2)}) ], \nonumber
\end{align}
and satisfying $\EE[ \| \grd_y f(x,y; \xi^{(1)}) \|^2] \leq C_y$, $\EE[ \| \grd_{xy}^2 g(x,y; \xi^{(2)}) \|^2] \leq C_g$, 
\begin{align} \label{eq:add_hf_main}
	\begin{split}
	& \EE[ \| \grd_x f(x,y) - \grd_x f(x,y; \xi^{(1)}) \|^2 ] \leq \sigma_{fx}^2, \quad \EE[ \| \grd_y f(x,y) - \grd_y f(x,y; \xi^{(1)}) \|^2 ] \leq \sigma_{fy}^2 , \\
	& \EE[ \| \grd_{xy}^2 g(x,y) - \grd_{xy}^2 g(x,y; \xi^{(2)}) \|_2^2 ] \leq \sigma_{gxy}^2 , 
	\end{split}
\end{align}
note that $\|\cdot \|_2$ is the Schatten-2 norm. For convenience of analysis, we assume $\frac{\mu_g}{\mu_g^2 + \sigma_{gxy}^2} \leq 1$, $L_g \geq 1$. The next lemma shows that $h_f^k$ satisfies \Cref{ass:stoc}.  
\begin{Lemma} \label{lem:hessinv_main}
	Under \Cref{ass:f}, \ref{ass:g}, \eqref{eq:add_hfe_main}, \eqref{eq:add_hf_main}, and ${\rm c}_h = {\mu_g} / ({\mu_g^2 + \sigma_{gxy}^2} )$, then for any $x \in X, y \in \RR^{d_2}, \tmax(k) \geq 1$, it holds that
	\beq \label{eq:bias_hfk_main}
	\left\| \Bgrd f(x^k,y^{k+1}) - \EE[ h_f^k ] \right\| \leq C_{gxy} C_{fy} \cdot \frac{1}{\mu_g} \cdot  \Big( 1 - \frac{\mu_g^2}{L_g (\mu_g^2+\sigma_{gxy}^2)} \Big)^{\tmax(k)}.
	\eeq
	Furthermore, the variance is bounded as 
	\beq \label{eq:var_hfk_main}
	\EE[ \| h_f^k - \EE[h_f^k] \|^2 ] \leq \sigma_{fx}^2 + \Big [   (\sigma_{fy}^2 + C_y^2)  \{ \sigma_{gxy}^2 + 2 C_{gxy}^2 \} + \sigma_{fy}^2 C_{gxy}^2 \Big ] \max\big\{ \frac{3}{\mu_g^2}, \frac{3d_1 / L_g}{\mu_g^2 + \sigma_{gxy}^2} \big\} .
	\eeq
\end{Lemma}
The proof of the above lemma is relegated to our online appendix, see \S E in \cite{Hong-TTSA-2020}.  Note that the variance bound \eqref{eq:var_hfk_main} relies on analyzing the expected norm of product of random matrices using the techniques inspired by \cite{durmus2021tight, huang2021matrix}. Finally, observe that the upper bounds in \eqref{eq:bias_hfk_main}, \eqref{eq:var_hfk_main} correspond to $b_k$, $\sigma_f^2$ in \Cref{ass:stoc}, respectively, and the requirements on $h_f^k$ are satisfied. 
% The additional conditions in \eqref{eq:add_hf} are standard as they only require the stochastic gradient/Jacobian/Hessian to have bounded variance.
	
To further understand the property of the {\sf TTSA} algorithm with \eqref{eq:bar:gradient}, we borrow the following results from \cite{ghadimi2018approximation} on the Lipschitz continuity of the maps $\grd \ell(x)$, $y^\star(x)$: 
\begin{Lemma} \cite[Lemma 2.2]{ghadimi2018approximation} \label{lem:lips}
Under \Cref{ass:f}, \ref{ass:g}, it holds
\begin{subequations}
\begin{align}
\|\Bgrd f(x,y) - \nabla \ell(x)\| \le L\|y^\star(x) -y\|,~~\|y^\star(x_1)-y^\star(x_2)\| & \le L_y \|x_1-x_2\|, \label{eq:lip:f:bar}\\
\|\nabla \ell(x_1)-\nabla \ell(x_2)\|= \|{\nabla} f(x_1, y^\star(x_1)) - \nabla f(x_2, y^\star(x_2))\|& \le L_f\| x_1 -x_2\|. \label{eq:lip:f} %\\
%. \label{eq:lip:y}
\end{align}
\end{subequations}
for any $x, x_1,x_2 \in X$ and $y \in \RR^{d_2}$, where we have defined
\begin{align} \label{eq:lips_const}
\begin{split}
L := L_{f_x} + \frac{L_{f_y} C_{g_{xy}}}{\mu_g}  & + C_{f_y} \bigg ( \frac{ L_{g_{xy}} }{\mu_g}   + \frac{ L_{g_{yy}} C_{g_{xy}} }{ \mu_g^2 } \bigg ) ,\\ 
L_f : = L_{f_x} + \frac{(\bar{L}_{f_y} + L) C_{g_{xy}}}{\mu_g} & + C_{f_y} \bigg (  \frac{ \bar{L}_{g_{xy}} }{\mu_g} + \frac{ \bar{L}_{g_{yy}} C_{g_{xy}} }{ \mu_g^2 } \bigg )  , \quad L_y = \frac{C_{g_{xy}}}{\mu_g}.
\end{split}
\end{align}
\end{Lemma}
\noindent The above properties will be pivotal in establishing the convergence of {\sf TTSA}. 
First, we note that \eqref{eq:lip:f} implies that the composite function $\ell(x)$ is weakly convex with a modulus that is at least $(-L_f)$.
% We remark that alternative forms of approximation to $\grd \ell(x)$ may be used to replace the gradient surrogate in \eqref{eq:bar:gradient}. 
Furthermore, \eqref{eq:property:hf:bound} in \Cref{ass:stoc} combined with \Cref{lem:lips} leads to the following estimate:
\beq \label{eq:hfbd}
\EE[ \| h_f^k \|^2 | {\cal F}_k' ] \leq \tilde\sigma_f^2 + 3 b_k^2 + 3 L^2 \| y^{k+1} - y^\star(x^k) \|^2, \quad \tilde\sigma_f^2 \eqdef \sigma_f^2 + 3 \sup_{x \in X} \| \grd \ell(x) \|^2.
\eeq
Throughout, we assume $\tilde\sigma_f^2$ is bounded, e.g., it can be satisfied if $X$ is bounded, or if $\ell(x)$ has bounded gradient.\vspace{-.2cm}
	
\subsection{Applications} \label{sec:app} Practical problems such as hyperparameter optimization \cite{franceschi2018bilevel,mehra2019penalty,shaban2019truncated}, Stackelberg games \cite{stackelberg} can be cast into special cases of bilevel optimization problem \eqref{eq:bilevel}. To be specific, we discuss three applications of the bilevel optimization problem \eqref{eq:bilevel} below. 
	
\vspace{-0.2cm}
\paragraph{Model-Agnostic Meta-Learning}
An important paradigm of machine learning is to find model that adapts to multiple training sets in order to achieve the best performance for individual tasks. Among others, a popular formulation is model-agnostic meta learning (MAML) \cite{finn2017model} which minimizes an outer objective of empirical risk on all training sets, and the inner objective is the one-step projected gradient. Let $D^{(j)} = \{z_i^{(j)}\}_{i=1}^n$ be the $j$-th ($j\in [J]$) training set with sample size $n$, MAML can be formulated as a bilevel optimization problem \cite{rajeswaran2019meta}:
\beqq
	\begin{array}{rl} 
		\ds \min_{\theta \in \Theta} & \ds \sum_{j=1}^J \sum_{i=1}^n\bar{\ell}\big (\theta^{\star (j)}(\theta),z_i^{(j)}\big ) \\
		\mbox{subject to} & \ds \theta^{\star (j)}(\theta) \in \argmin_{\theta^{(j)}}~\biggl\{  {\ds\sum_{i=1}^n}\langle \theta^{(j)}, \nabla_\theta\bar{\ell}(\theta, z_{i}^{(j)})\rangle + \frac{\lambda}{2} \| \theta^{(j)} - \theta \|^2
		\biggr\}.
	\end{array}
\eeqq
Here $\theta$ is the shared model parameter, $\theta^{(j)}$ is the adaptation of $\theta$ to the $j$th training set, and $\bar{\ell}$ is the loss function. It can be checked that the inner problem is strongly convex. We have \Cref{ass:f}, \ref{ass:g}, \ref{ass:stoc} for stochastic gradient updates, assuming $\bar{\ell}$ is sufficiently regular, and the losses are the logistic loss. Moreover, \cite{finn2019online} proved that, assuming $\lambda$ is sufficiently large and $\bar{\ell}$ is strongly convex, the outer problem is also strongly convex. In fact, \cite{raghu2019rapid} demonstrated that an algorithm with no inner loop achieves a comparable performance to \cite{finn2019online}.
	
\vspace{-0.2cm}
\paragraph{Policy Optimization} 
Another application of the bilevel optimization problem is the policy optimization problem, particularly when combined with an actor-critic scheme. The optimization involved is to find an optimal policy to maximize the expected (discounted) reward. Here, the `actor' serves as the outer problem and the `critic' serves as the inner problem which evaluates the performance of the `actor' (current policy).
To avoid redundancy, we refer our readers to \S\ref{zwsec:RL} where we present a detailed case study. The latter will also shed lights on the generality of our proof techniques for {\sf TTSA} algorithms.

\paragraph{Data hyper-cleaning}
The data hyper-cleaning problem trains a classifier with a dataset of randomly corrupted labels \cite{shaban2019truncated}. The problem formulation is given below: 
\begin{align}\label{eq:clean}
\textstyle \min_{x \in \mathbb{R}^{d_1}}  &~~ \textstyle \ell(x) := \sum_{i \in \mathcal{D}_{\text{val}}} L(a_i^\top y^\star(x), b_i) \\
\text{s.t.} &~~ \textstyle y^\star(x) = \argmin_{y \in \mathbb{R}^{d_2}}  \big\{ \lambda \| y \|^2 + \sum_{i\in \mathcal{D}_{\text{tr}}} \sigma(x_i) L(a_i^\top y , b_i) \big\} .\nonumber
\end{align}
In this problem, we have $d_1 = |\mathcal{D}_{\rm tr}|$, and $d_2$ is the dimension of the classifier $y$;
$(a_i,b_i)$ is the $i$th data point; $L(\cdot)$ is the loss function; $x_i$ is the parameter that determines the weight for the $i$th data sample; $\sigma: \mathbb{R} \rightarrow \mathbb{R}_+$ is the weight function; $\lambda >0$ is a regularization parameter; $\mathcal{D}_{\rm val}$ and $\mathcal{D}_{\rm tr}$ are validation and training sets, respectively.
Here, the inner problem finds the classifier $y^\ast(x)$ with the training set ${\cal D}_{\sf tr}$, while the outer problem finds the best weights $x$ with respect to the validation set ${\cal D}_{\sf val}$.

Before ending this subsection, let us mention that, we do not aware any general sufficient conditions that can be used to verify whether the outer function $\ell(x)$ is (strongly) convex or not. To our knowledge, the convexity of $\ell(x)$ has to be verified in a case-by-case manner; please see \cite[Appendix B3]{finn2019online} for how this can be done.
 
\section{Main Results} \label{sec:main}
This section presents the convergence results for {\sf TTSA} algorithm for \eqref{eq:bilevel}. We first summarize a list of important constants in \Cref{tab:constant} to be used in the forthcoming analysis. Next, we discuss a few concepts pivotal to our analysis.
	% To quantify convergence, we rely on a few metrics developed as follows.
	
\vspace{-0.2cm}
\paragraph{Tracking Error} {\sf TTSA}   tackles the inner and outer problems simultaneously using single-loop updates. Due to the coupled nature of the inner and outer problems, in order to obtain an upper bound  on the \emph{optimality gap} $\Delta_x^k := \EE[\| x^k - x^\star \|^2]$, where $x^\star$ is an optimal solution to \eqref{eq:bilevel}, we need to estimate the \emph{tracking error} defined as  
\beq  
\Delta_y^k := \EE[ \| y^k - y^\star(x^{k-1}) \|^2 ] \quad \text{where} \quad y^\star(x)   = \argmin_{y \in \RR^{d_2}} ~g(x, y).
\eeq
For any $x \in X$, $y^\star(x)$ is well defined since the inner problem is strongly convex due to \Cref{ass:g}. By definition, $\Delta_y^k$ quantifies how close $y^k$ is from  the optimal solution to inner problem given $x^{k-1}$. 
	
\vspace{-0.2cm}
\paragraph{Moreau Envelop} 
Fix $\rho > 0$, define the Moreau envelop and proximal map as 
\beq \label{eq:moreaum}
	\Phi_{1/\rho} (z) := \min_{x \in X} \big\{ \ell(x) + (\rho/2) \|x-z\|^2 \big\},~~\hat{x}(z) : =  \argmin_{x \in X} \big\{ \ell(x) + (\rho/2) \|x-z\|^2 \big\}.
\eeq
For any $\epsilon > 0$, $x^k \in X$ is said to be an \emph{$\epsilon$-nearly stationary solution} \cite{davis2018stochastic} if  $x^k$ is an approximate fixed point of $\{ \hat{x} - {\rm I} \}(\cdot)$, where
\beq \label{eq:nearstat}
\tilde{\Delta}_x^k := \EE[\| \hat{x}( x^k ) - x^k \|^2] \leq  \rho^{-2} \cdot \epsilon  .
\eeq 
We observe that if $\epsilon = 0$, then $x^k \in X$ is a stationary solution to \eqref{eq:bilevel} satisfying the second condition in \eqref{eq:opt_cond}. As we will demonstrate next, the \emph{near-stationarity} condition \eqref{eq:nearstat} provides an apparatus to quantify the finite-time convergence of {\sf TTSA} in the case when $\ell(x)$ is non-convex. 

\begin{table}
	\centering
	\caption{Summary of the Constants for \Cref{sec:main}} \label{tab:constant}
	{\footnotesize \begin{tabular}{l l l }
		\toprule
		\bfseries Constant & \bfseries Description & \bfseries Reference \\
		\midrule 
		$L_{fx}, L_{fy}$ & Lipschitz constants for $\grd_x f(x,\cdot)$, $\grd_y f(x,\cdot)$ w.r.t.~$y$, resp. & \Cref{ass:f}\\
		$\bar{L}_{fy}$ & Lipschitz constants for $\grd_y f(\cdot,y)$ w.r.t.~$x$ & \Cref{ass:f} \\
		$C_{fy}$ & Upper bound on $\| \grd_y f(x,y) \|$ & \Cref{ass:f} \\
		$L_g$ & Lipschitz constant of $\grd_y g(x,\cdot)$ & \Cref{ass:g} \\
		$\mu_g$ & Strong convexity modulus $g(x,\cdot)$ w.r.t.~$y$ & \Cref{ass:g} \\
		$L_{gxy}, L_{gyy}$ & Lipschitz constants of $\grd_{xy}^2g(x,\cdot), \grd_{yy}^2 g(x,\cdot)$ w.r.t.~$y$, resp. & \Cref{ass:g} \\
		$\bar{L}_{gxy}, \bar{L}_{gyy}$ & Lipschitz constants of $\grd_{xy}^2 g(\cdot,y), \grd_{yy}^2 g(\cdot,y)$ w.r.t.~$x$, resp. & \Cref{ass:g} \\
		$C_{gxy}$ & Upper bound on $\| \grd_{xy}^2 g(x,y) \|$ & \Cref{ass:g} \\
		$b_k$ & Bound on the bias of $h_f^k$ at iteration $k$ & \Cref{ass:stoc} \\ 
		$\sigma_g^2, \sigma_f^2$ & Variance of stochastic estimates $h_g^k$, $h_f^k$, resp. & \Cref{ass:stoc} \\
		$\tilde\sigma_f^2$ & Constant term on the bound for $\EE[ \| h_f^k\|^2 ]$ & \eqref{eq:hfbd} \\
		$L$ & Difference between $\Bgrd f(x,y)$, $\grd \ell(x)$ w.r.t.~$\| y^\star(x) - x \|$ & \Cref{lem:lips} \\
		$L_y$ & Lipschitz constant of $y^\star(x)$ & \Cref{lem:lips} \\
		$L_f$ & Lipschitz constant of $\grd \ell(x)$ & \Cref{lem:lips} \\
		\bottomrule
	\end{tabular} }
\end{table}
	 
\subsection{Strongly Convex Outer Objective Function} \label{sec:scsc}
Our first result considers the instance of \eqref{eq:bilevel} where $\ell(x)$ is strongly convex. We obtain:
\begin{Theorem}\label{th:sc:uc}
Under \Cref{ass:f}, \ref{ass:g}, \ref{ass:stoc}. Assume that $\ell(x)$ is weakly convex with a modulus $\mu_\ell>0$ (i.e., it is strongly convex), and the step sizes satisfy\begin{subequations} \label{eq:stepsize}
\begin{align}
& \alpha_k \leq {\rm c}_0 \,  \beta_k^{3/2},~\beta_k \leq {\rm c}_1 \alpha_k^{2/3 }, ~ \frac{ \beta_{k-1} }{ \beta_k } \leq 1 + \beta_k \mu_g / 8,~\frac{\alpha_{k-1} }{ \alpha_k } \leq 1 + 3 \alpha_k \mu_\ell / 4, \label{eq:stepsize0} \\ 
& \alpha_k \leq \frac{1}{\mu_\ell}, ~ \beta_k \leq \min \left\{ \frac{1}{ \mu_g}, \frac{ \mu_g}{ L_g^2(1+\sigma_g^2)}, \frac{ \mu_g^2 }{48 {\rm c}_0^2 L^2 L_y^2 } \right\},~ 8 \mu_\ell \alpha_k \leq \mu_g \beta_k,~ \forall~k \geq 0, \label{eq:stepsize1}
\end{align}
\end{subequations}
where the constants $L, L_y$ were defined in \Cref{lem:lips} and ${\rm c}_0, {\rm c}_1 > 0$ are free parameters. 
If the bias is bounded as $b_k^2 \leq \tilde{\rm c}_b \alpha_{k+1}$, then for any $k \geq 1$, the {\sf TTSA} iterates satisfy 
\begin{align} \label{eq:thm31bd}
\begin{split}
& \Delta_x^k \lesssim \prod_{i=0}^{k-1} (1 - \alpha_i \mu_\ell ) \big[ \Delta_x^0 + \frac{L^2}{\mu_\ell^2} \Delta_y^0 \big] + \frac{ {\rm c}_1 L^2 }{ \mu_\ell^2 } \Big[ \frac{\sigma_g^2}{\mu_g} + \frac{ {\rm c}_0^2 L_y^2 }{\mu_g^2} \tilde{\sigma}_f^2 \Big] \alpha_{k-1}^{2/3}, \\
& \Delta_y^k \lesssim \prod_{i=0}^{k-1} (1 - \beta_i \mu_g / 4) \Delta_y^0 + \Big[ \frac{\sigma_g^2}{\mu_g} + \frac{{\rm c}_0^2 L_y^2}{\mu_g^2} \tilde{\sigma_f}^2 \Big] \beta_{k-1} ,
\end{split}
\end{align}
where the symbol $\lesssim$ denotes that the numerical constants are omitted (see \Cref{sec:gen_ana}).
\end{Theorem}

\noindent Notice that the bounds in \eqref{eq:thm31bd} show that the expected optimality gap and tracking error at the $k$th iteration shall compose of a transient and fluctuation terms. For instance, for the bound on $\Delta_x^k$, the first (transient) term decays sub-geometrically as $ \prod_{i=0}^{k-1} (1 - \alpha_i \mu_\ell )$, while the second (fluctuation) term scales as $\alpha_{k-1}^{2/3}$. {\blue Note that if $\alpha_k, \beta_k \to 0$, then the r.h.s.~in \eqref{eq:thm31bd} converges to zero as $k \to \infty$. While for non-vanishing step sizes, the r.h.s.~in \eqref{eq:thm31bd} may not converge to zero.}
	
The conditions in \eqref{eq:stepsize} are satisfied by both \emph{diminishing} and \emph{constant} step sizes. For example, we define the constants:
% we set $\alpha_k = c_\alpha / ( k + k_\alpha), \beta_k = c_\beta / ( k + k_\beta)^{2/3}$ with 
\beq \label{eq:kalpha}
	k_\alpha = \max\Big\{ 35 \Big( \frac{L_g}{\mu_g} \Big)^3 (1 + \sigma_g^2)^{\frac{3}{2}} , \frac{ (512)^{\frac{3}{2}} L^2 L_y^2 }{ \mu_\ell^2} \Big\},~ c_\alpha = \frac{8}{3 \mu_\ell},~
	% \eeq
	% \beq \label{eq:kbeta}
	k_\beta = \frac{1}{4} k_\alpha,~ c_\beta = \frac{32}{3 \mu_g}.
\eeq
Then, for \emph{diminishing step sizes}, we set $\alpha_k = c_\alpha / ( k + k_\alpha), \beta_k = c_\beta / ( k + k_\beta)^{2/3}$, and for \emph{constant step sizes}, we set $\alpha_k = c_\alpha / k_\alpha$, $\beta_k = c_\beta / k_\beta^{2/3}$. Both pairs of the step sizes satisfy \eqref{eq:stepsize} with ${\rm c}_0 = \frac{\mu_g^{3/2}}{\mu_\ell}$, ${\rm c}_1 = 10 \frac{\mu_\ell^{2/3}}{\mu_g} $. For \emph{diminishing step sizes}, \Cref{th:sc:uc} shows the {last iterate convergence} rate for the optimality gap and the tracking error to be ${\cal O}(k^{-2/3})$. 
To compute an $\epsilon$-optimal solution with $\Delta_x^k \leq \epsilon$, the {\sf TTSA} algorithm with diminishing step size requires a total of ${\cal O}( \log(1/\epsilon) /\epsilon^{3/2} )$ calls of stochastic (gradient/Hessian/Jacobian) oracles of both outer  ($f(\cdot,\cdot)$) and inner ($g(\cdot,\cdot)$) functions\footnote{Notice that as we need $b_k = {\cal O}( \sqrt{\alpha_{k+1}})$, from \Cref{lem:hessinv_main}, the polynomial bias decay requires using $\tmax(k) = {\cal O}(1+\log(k))$ samples per iteration, justifying the $\log$ factor in the bound.}.
	
While this is arguably the easiest case of \eqref{eq:bilevel}, we notice that the double-loop algorithm in \cite{ghadimi2018approximation} requires ${\cal O}( 1/ \epsilon )$, ${\cal O}( 1/ \epsilon^2 )$ stochastic oracles for the outer (i.e., $f(\cdot,\cdot)$), inner (i.e., $g(\cdot,\cdot)$) functions, respectively. As such, the {\sf TTSA} algorithm requires less number of stochastic oracles for the inner function. 
	
\subsection{Smooth (Possibly Non-convex) Outer Objective Function} \label{sec:const} 
We focus on the case where $\ell(x)$ is weakly convex. We obtain 
	
\begin{Theorem}\label{th:wc:c}
Under \Cref{ass:f}, \ref{ass:g}, \ref{ass:stoc}, assume that $\ell(\cdot)$ is weakly convex with modulus $\mu_\ell \in \RR$. Let $\Kmax \geq 1$ be the maximum iteration number and set 
\beq \label{eq:stepsize_const}
\alpha = \min\Big\{ \frac{\mu_g^2}{8 L_y L L_g^2(1+\sigma_g^2)},  \frac{1}{4 L_y L} \Kmax^{-3/5} \Big\}, \; \; \beta = \min \Big\{ \frac{\mu_g}{ L_g^2 (1 + \sigma_g^2) }, \frac{2}{\mu_g} \Kmax^{-2/5} \Big\}.
\eeq 
If $b_k^2 \le \alpha$, then for any $\Kmax \geq 1$, the iterates from the {\sf TTSA} algorithm satisfy 
\beqq 
\EE[ \widetilde{\Delta}_x^{\sf K} ] = {\cal O}( \Kmax^{-2/5}),~~
{\EE[\Delta_y^{\sf K+1}]} = {\cal O}( \Kmax^{-2/5}), 
\eeqq
where ${\sf K}$ is an independent uniformly distributed random variable on $\{0,...,\Kmax-1\}$; and we recall $\tilde{\Delta}^k_x:=\| \hat{x}(x^{k} ) - x^{ k}\|^2$. When $\Kmax$ is large and $\mu_\ell <0$, setting $\rho = 2|\mu_\ell|$ yields
\beqq
\EE[ \widetilde{\Delta}_x^{\sf K} ] \lesssim \Big[ L^2 \Big( \Delta^0 + \frac{\sigma_g^2}{\mu_g^2} \Big) + \mu_g \tilde\sigma_f^2  \Big] \frac{ \Kmax^{-\frac{2}{5} } }{ |\mu_\ell|^2 }, \quad \EE[ \Delta_y^{\sf K+1} ] \lesssim \Big[ \frac{\Delta^0}{\mu_g} + \frac{ \sigma_g^2 }{ \mu_g^2 } + \frac{\mu_g \sigma_f^2}{L^2} \Big] \Kmax^{-\frac{2}{5}},
\eeqq
where we defined $\Delta^0 := \max\{ \Phi_{1/\rho}(x^0), \frac{L_y}{L} {\rm OPT}^0, \Delta_y^0\}$, and used the conditions $\alpha < 1$, $\mu_g \ll 1$;  the symbol $\lesssim$ denotes that the numerical constants are omitted (see \Cref{sec:gen_ana}).
\end{Theorem}
\vspace{-0.1cm}

The above result uses constant step sizes determined by the maximum number of iterations, $\Kmax$. Here we set the  step sizes as $\alpha_k \asymp \Kmax^{-3/5}$  and $\beta_k  \asymp \Kmax^{-2/5}$.
Similar to the previous case of strongly convex outer objective function,  $\alpha_k / \beta_k $ converges to zero as $\Kmax$ goes to infinity. Nevertheless, \Cref{th:wc:c} shows that {\sf TTSA} requires ${\cal O}(\epsilon^{-5/2} \log(1/\epsilon) )$ calls of stochastic oracle for sampled gradient/Hessian to find an $\epsilon$-nearly stationary solution. In addition, it is worth noting that when $X  =  \RR^{d_1}$, \Cref{th:wc:c} implies that {\sf TTSA}   achieves   $\mathbb{E} [ \|\nabla \ell(  x^{\sf K} )\|^2 ] = \mathcal{O}(\Kmax^{-2/5})$ \cite[Sec.~2.2]{davis2018stochastic}, i.e., $x^{\sf K}$ is an $\mathcal{O} ( \Kmax^{-2/5})$-approximate (near) stationary point of $\ell(x)$ in expectation.
	
Let us compare our sampling complexity bounds to the double loop algorithm in \cite{ghadimi2018approximation}, which requires ${\cal O}(\epsilon^{-3})$ (resp.~{${\cal O}(\epsilon^{-2})$) stochastic oracle calls for the inner problem (resp.~outer problem), to reach an $\epsilon$-stationary solution. The sample complexity of {\sf TTSA} yields a tradeoff for the inner and outer stochastic oracles. We also observe that {a trivial extension to a single-loop algorithm results in a {\it constant} error bound}\footnote{To see this, the readers are referred to \cite[Theorem 3.1]{ghadimi2018approximation}. If a single inner iteration is performed, $t_k=1$, so $\bar{A}^k\ge \|y^0-y^{\star}(x^k)\|$ which is a constant. Then the r.h.s.~of (3.70), (3.73), (3.74) in \cite[Theorem 3.1]{ghadimi2018approximation} will all have a constant term.}. Finally, we can extend \Cref{th:wc:c} to the case where $\ell(\cdot)$ is a convex function.
\begin{Corollary}\label{cor:c:c}
Under \Cref{ass:f}, \ref{ass:g}, \ref{ass:stoc} and assume that $\ell(x)$ is weakly convex with modulus $\mu_\ell \geq 0$. Consider \eqref{eq:bilevel} with $X \subseteq \RR^{d_1}$, $D_x = \sup_{ x,x' \in X} \|x - x' \| < \infty$. Let $\Kmax \geq 1$ be the maximum iteration number and set
\beq \label{eq:stepsize_cvx}
\alpha = \min\Big\{ \frac{\mu_g^2}{8 L_y L L_g^2(1+\sigma_g^2)}, \frac{1}{4 L_y L} \Kmax^{-3/4} \Big\} , \quad \beta = \min \Big\{ \frac{\mu_g}{ L_g^2 (1 + \sigma_g^2) },  \frac{2}{\mu_g} \Kmax^{-1/2} \Big\}.
\eeq
If $b_k \leq c_b \Kmax^{-1/4}$, then for large $K$, the {\sf TTSA} algorithm satisfies
\beqq 
\EE[ \ell(x^{\sf K})-\ell(x^{\star})]  = {\cal O}(\Kmax^{-1/4}),\quad \EE[\Delta_y^{\sf K+1}] = {\cal O}(\Kmax^{-1/2}), 
\eeqq
where ${\sf K}$ is an independent uniform random variable on $\{0,...,\Kmax-1\}$. By convexity, the above implies $\EE[ \ell( \frac{1}{\Kmax}\sum_{k=1}^{\Kmax} x^k ) - \ell(x^\star) ] = {\cal O}(\Kmax^{-1/4})$.
\end{Corollary}
From \Cref{cor:c:c}, the {\sf TTSA} algorithm requires ${\cal O}(\epsilon^{-4} \log(1/\epsilon))$ stochastic oracle calls to find an $\epsilon$-optimal solution (in terms of the optimality gap defined with objective values). This is comparable to the complexity bounds in \cite{ghadimi2018approximation}, which requires ${\cal O}(\epsilon^{-4} \log(1/\epsilon) )$ (resp.~${\cal O}(\epsilon^{-2} )$) stochastic oracle calls for the inner problem (resp.~outer problem). Additionally, we mention that the constant $D_x$, which represents the diameter of the constraint set, appears in the constant of the convergence bounds, therefore it is omitted in the big-O notation in \eqref{eq:stepsize_cvx}. For details, please see the proof in Appendix ~\ref{sec:pfcor}.
		
\vspace{-0.2cm}
\subsection{Convergence Analysis}\label{sec:gen_ana}
We now present the proofs for \Cref{th:sc:uc}, \ref{th:wc:c}. The proof for \Cref{cor:c:c} is similar to that of \Cref{th:wc:c}. Due to the space limitation, we refer the readers to \cite{Hong-TTSA-2020}. We highlight that the proofs of both theorems rely on the similar ideas of tackling coupled inequalities.
		
\vspace{-0.2cm}
\paragraph{Proof of \Cref{th:sc:uc}} Our proof relies on bounding the optimality gap and tracking error \emph{coupled with each other}. First we derive the convergence of the inner problem.
\begin{Lemma} \label{lemma:refined_dy}
Under \Cref{ass:f}, \ref{ass:g}, \ref{ass:stoc}. Suppose that the step size satisfies \eqref{eq:stepsize0}, \eqref{eq:stepsize1}. For any $k \geq 1$, it holds that
\beq \label{eq:upper_bound_y_seq}
\Delta_y^{k+1} \leq \prod_{\ell=0}^k (1 - \beta_\ell \mu_g / 2) \, \Delta_y^0 + \frac{8}{ \mu_g} \Big\{ {\sigma}_g^2 + \frac{4 {\rm c}_0^2 L_y^2}{\mu_g} \big[ \tilde\sigma_f^2 + 3 b_0^2 \big] \Big\} \beta_k. 
% \quad \text{where} \quad {\rm C}_y^{(1)} :=  .
\eeq
\end{Lemma}
\noindent Notice that the bound in \eqref{eq:upper_bound_y_seq} relies on the strong convexity of the inner problem and the Lipschitz properties established in \Cref{lem:lips} for $y^\star(x)$; see \S\ref{sec:pflem31}. We emphasize that the step size condition $\alpha_k \leq {\rm c}_0 \beta_k^{3/2}$ is crucial in establishing the above bound. As the second step, we bound the convergence of the outer problem.
\begin{Lemma} \label{lemma:refined_dx}
Under \Cref{ass:f}, \ref{ass:g}, \ref{ass:stoc}. Assume that the bias satisfies $b_k^2 \leq \tilde{\rm c}_b \alpha_{k+1}$. With \eqref{eq:stepsize0}, \eqref{eq:stepsize1}, for any $k \geq 1$, it holds that\vspace{-.1cm}
\begin{align}  \label{eq:upper_bound_x_seq}
\begin{split}
\Delta_x^{k+1} & \leq \prod_{\ell=0}^k (1 - \alpha_\ell \mu_\ell)  \Delta_x^0 + \Big[ \frac{4 \tilde{\rm c}_b}{ \mu_\ell^2} + \frac{2\tilde\sigma_f^2 + 6 b_0^2}{\mu_\ell} \Big] \alpha_k  \\
& \quad + \Big[ \frac{ 2L^2 }{\mu_\ell} + 3\alpha_0 L^2 \Big] \sum_{j=0}^k \alpha_j  \prod_{\ell=j+1}^k (1 - \alpha_\ell \mu_\ell) \Delta_y^{j+1}.\\[-.3cm]
\end{split}
\end{align}
% where the constants $L$ and  $L_f$, are specified in Lemma \ref{lem:lips}. 
\end{Lemma} 
see \S\ref{sec:pflem32}. We observe that \eqref{eq:upper_bound_y_seq}, \eqref{eq:upper_bound_x_seq} lead to a pair of  coupled inequalities. To compute the final bound in the theorem, we substitute \eqref{eq:upper_bound_y_seq} into \eqref{eq:upper_bound_x_seq}. As $\Delta_y^{j+1} = {\cal O}( \beta_j ) = {\cal O}( \alpha_j^{2/3} )$, %due to the difference in time scales of step sizes, 
the dominating term in \eqref{eq:upper_bound_x_seq} can be estimated as
\beq
\sum_{j=0}^k \alpha_j \prod_{\ell=j+1}^k (1 - \alpha_\ell \mu_\ell) \Delta_y^{j+1} = \sum_{j=0}^k {\cal O}(\alpha_j^{5/3}) \prod_{\ell=j+1}^k (1 - \alpha_\ell \mu_\ell ) = {\cal O}( \alpha_k^{2/3} ),
\eeq
yielding the desirable rates in the theorem. See \S\ref{sec:pfthm31} for details.
		
\vspace{-0.2cm}
\paragraph{Proof of \Cref{th:wc:c}} Without strong convexity in the outer problem, the analysis becomes more challenging. To this end, we first develop the following lemma on coupled inequalities with numerical sequences, which will be pivotal to our analysis:
		\begin{Lemma} \label{lem:recur_lem}
			Let $K \geq 1$ be an integer. 
			Consider sequences of non-negative scalars $\{ \Omega^k \}_{k=0}^K$, $\{ \Upsilon^k \}_{k=0}^K$, $\{ \Theta^k \}_{k=0}^K$. Let ${\rm c}_0, {\rm c}_1, {\rm c}_2$, ${\rm d}_0, {\rm d}_1, {\rm d}_2$ be some positive constants. If the recursion holds 
			\begin{align} \label{eq:recurnew}
			\begin{split}
				& \Omega^{k+1} \leq \Omega^k - {\rm c}_0 \Theta^{k+1} + {\rm c}_1 \Upsilon^{k+1} + {\rm c}_2,\quad  \Upsilon^{k+1} \leq (1 - {\rm d}_0) \Upsilon^k + {\rm d}_1 \Theta^k + {\rm d}_2, 
			\end{split}
			\end{align}
			for any $k \geq 0$. Then provided that 
			${\rm c}_0 - {\rm c}_1 {\rm d}_1  ({\rm d}_0)^{-1} > 0, {\rm d}_0 - {\rm d}_1 {\rm c}_1 ({\rm c}_0)^{-1} > 0$, 
			it holds
			\begin{align} \label{eq:recur_concl}
			\begin{split}
				\frac{1}{K} \sum_{k=1}^K \Theta^k & \leq \frac{ \Omega^0 + \frac{{\rm c}_1}{ {\rm d}_0 } \big( \Upsilon^0 + {\rm d}_1 \Theta^0 + {\rm d}_2 \big) }{ \big( {\rm c}_0 - {\rm c}_1 {\rm d}_1  ({\rm d}_0)^{-1} \big) K } + \frac{ {\rm c}_2 + {\rm c}_1 {\rm d}_2 ({\rm d}_0)^{-1} }{ {\rm c}_0 - {\rm c}_1 {\rm d}_1  ({\rm d}_0)^{-1} } \\
				\frac{1}{K} \sum_{k=1}^K \Upsilon^k & \leq \frac{ \Upsilon^0 + {\rm d}_1 \Theta^0 + {\rm d}_2 + \frac{ {\rm d}_1 }{ {\rm c}_0 } \Omega^0 }{ \big( {\rm d}_0 - {\rm d}_1 {\rm c}_1 ({\rm c}_0)^{-1} \big) K } + \frac{ {\rm d}_2 + {\rm d}_1 {\rm c}_2 ({\rm c}_0)^{-1} }{ {\rm d}_0 - {\rm d}_1 {\rm c}_1 ({\rm c}_0)^{-1} } .
			\end{split}
			\end{align}
		\end{Lemma}
		\noindent The proof of the above lemma is simple and is relegated to \S\ref{sec:pfrecur_n}. 
		
		We demonstrate that stationarity measures of the {\sf TTSA} iterates satisfy \eqref{eq:recurnew}. The conditions ${\rm c}_0 - {\rm c}_1 {\rm d}_1  ({\rm d}_0)^{-1} > 0, {\rm d}_0 - {\rm d}_1 {\rm c}_1 ({\rm c}_0)^{-1} > 0$ impose constraints on the step sizes and \eqref{eq:recur_concl} leads to a finite-time bound on the convergence of {\sf TTSA}. To begin our derivation of \Cref{th:wc:c}, we observe the following coupled descent lemma:
		\begin{Lemma}\label{lemma:x:descent_n} 
			Under \Cref{ass:f}, \ref{ass:g}, \ref{ass:stoc}. If $\mu_g \beta / 2 < 1$, $\beta L_g^2(1+\sigma_g^2) \leq \mu_g$, then the following inequalities hold for any $k \geq 0$:
			\begin{subequations} \label{eq:descent_n}
				\begin{align}
				\mbox{\rm OPT}^{k+1} & \le \mbox{\rm OPT}^k - \frac{1- \alpha L_f}{2\alpha}  \mathbb{E}[\|x^{k+1}-x^k\|^2] + \alpha \big[ 2   L^2 \Delta^{k+1}_y +  2b_0^2 +  {\sigma^2_f} \big] \label{eq:descent:non-convex:x_n}\\
				\Delta^{k+1}_y & \le \big (1- \mu_g  \beta / 2 \bigr )\Delta^k_y+ \Big (\frac{2}{\mu_g  \beta}-1\Big ) L_y^2 \cdot \mathbb{E}[\|x^k- x^{k-1}\|^2]+ \beta^2 \sigma_g^2 \label{eq:descent:non-convex:y_n}.
				\end{align}
			\end{subequations}
		\end{Lemma}
		\noindent The proof of \eqref{eq:descent:non-convex:x_n} is due to the smoothness of outer function $\ell(\cdot)$ established in \Cref{lem:lips}, while \eqref{eq:descent:non-convex:y_n} follows from the strong convexity of the inner problem. See the details in \S\ref{sec:pfxdescent}. Note that \eqref{eq:descent:non-convex:x_n}, \eqref{eq:descent:non-convex:y_n} together is a special case of \eqref{eq:recurnew} with:
		\begin{align} \label{eq:main_ref_ncvx}
		\begin{split}
		& \Omega^k = {\rm OPT}^k,~\Theta^k = \EE[ \| x^{k} - x^{k-1} \|^2 ],~{\rm c}_0 = \frac{1}{2 \alpha} - \frac{L_f}{2},~{\rm c}_1 = 2 \alpha L^2,~ {\rm c}_2 = \alpha ( 2 b_0^2 + \sigma_f^2 ), \\
		& \Upsilon^k = \Delta_y^k,~~{\rm d}_0 = \mu_g \beta / 2, ~~ {\rm d}_1 = \Big( \frac{2}{\mu_g \beta} - 1 \Big) L_y^2,~~ {\rm d}_2 = \beta^2 \sigma_g^2.
		\end{split}
		\end{align}
		Notice that $\Theta^0 = 0$. 
		Assuming that $\alpha \leq 1/ 2L_f$, we notice the following implications:
		\beq \label{eq:stepratio_c}
		\frac{\alpha}{\beta} \leq \frac{ \mu_g }{ 8 L_y L } ~~\Longrightarrow~~ {\rm c}_0 - {\rm c}_1 \frac{ {\rm d}_1 }{ {\rm d}_0 } \geq \frac{1}{8 \alpha} > 0,~~{\rm d}_0 - {\rm d}_1 \frac{ {\rm c}_1 }{ {\rm c}_0 } \geq \frac{ \mu_g \beta }{ 4 } > 0,
		\eeq
		i.e., if \eqref{eq:stepratio_c} holds, then the conclusion \eqref{eq:recur_concl} can be applied.
		It can be shown that the step sizes in \eqref{eq:stepsize_const} satisfy \eqref{eq:stepratio_c}. 
		Applying \Cref{lem:recur_lem} shows that
		\begin{align*} \notag
		\begin{split}
			& \frac{1}{K} \sum_{k=1}^K \EE[ \| x^k - x^{k-1} \|^2 ] \leq \frac{ 2 {\rm OPT}^0 + \frac{L}{L_y} (\Delta_y^0 + 4 \sigma_g^2 / \mu_g^2 ) }{ L_y L \cdot  K^{8/5} } + \frac{ (2 b_0^2 + \sigma_f^2) + 8 \frac{ \sigma_g^2 L^2 }{\mu_g^2} }{2 L_y^2 L^2 \cdot K^{6/5} }, \\
			& \frac{1}{K} \sum_{k=1}^K \Delta_y^k \leq \frac{2 \Delta_y^0 + \frac{8 \sigma_g^2}{\mu_g^2} + \frac{4 L_y}{ \mu_g L } {\rm OPT}^0 }{K^{3/5}} + \frac{ 8 \frac{ \sigma_g^2 }{ \mu_g^2 } + \frac{ \mu_g (2b_0^2 + \sigma_f^2) }{ 2 L^2 } }{K^{2/5}}. 
		\end{split}
		\end{align*}
		Again, we emphasize that the two timescales step size design is crucial to establishing the above upper bounds.
		Now, recalling the properties of the Moreau envelop in \eqref{eq:moreaum}, we obtain the following descent estimate:
		\begin{Lemma} \label{lem:moreau}
			Under \Cref{ass:f}, \ref{ass:g}, \ref{ass:stoc}. Set $\rho > -\mu_\ell$, $\rho \geq 0$, then for any $k \geq 0$, it holds that
			\begin{align} \label{eq:descent:second_n}
			\begin{split}
				\EE[\Phi_{1/\rho}(x^{k+1}) - \Phi_{1/\rho}(x^{k})] & \leq \frac{5\rho}{2} \EE [\|x^{k+1}-x^k\|^2] + \Big[ \frac{2\alpha \rho L^2}{\rho+\mu_\ell} + 3 \alpha^2 \rho L^2 \Big] \Delta^{k+1}_y  \\
				& \hspace{-1cm} - \frac{(\rho+ \mu_\ell)\rho\alpha}{4} \EE [\|\hx^k-x^k\|^2] + \Big[ \frac{2\rho}{\rho+\mu_\ell} +  \rho ( \tilde{\sigma}_f^2 + 3 \alpha ) \Big] \alpha^2 .
			\end{split}
			\end{align}
		\end{Lemma}
		\noindent See details in \S\ref{sec:pfmoreau}. Summing up the inequality \eqref{eq:descent:second_n} from $k=0$ to $k=\Kmax-1$ gives the following upper bound:
		\begin{align*} \notag
		\begin{split}
			\frac{1}{\Kmax}\sum_{k=0}^{\Kmax-1} \EE[ \| \hat{x}(x^k) - x^k \|^2 ]
			& \leq \frac{4}{(\rho+\mu_\ell) \rho} \Big[ \frac{\Phi_{1/\rho}(x^0)}{\alpha \Kmax} + \frac{5\rho}{2\alpha \Kmax} \sum_{k=1}^{\Kmax} \EE[ \| x^k - x^{k-1} \|^2] \Big] \\
			& + \frac{4}{ \rho+\mu_\ell } \Big[ \frac{ \frac{ L^2 }{\rho + \mu_\ell} + 3 \alpha L^2 }{ \Kmax } \sum_{k=1}^{\Kmax} \Delta_y^k + \Big[ \frac{2}{\rho+\mu_\ell} + \tilde{\sigma}_f^2 + 3 \alpha \Big] \alpha \Big].
		\end{split}
		\end{align*}
		Combining the above and $\alpha \asymp \Kmax^{-3/5}$ yields the desired 
		$\frac{1}{\Kmax}\sum_{k=0}^{\Kmax-1} \EE[ \| \hat{x}(x^k) - x^k \|^2 ] = {\cal O}(\Kmax^{-2/5})$.
		In particular, the asymptotic bound is given by setting $\rho = 2 |\mu_\ell|$.
		
		\ifonlineapp
		\paragraph{Proof of \Cref{cor:c:c}} 
		We observe that \Cref{lemma:x:descent_n} can be applied directly in this setting since convex functions are also weakly convex. With the step size choice \eqref{eq:stepsize_cvx}, similar conclusions hold as:
		\begin{align*} \notag
%		\label{eq:double:descent_cvx}
		\begin{split}
			& \frac{1}{K} \sum_{k=1}^K \EE[ \| x^k - x^{k-1} \|^2 ] \leq \frac{ 2 {\rm OPT}^0 + \frac{L}{L_y} (\Delta_y^0 + 4 \sigma_g^2 / \mu_g^2 ) }{ L_y L \cdot  K^{7/4} } + \frac{ (2 b_0^2 + \sigma_f^2) + 8 \frac{ \sigma_g^2 L^2 }{\mu_g^2} }{2 L_y^2 L^2 \cdot K^{6/4} }, \\
			& \frac{1}{K} \sum_{k=1}^K \Delta_y^k \leq \frac{2 \Delta_y^0 + \frac{8 \sigma_g^2}{\mu_g^2} + \frac{4 L_y}{ \mu_g L } {\rm OPT}^0 }{K^{1/2}} + \frac{ 8 \frac{ \sigma_g^2 }{ \mu_g^2 } + \frac{ \mu_g (2b_0^2 + \sigma_f^2) }{ 2 L^2 } }{K^{1/2}}. 
		\end{split}
		\end{align*}
		With the additional property $\mu_\ell \geq 0$, in \S\ref{sec:pfcor} we further derive an alternative descent estimate to \eqref{eq:descent:non-convex:x_n} that leads to the desired bound of $K^{-1} \sum_{k=1}^K {\rm OPT}^k$.
		\fi

		\vspace{-0.1cm}
		\section{Application to Reinforcement Learning}\label{zwsec:RL}  
		\vspace{-0.1cm}
		Consider a Markov decision process (MDP) $(S, A, \gamma, P, r)$, where $S$ and $A$ are the state and action spaces, respectively, $\gamma \in [0,1)$ is the discount factor, $P(s'|s,a)$ is the transition kernel to the next state $s'$ given the current state $s$ and action $a$, and $r(s, a) \in [0,1]$ is the reward at $(s,a)$. 
		Furthermore, the initial state $s_0$ is drawn from a fixed distribution $\rho_0$.
		We follow a stationary policy $\pi : {S} \times A \rightarrow \RR$. For any $(s,a) \in S \times A$, $\pi(a|s) $ is the probability of the agent choosing action $a \in A$ at state $s \in S$. Note that a policy $\pi \in X$ induces a Markov chain on $S$. Denote the induced Markov transition kernel as $P^\pi$ such that $s_{t+1} \sim P^\pi( \cdot | s_t )$. For any $s, s' \in S$, we have $P^{\pi} (s' | s) = \sum_{a\in A} \pi(a | s)  P(s' | s,a)$. For any $\pi \in X$, $P^{\pi}$ is assumed to induce a stationary distribution over $S$, denoted by $\mu^{\pi} $. We assume that  $|A| < \infty$ while $|S|$ is possibly infinite (but countable). 
		To simplify our notations, for any distribution $\rho$ on $S$, we let $\langle \cdot , \cdot \rangle_\rho$ be the inner product with respect to $\rho$, and $\| \cdot \|_{\mu^\pi \otimes \pi}$ be the weighted $\ell_2$-norm with respect to the probability measure $\mu^\pi \otimes \pi$ over $S \times A$ (where $f,g$ are measurable functions on $S \times A$)
		\beqq \notag
		\langle f ,g \rangle_{\rho} = \sum_{s \in S} \langle f(s,\cdot) , g(s, \cdot) \rangle  \rho(s), ~~
		\| f\|_{\mu^\pi \otimes \pi} = \sqrt{ \sum_{s \in S} \big\{ \sum _{a \in A} \pi(a | s) \cdot [ f(s,a) ]^2 \big\} \mu^\pi (s) }.
		\eeqq
		
		In policy optimization, our objective is to maximize the expected total discounted reward received by the agent with respect to the policy $\pi$, i.e., 
		\beq \label{eq:policy_opt_new}  
		\max_{ \pi \in X \subseteq \RR^{|S| \times |A|} }~-\ell ( \pi ) = \EE_\pi \big[ \sum_{t \geq 0} \gamma^t \cdot r(s_t, a_t) \!~|\!~ s_0 \sim \rho_0 \big],
		\eeq
		where $\EE_\pi$ is the expectation with the actions taken according to policy $\pi$. Here we let $\rho_0$ to denote the distribution of the initial state. 
		To see that \eqref{eq:policy_opt_new} is approximated as a bilevel problem, set $P^\pi$ as the Markov operator under the policy $\pi$. We let $Q^{\pi}$ be the unique solution to the following Bellman  equation \cite{sutton2018reinforcement}:
		\beq \label{zweq:bellman}
		Q (s,a ) =
%		 r(s,a) + \gamma \EE_{s' \sim P(\cdot | s,a), a'\sim \pi(\cdot | s')} [ Q(s', a') ] =  
		 r(s,a) + \gamma (P^{\pi } Q)(s,a),~\forall~s,a \in S \times A.
		\eeq
		Notice that the following holds:
		\beqq \notag
		Q^\pi (s,a) = \EE_\pi \big[ \sum_{t \geq 0} \gamma^t r(s_t, a_t) | s_0 = s, a_0 = a \big],~ \EE_{a \sim \pi(\cdot|s)} [ Q^\pi (s,a) ] = \langle Q^\pi(s, \cdot) , \pi( \cdot | s ) \rangle.
		\eeqq
		Further, we  parameterize $Q$ using a linear approximation $Q(s,a) \approx Q_{\theta}(s,a) \eqdef \phi^\top(s,a) \theta$, where $\phi : S \times A \rightarrow \RR^d$ is a known feature mapping and $\theta \in \RR^d$ is a finite-dimensional parameter. 
		Using the fact that $\ell(\pi) = - \EE_{\pi} [ Q^\pi (s,a) ]$, problem \eqref{eq:policy_opt_new} can be approximated as a bilevel optimization problem such that:
		\beq \label{zweq:obj:2}
		\begin{array}{rl} 
		\ds \min_{\pi \in X \subseteq \RR^{|S|\times|A|}} & \ell(\pi) = - \langle Q_{\theta^\star(\pi)}, \pi\rangle_{\rho_0} \\
		\mbox{subject to} & \ds \theta^\star(\pi) \in \argmin_{ \theta \in \RR^d }~ {\textstyle \frac{1}{2}} \|Q_\theta - r - \gamma P^\pi Q_\theta \|_{\mu^\pi \otimes \pi}^2.
		\end{array}
		\eeq 
		
		\vspace{-0.4cm}
		\paragraph{Solving Policy Optimization Problem} We illustrate how to adopt the {\sf TTSA} algorithm to solve \eqref{zweq:obj:2}. First, the inner problem is the policy evaluation (a.k.a.~`critic') which minimizes the mean squared Bellman error (MSBE). A standard approach is TD learning \cite{sutton1988learning}. We draw two consecutive state-action pairs $(s,a,s',a')$  satisfying  $s \sim \mu^{\pi^k}$, $a\sim \pi^k (\cdot |s)$,  $s' \sim P(\cdot |s,a)$, and $a' \sim \pi^k(\cdot | s')$,
		and update the critic  via
		\beq
		\label{zweq:td_restate}
		\theta^{k+1} = \theta^k - \beta h_g^k ~~~~\text{with}~~~~ h_g^k = [ \phi^\top(s, a) \theta^k - r(s, a) - \gamma \phi^\top(s', a') \theta^k ] \phi(s, a),
		\eeq
		where $\beta$ is the step size.
		This step resembles \eqref{eq:y:update} of {\sf TTSA} except that the mean field $\EE[ h_g^k | {\cal F}_k ]$ is a semigradient of the MSBE function. 
% 		We shall demonstrate that our analysis can be applied with slight modifications.
		
		Secondly, the outer problem searches for the policy (a.k.a.~`actor') that maximizes the expected discounted reward. 
		To develop this step, let us define the visitation measure and the Bregman divergence as:
		\beqq \notag
		\rho^{\pi^k} (s) := (1 - \gamma)^{-1}  \sum_{t\geq 0} \gamma^t  \PP( s_t = s),~ 
		\bar{D}_{\psi, \rho^{\pi^k} }(\pi, \pi^k) := \sum_{s \in S} D_{\psi}\big(\pi(\cdot | s), \pi^k(\cdot | s)\bigr ) \rho^{\pi^k}(s), 
		\eeqq
		such that  $\{s_t \}_{t\geq 0}$ is a trajectory of states obtained by drawing $s_0 \sim \rho_0$ and following the policy $\pi^k$, and  $D_{\psi}$ is the Kullback-Leibler (KL) divergence between  probability distributions over $A$. We also define the following gradient surrogate:
		\beq \label{eq:other_surrogate}
		[ \overline{\grd}_{\pi} f( \pi^k, \theta^{k+1})] (s,a)  = -(1-\gamma)^{-1} Q_{\theta^{k+1}}(s, a) \rho^{\pi^k}(s),~\forall~(s,a).
		\eeq		
		Similar to \eqref{eq:bar:gradient} and under the additional assumption that the linear approximation is exact, i.e., $Q_{\theta^\star(\pi^k)} = Q^{\pi^k}$, we can show $\overline{\grd}_{\pi} f( \pi^k, \theta^\star(\pi^k) ) = \grd \ell(\pi^k)$ using the policy gradient theorem \cite{sutton2018reinforcement}.
		In a similar vein as \eqref{eq:x:update} in {\sf TTSA}, we consider the mirror descent step for improving the policy (cf.~proximal policy optimization in \cite{schulman2017proximal}):
		\begin{align}
		\pi^{k+1} & = \argmin_{\pi \in X} \Bigl\{-(1-\gamma)^{-1}\langle Q_{\theta^{k+1}}, \pi-\pi^k \rangle_{\rho^{\pi^k} } + \frac{1}{\alpha} \bar{D}_{\psi, \rho^{\pi^k} }(\pi, \pi^k)\Bigr\}, \label{zweq:ppo_restate_n}
		\end{align}
		where $\alpha $ is the step size. Note that the above update can be performed as:
		\beqq \notag
		\pi^{k+1}(\cdot|s) \propto \pi^k(\cdot|s)  \exp\big[ \alpha_k (1-\gamma)^{-1}   Q_{\theta^{k+1}}(s,\cdot) \big] = \pi^0(\cdot|s) \exp \big[ (1-\gamma)^{-1} \phi(s,\cdot)^\top \sum_{i=0}^k \alpha \theta^{i+1}\big].
		\eeqq
		In other words, $\pi^{k+1}$ can be represented using    the running sum of critic $\sum_{i=0}^k \alpha \theta^{i+1}$. This is similar to the natural policy gradient method \cite{kakade2002natural}, and the algorithm requires a low memory footprint.
		Finally, the recursions \eqref{zweq:td_restate}, \eqref{zweq:ppo_restate_n} give the two-timescale natural actor critic ({\sf TT-NAC}) algorithm.
		
		 \vspace{-0.1cm}
		\subsection{Convergence Analysis of \acm}  
		Consider the following assumptions on the MDP model of interest.
		\begin{assumption}
			\label{ass:bdd_reward}
			\vspace{-0.1cm}
			The reward function is uniformly bounded by a constant $\overline r$. That is, $|r(s,a)| \leq \overline{r}$ for all $(s,a) \in S \times A$. 
% 			The action space is finite, i.e., $|A| < \infty$.  
		\end{assumption}
		
		\begin{assumption} \label{ass:linear_assumption} 
			The feature map $\phi \colon S \times A \rightarrow \RR^d$ satisfies $\| \phi(s,a) \|_2 \leq 1$ for all $(s,a) \in S \times A$. 
			The action-value function associated with each policy is a linear function of $\phi$. 
			That is, for  any policy $\pi \in X$,  there exists $\theta^{\star} (\pi)  \in \RR^d$ such that $Q^{\pi} (\cdot, \cdot)  = \phi (\cdot, \cdot) ^\top \theta^\star (\pi)  = Q_{ \theta^\star( \pi ) } (\cdot, \cdot) $.
		\end{assumption}
		
		\begin{assumption} \label{ass:stationary}
			For each policy $\pi \in X$, the induced Markov chain $P^{\pi}$ admits a  unique stationary distribution $\mu^{\pi}$ for all $\pi \in X$.
			Let there exists $\mu_\phi > 0$ such that $\EE_{s \sim \mu^{\pi}, a \sim \pi(\cdot|s) }  [\phi(s,a)  \phi (s,a)^\top ] \succeq \mu_{\phi} ^2 \cdot I_d$ for all $\pi \in X$.
		\end{assumption}
		
		\begin{assumption} \label{ass:concentrability}
			For any $(s,a) \in S \times A$ and any $\pi \in X$, let $ \varrho(s,a, \pi) $ be a probability measure over $S$,  defined by 
			\begin{align} \textstyle
			[\varrho(s,a, \pi) ] (s' ) = (1 - \gamma)^{-1} \sum_{t\geq 0} \gamma ^t \cdot \PP(s_t  = s'), \qquad \forall s' \in S.  \label{eq:define_sa_visitation}  
			\end{align} 
			That is, $\varrho(s,a, \pi) $ is the visitation measure induced by the Markov chain starting from  $(s_0, a_0 ) = (s,a)$ and follows $\pi$ afterwards. 
			For any $\pi^\star$, there exists $C_{\rho} > 0$ such that 
			$$ 
			\EE_{s'\sim \rho^{\star} } \bigg[  \biggl |  \frac{  \varrho (s,a,\pi) }{ \rho^{ \star}} (s')    \bigg| ^2  \bigg]  \leq C_{\rho} ^2, \quad \forall~(s,a) \in S \times A,  \; \pi \in X.
			$$
		{Here we let $\rho^{\star}$ denote $\rho^{\pi^\star}$ to simplify the notation, which is the visitation measure induced by $\pi^\star$ with $s_0 \sim \rho_0$.}
		\end{assumption}
{\blue We remark that  \Cref{ass:bdd_reward} is standard in the reinforcement learning literature \cite{sutton2018reinforcement,szepesvari2010algorithms}. In \Cref{ass:linear_assumption}, we assume that each $Q^{\pi}$ is linear which implies that  the linear function approximation is exact. 
A sufficient condition for  \Cref{ass:linear_assumption} is that the underlying MDP is a linear MDP \cite{yang2019sample, jin2020provably}, where both the reward function and Markov transition kernel are linear in $\phi$. Linear MDP contains the tabular MDP as a special case, where the feature mapping $\phi(s,a)$ becomes the canonical vector in $\RR^{S\times A}$. \Cref{ass:stationary} assumes that the stationary distribution $\mu^{\pi} $ exists for any policy $\pi$, which is a common property for the MDP analyzed in TD learning, e.g., \cite{dann2014policy, bhandari2018finite}. \Cref{ass:stationary} further requires the smallest eigenvalue of  $\Sigma_{\pi}$ to be bounded uniformly away from zero. Such an assumption is commonly made in the literature on policy evaluation with linear function approximation, e.g., \cite{bhandari2018finite, liu2018proximal}.
Finally, \Cref{ass:concentrability} postulates that $\rho^\star $ is regular such that the density ratio between $\varrho(s,a, \pi)$ and $\rho^\star$ has uniformly bounded second-order moments under $\rho^\star$. Such an assumption is closely related to the concentratability coefficient \cite{munos2008finite,antos2008learning,agarwal2019optimality}, 
which characterizes  the distribution shift incurred by policy updates and is conjectured essential for the sample complexity analysis of   reinforcement learning methods \cite{chen2019information}. \Cref{ass:concentrability} is satisfied if the initial distribution $\rho_0$ has lower bounded density over $S\times A$  \cite{agarwal2019optimality}. \ifonlineapp
For details, please refer to Appendix~\ref{app:just}.
\else
For details, please refer to our online appendix \cite{Hong-TTSA-2020}.
\fi}
				
		To state our main convergence results, let us define the quantities of interest:
		\begin{align}
		\Delta^{k+1}_Q : =\mathbb{E}[\| \theta^{k+1} - \theta^{\star} (\pi^k) \|_2^2], \quad \mbox{OPT}^k : = \EE[ \ell( \pi^k ) - \ell( \pi^\star ) ],  \label{eq:define_rl_tracking}
		\end{align} 
		where the expectations above are taken with respect to the i.i.d.~draws of state-action pairs in \eqref{zweq:td_restate} for \acm. 
		We remark that $\Delta_Q^k$, analogous to $\Delta_y^k$ used in ${\sf TTSA}$, is the tracking error that characterizes the performance of TD learning when the target value function, $Q^{\pi^k}$, is  time-varying due to policy updates. 
		We obtain: 
% 		the following finite-time convergence results for \acm:
		\begin{Theorem} \label{cor:rl}
			Consider the \acm~algorithm \eqref{zweq:td_restate}-\eqref{zweq:ppo_restate_n} for the policy optimization problem \eqref{zweq:obj:2}.  Let $\Kmax \geq 32^2$ be the maximum number of iterations. Under \Cref{ass:bdd_reward} -- \ref{ass:concentrability}, and we set the step sizes as
			\beq \label{eq:alpha:convex:new}
			\alpha = \frac{(1-\gamma)^3 \mu_\phi }{\sqrt{ \overline r \cdot C_{\rho}^2}} \min \Big\{ \frac{(1-\gamma)^2}{128\mu_\phi^{-2}}, \Kmax^{-3/4} \Big\},~ \beta=\min\left\{\frac{(1 - \gamma) \mu_\phi ^2}{8}
% 			, \; \frac{1}{2(1 - \gamma) \mu_\phi ^2}
			,\frac{16}{(1 - \gamma) \mu_\phi ^2} {\Kmax^{-1/2}}\right\}.
			\eeq
			Then the following holds
			\beqq
			\EE[ \mbox{\rm OPT}^{\sf K} ] 
			= {\cal O}(\Kmax^{-1/4}), \quad \EE[ \Delta_Q^{{\sf K} + 1} ] = {\cal O}(\Kmax^{-1/2}),
			\eeqq
			where ${\sf K}$ is an independent random variable uniformly distributed over    $\{0,...,\Kmax-1\}$.
		\end{Theorem}
				\vspace{-0.2cm}
		To shed lights on our analysis, we first observe the following performance difference lemma proven in \cite[Lemma 6.1]{kakade2002approximately}: 
		\beq \label{eq:key:1}
		\ell(\pi) - \ell(\pi^\star) =  (1-\gamma)^{-1}  \langle Q^\pi, \pi^\star   - \pi  \rangle_{\rho^{\pi^\star}}, \qquad \forall \pi \in X, 
		\eeq
		where $\pi^\star$ is an optimal policy solving \eqref{zweq:obj:2}.
		The above implies a restricted form of convexity, and our analysis uses the insight that \eqref{eq:key:1} plays a similar role as \eqref{eq:weakly:convex} [with $\mu_\ell \geq 0$] and characterizes the loss geometry of the outer problem.
		
Our result shows that the \acm~algorithm finds an optimal policy at the rate of ${\cal O}(K^{-1/4})$ in terms of the objective value. This rate is comparable to another variant of the \acm~algorithm in \cite{xu2020non}, which provided a customized analysis for \acm.
% 		which is analyzed under similar conditions to ours, yet their analysis lacks
% 		their samples are drawn from a single trajectory while assuming uniform ergodicity on the induced Markov chains. 
In contrast, the analysis for our \acm~algorithm is rooted in the general {\sf TTSA} framework developed in \S\ref{sec:gen_ana} for tackling bilevel optimization problems. Notice that analysis for the two-timescale actor-critic algorithm can also be found in \cite{wu2020finite}, which provides an ${\cal O}(K^{-2/5})$ convergence rate to a stationary solution. \vspace{-0.2cm}

\section{Numerical Experiments} 
We consider the data hyper-cleaning task \eqref{eq:clean}, and compare TTSA~with several algorithms such as the BSA algorithm \cite{ghadimi2018approximation}, the stocBiO \cite{Ji_ProvablyFastBilevel_Arxiv_2020} for different batch size choices, and the HOAG algorithm in \cite{Pedregosa_ICML_2016}. %All the parameter settings are the same as in Section \ref{Sec: Experiments}, except that we use a higher level 40$\%$ corruption rate. 
Note that HOAG is a deterministic algorithm and it requires full gradient computation at each iteration. In contrast, stocBiO is a stochastic algorithm but it relies on large batch gradient computations.

%Note that in \cite{Ji_ProvablyFastBilevel_Arxiv_2020}, the authors shown that stocBio exhibits better practical performance compared with other bilevel optimization algorithms.

We consider problem \eqref{eq:clean} with $L(\cdot)$ being the cross-entropy loss (i.e., a data cleaning problem for logistic regression); $ \sigma(x) := \frac{1}{1 + \exp(-x)} $; $c = 0.001$; see \cite{shaban2019truncated}. The problem is trained on the {\tt FashionMNIST} dataset \cite{FashionMNIST_Xiao_2017} with $50$k, $10$k, and $10$k image samples allocated for training, validation and testing purposes, respectively. We consider the setting where each sample in the training dataset is corrupted with  probability $0.4$. Note that the outer problem $\ell(x)$ is non-convex while the lower level problem is strongly-convex. The simulation results are an average of $3$ independent runs. 
The step sizes for different algorithms are chosen according to their theoretically suggested values. Let the outer iteration be indexed by $t$, for TTSA~we choose $\alpha_t = c_\alpha/(1 + t)^{3/5}, \; \beta_t = c_\beta/(1 + t)^{2/5},$ and tune for $c_{\alpha}$ and $c_\beta$ in the set $\{10^{-3}, 10^{-2},10^{-1},10\}$. For BSA \cite{ghadimi2018approximation}, we index the outer iteration by $t$ and the inner iteration by $\bar{k} \in \{1, \ldots, \bar{k}_t\}$. We set $\bar{k}_t = \lceil \sqrt{t + 1}\rceil$ as suggested in \cite{ghadimi2018approximation} and choose the outer and inner step-sizes  $\alpha_t$ and $\beta_{\bar{k}}$, respectively, as $\alpha_t = d_\alpha/(1 + t)^{1/2}$ and $\beta_{\bar{k}} = {d}_\beta/(\bar{k}+ 2)$. We tune for $d_{\alpha}$ and $d_\beta$ in the set $\{10^{-3}, 10^{-2},10^{-1},10\}$. Finally, for stocBiO we tune for parameters $\alpha_t$ and $\beta_t$ in the range $[0, 1]$. 

In Figure \ref{exper: data-cleaning:1},
we compare the performance of different algorithms against the total number of outer samples accessed. As observed, TTSA outperforms BSA, stocBiO and HOAG. We remark that HOAG is a deterministic algorithm and hence requires full batch gradient computations at each iteration.  Similarly, stocBio relies on large batch gradients which  results in relatively slow convergence.

\begin{figure}[t] 
\centering
\includegraphics[width=0.495\linewidth]{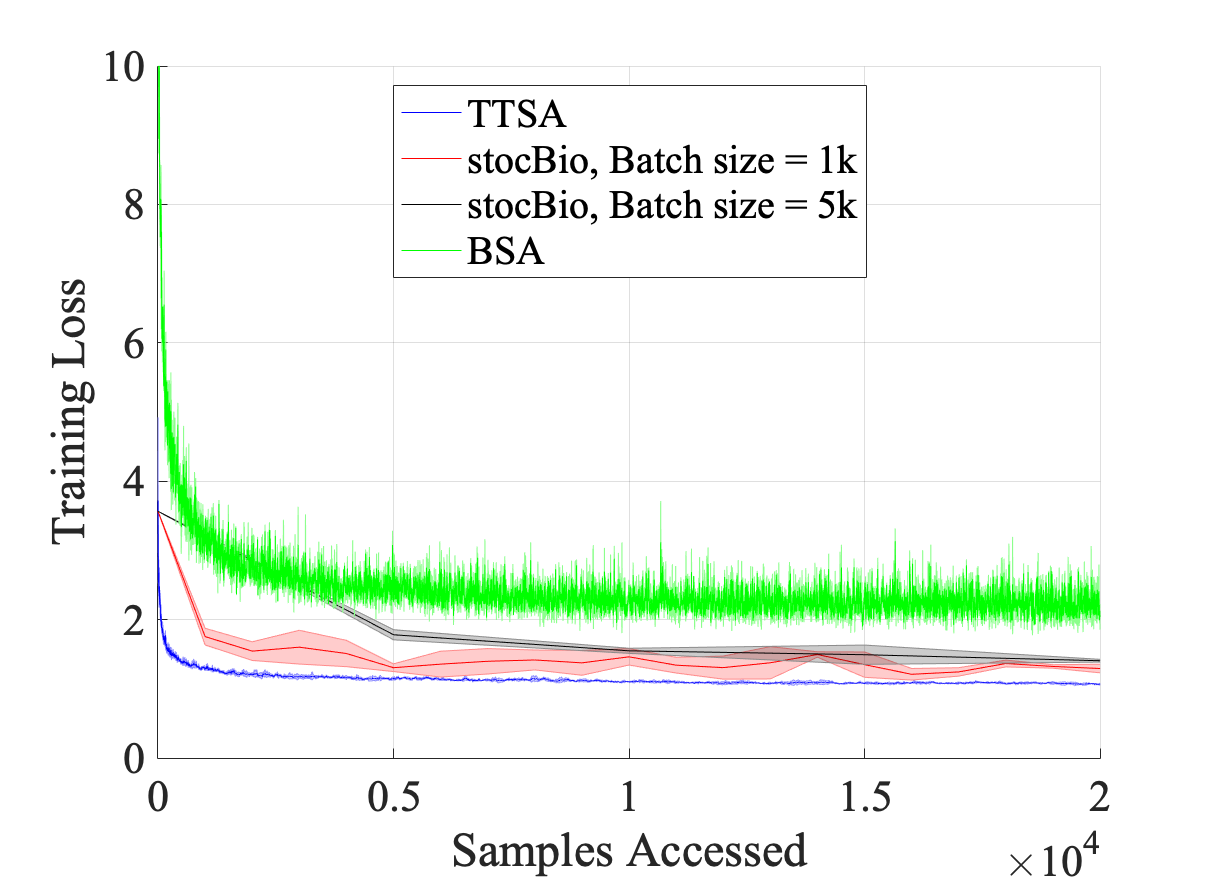}
\includegraphics[width=0.495\linewidth]{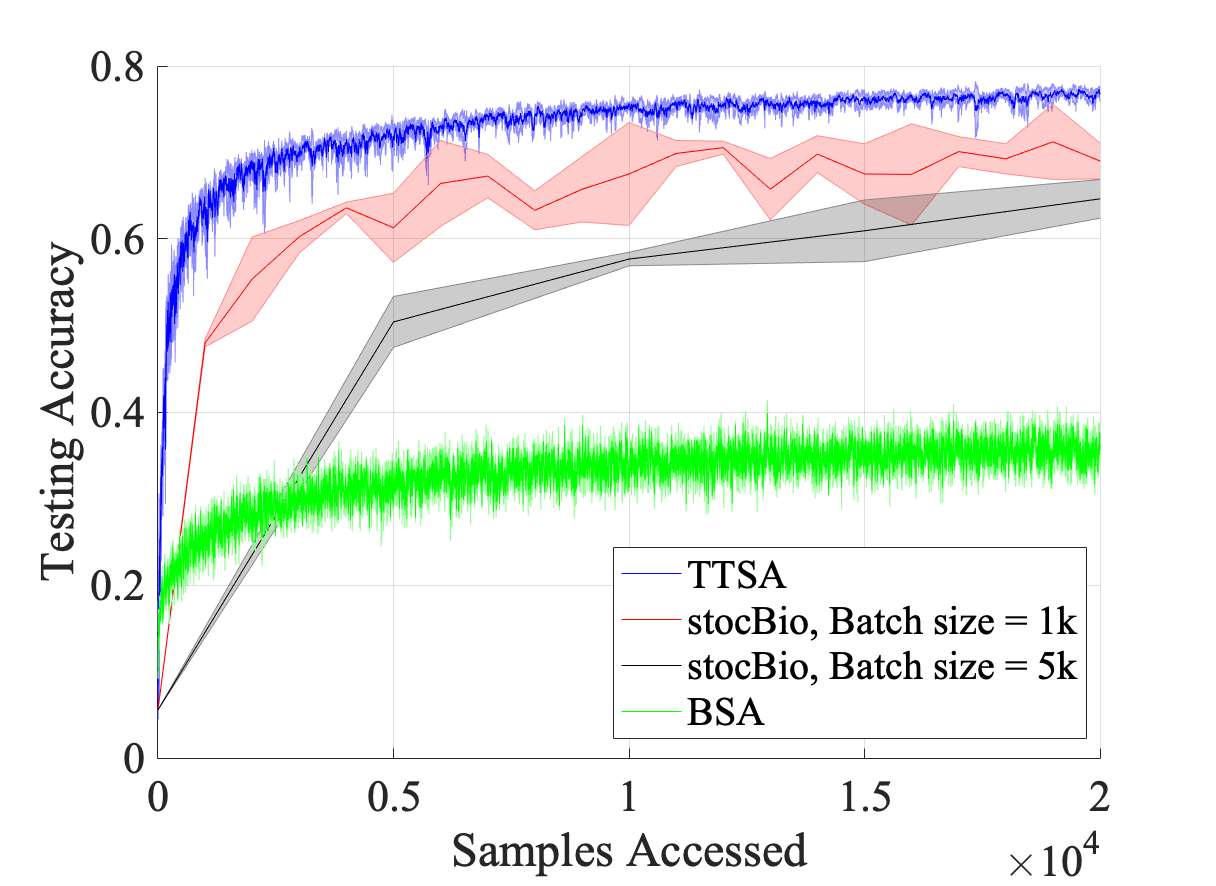}
\caption{Data hyper-cleaning task on the {\tt FashionMNIST} dataset. We plot the training loss and testing accuracy against the number of gradients evaluated with corruption rate {$p = 0.4$}.}
\label{exper: data-cleaning:1}
\end{figure}

\section{Conclusion} 
This paper develops efficient two-timescale stochastic approximation algorithms for a class of bi-level optimization problems where the inner problem is unconstrained and strongly convex. We show the convergence rates of the proposed {\sf TTSA} algorithm under the settings where the outer objective function is either strongly convex, convex, or non-convex.  Additionally, we show how our theory and analysis can be customized to a two-timescale actor-critic proximal policy optimization algorithm in reinforcement learning, and obtain a comparable convergence rate to existing literature. 
		
\appendix
		
		% !TEX root = Bilevel_Arxiv.tex

\vspace{-0.2cm}
\section{Omitted Proofs of \Cref{th:sc:uc}} \label{app:scuc}
To simplify notations, for any $n,m \in \NN$, we define the following quantities for brevity of notations. 
\beq \label{eq:define_operators}
G_{m:n}^{(1)} = \prod_{i=m}^n ( 1 - \beta_i \mu_g / 4), \quad G_{m:n}^{(2)} = \prod_{i=m}^n ( 1- \alpha_i \mu_\ell ). 
\eeq

% \paragraph{Choice of Step Size Parameters} The step sizes are chosen as
% $\alpha_k = c_\alpha / ( k + k_\alpha), \beta_k = c_\beta / ( k + k_\beta)^{2/3}$, and a possible option is to set 
% \beq \label{eq:kalpha}
% k_\alpha = \max\left\{ 35 \Big( \frac{L_g}{\mu_g} \Big)^3 (1 + \sigma_g^2)^{3/2}, \frac{ (512)^{3/2} L^2 L_y^2 }{ \mu_\ell^2} \right\} ,\quad c_\alpha = \frac{8}{3 \mu_\ell},\quad
% % \eeq
% % \beq \label{eq:kbeta}
% k_\beta = \frac{1}{4} k_\alpha, \quad c_\beta = \frac{32}{3 \mu_g}.
% \eeq
% The above choice satisfies \eqref{eq:stepsize0} with ${\rm c}_0 = \frac{\mu_g^{3/2}}{\mu_\ell}$, ${\rm c}_1 = 10 \cdot \frac{\mu_\ell^{2/3}}{\mu_g} $ as well as \eqref{eq:stepsize1}.
% As stated in the theorem, the bias in estimating the outer gradient decays as $b_k \leq   {\rm c}_b k^{-1/2}$, 
% it is easy to check that \eqref{eq:bias_decay} holds.

% \paragraph{Convergence of Tracking Error $\Delta^{k+1}_y$.} Consider

% \noindent With the step size parameters in \eqref{eq:kalpha}, it can be shown that \beqq
% \Delta_y^{k+1} = {\cal O}\Big( \Big[ \frac{\sigma_g^2}{\mu_g} + \frac{ \mu_g L_y^2}{ \mu_\ell^2} \big[ \tilde\sigma_f^2 + 3 b_0^2 \big] \Big] \cdot \frac{1}{k^{2/3}} \Big).
% \eeqq
\vspace{-0.4cm}
\subsection{Proof of \Cref{lemma:refined_dy}} \label{sec:pflem31}
% By the updating rule for $y^{k+1}$, we have
% \beq \label{eq:apply_y_update}
% \begin{split}
% \| y^{k+1} - y^\star(  x^k) \|^2 = \| y^k - y^\star(x^k) \|^2 - 2 \beta_k \pscal{ h_g^k }{ y^k - y^\star(x^k)} + \beta_k^2 \| h_g^k \|^2.
% \end{split}
% \eeq
Following a direct expansion of the updating rule for $y^{k+1}$ and taking the conditional expectation given filtration $\cF_k$ yield that 
\begin{align*}
\begin{split}
\EE[ \| y^{k+1} - y^\star(  x^k) \|^2 | {\cal F}_k ] 
%& \stackrel{(i)}= \| y^k - y^\star(x^k) \|^2 - 2 \beta_k \pscal{ \grd g(x^k,y^k) }{ y^k - y^\star(x^k) } + \beta_k^2 \EE[ \| h_g^k \|^2 | {\cal F}_k ] \\
& \leq ( 1 - 2 \beta_k \mu_g ) \| y^k - y^\star(x^k) \|^2 + \beta_k^2 \EE[ \| h_g^k \|^2 | {\cal F}_k ],
\end{split}
\end{align*}
where we used  the unbiasedness of $h_g^k$ [cf.~\Cref{ass:stoc}] and the strong convexity of $g$.
By direct computation and \eqref{eq:property:hg} in \Cref{ass:stoc}, 
we have 
\begin{align}
\EE[ \| h_g^k \|^2 | {\cal F}_k ] &  = \EE \bigl [ \| h_g^k -  \grd g( x^k, y^k ) \|^2 \big | {\cal F}_k \bigr ] + 
\| \grd g( x^k, y^k )  \|^2 \notag \\
& \leq \sigma_g^2 + (1 + {\sigma_g^2}) \| \grd g( x^k, y^k ) \|^2 \leq \sigma_g^2 + (1 + \sigma_g^2) \cdot  L_g^2 \| y^k - y^\star(x^k) \|^2,
\end{align}
where the last inequality uses \Cref{ass:g} and the optimality of the inner problem $\grd_y g( x^k, y^\star(x^k) ) = 0$. 
% Hence,  we obtain 
% \beq \notag
% \begin{split}
% \EE[ \| y^{k+1} - y^\star(  x^k) \|^2 | {\cal F}_k ] & \leq \bigl [  1 - 2 \beta_k \mu_g + \beta_k^2 L_g^2 (1 + \sigma_g^2) \bigr ] \cdot  \| y^k - y^\star(x^k) \|^2 + \beta_k^2 \sigma_g^2 . 
% \end{split}
% \eeq
% Notice that $L_g$, $\mu_g$, and $\sigma_g$ are all absolute constants. 
As $\beta_k L_g^2 (1 + \sigma_g^2) \leq \mu_g$,
we have 
\begin{align} \label{eq:yfirst}
\begin{split}
\EE[ \| y^{k+1} - y^\star(  x^k) \|^2 | {\cal F}_k ] & \leq ( 1 - \beta_k \mu_g ) \cdot \| y^k - y^\star(x^k) \|^2 + \beta_k^2 \sigma_g^2.
\end{split}
\end{align}
Using the basic inequality $2ab \leq  1/c\cdot a^2 + c \cdot b^2$ for all $c \geq 0$ and $a,b\in \RR$, 
we have 
\beq \label{eq:use_basic_ineq}
\| y^k - y^\star(x^k) \|^2 \leq \big(1 +  1/ c \big) \cdot  \| y^k - y^\star(x^{k-1}) \|^2 + \big( 1 + c \big) \| y^\star(x^{k-1}) - y^\star(x^k) \|^2.
\eeq
Note we have taken the convention that $x^{-1} = x^0$.
Furthermore, we observe that 
\beqq
\| y^\star(x^{k-1}) - y^\star(x^k) \|^2 \leq L_y^2 \| x^k - x^{k-1} \|^2 \leq \alpha_{k-1}^2 L_y^2 \cdot \| h_f^{k-1} \|^2,
\eeqq
where the first inequality follows from \Cref{lem:lips}, and the second inequality follows from the non-expansive property of projection. We have set $h_f^{-1} = 0$ as convention.

Through setting $c = \frac{2(1-\beta_k \mu_g)}{\beta_k \mu_g}$, we have
$\big(1 + 1  / c \big) ( 1 - \beta_k \mu_g) = 1 - \frac{\mu_g}{2} \beta_k$.
%$$
Substituting the above quantity $c$ into \eqref{eq:use_basic_ineq} and combining with \eqref{eq:yfirst} show that
\begin{align} 
    & \EE[ \| y^{k+1} - y^\star(  x^{k} ) \|^2 | {\cal F}_k ]  \notag \\
    & \quad  \leq  \big( 1- \frac{\beta_k \mu_g}{2} \big) \cdot \| y^k - y^{\star} (x^{k-1} ) \|^2 + \beta_k^2 \cdot \sigma_g^2  +\frac{2 - \mu_g \beta_k}{\mu_g \beta_k} \cdot \alpha_{k-1}^2 L_y^2 \cdot  \| h_f^{k-1} \|^2.\label{eq:combine_y_bound}
\end{align}
Taking  the total expectation and using \eqref{eq:hfbd}, we have 
\begin{align*}
\begin{split}
\Delta_y^{k+1} &  \leq \big( 1- \beta_k \mu_g / 2 \big) \cdot \Delta_y^k + \beta_k^2 \sigma_g^2 + \frac{2 - \mu_g \beta_k}{\mu_g \beta_k} \alpha_{k-1}^2 L_y^2 \big[ \tilde\sigma_f^2 + 3 b_{k-1}^2 + 3 L^2 \Delta_y^k \big] ,
\end{split}
\end{align*}
with the convention $\alpha_{-1} = 0$.
Using $\alpha_{k-1} \leq 2 \alpha_k$, $\alpha_k \leq {\rm c}_0 \beta_k^{3/2}$, we have 
\begin{align*}
\begin{split}
\Delta_y^{k+1} & 
% \leq \Big [  1- \beta_k \mu_g / 2 \Big ] \cdot  \Delta_y^k + \beta_k^2 \sigma_g^2  + \frac{4 L_y^2}{ \mu_g} \cdot \frac{\alpha_{k}^2}{\beta_k} \cdot D^2 
\leq \Big [  1- \beta_k \mu_g / 2 + \frac{12 {\rm c}_0^2 L_y^2 L^2}{\mu_g} \beta_k^2 \Big ] \cdot  \Delta_y^k  + \beta_k^2 \sigma_g^2 + \frac{4 {\rm c}_0^2 L_y^2}{\mu_g} \beta_k^2\cdot \big[ \tilde\sigma_f^2 + 3 b_0^2 \big] \\
& \leq \Big [  1- \beta_k \mu_g / 4 \Big ] \cdot  \Delta_y^k  + \beta_k^2 \sigma_g^2 + \frac{4 {\rm c}_0^2 L_y^2}{\mu_g} \beta_k^2\cdot \big[ \tilde\sigma_f^2 + 3 b_0^2 \big] ,
\end{split}
\end{align*}
where the last inequality is due to \eqref{eq:stepsize0}.
Solving the recursion leads to 
\begin{align} \label{eq:y_recursion2}
\begin{split}
\Delta_y^{k+1} & \textstyle \leq G_{0:k}^{(1)} \, \Delta_y^0 + \sum_{j=0}^k \beta_j^2 G_{j+1:k}^{(1)} \Big\{ {\sigma}_g^2 + \frac{4 {\rm c}_0^2 L_y^2}{\mu_g} \big[ \tilde\sigma_f^2 + 3 b_0^2 \big]  \Big\} .
% \\
% & \leq G_{0:k}^{(1)} \, \Delta_y^0 +  \beta_k \frac{8}{\mu_g} \Big\{ \frac{8 {\rm c}_0^2 L_y^2}{\mu_g} \bar{\sigma}_f^2 + \tilde{\sigma}_g^2 \Big\}
% + \frac{ 32 {\rm c}_0^2 L_y^2 L_f^2 (2+\sigma_f^2)}{\mu_g} \sum_{j=1}^k \beta_j^2 G_{j+1:k}^{(1)} \Delta_x^{j-1},
\end{split}
\end{align}
Since $\beta_{k-1} / \beta_k \leq 1 + \beta_k \cdot ( \mu_g / 8)$, 
applying Lemma~\ref{lem:aux2} to $\{ \beta_k \}_{k\geq 0}$ with $a = \mu_g / 4$ and $q = 2$,
we have $\sum_{j=0}^k \beta_j^2 G_{j+1:k}^{(1)}  \leq  \frac{8 \beta_k }{\mu_g}$.
% Finally, we define 
% \beq \notag
% {\rm C}_y^{(1)} :=  \frac{8}{ \mu_g} \Big\{ {\sigma}_g^2 + \frac{4 {\rm c}_0^2 L_y^2}{\mu_g} D^2 \Big\}.
% % , \quad {\rm C}_y^{(2)} := \frac{ 32 {\rm c}_0^2 L_y^2 L_f^2 (2 + \sigma_f^2) }{\mu_g}.
% \eeq
Finally, we can simplify \eqref{eq:y_recursion2} as
\begin{align} \label{eq:yrecur}
\begin{split}
\Delta_y^{k+1} & \leq G_{0:k}^{(1)} \, \Delta_y^0 + {\rm C}_y^{(1)} \beta_k, \quad \text{where} \quad {\rm C}_y^{(1)} :=  \frac{8}{ \mu_g} \Big\{ {\sigma}_g^2 + \frac{4 {\rm c}_0^2 L_y^2}{\mu_g} \big[ \tilde\sigma_f^2 + 3 b_0^2 \big] \Big\}.% + {\rm C}_y^{(2)} \sum_{j=1}^k \beta_j^2 G_{j+1:k}^{(1)} \Delta_x^{j-1}.
\end{split}
\end{align}

\subsection{Proof of \Cref{lemma:refined_dx}} \label{sec:pflem32} \vspace{-.2cm}
Due to the projection property, we get
\begin{align*} \notag
\begin{split}
 \| x^{k+1} - x^\star \|^2 & \leq \| x^k - \alpha_k h_f^k - x^\star \|^2 = \| x^k - x^\star \|^2 - 2 \alpha_k \pscal{ h_f^k }{ x^k - x^\star } + \alpha_k^2 \| h_f^k \|^2.  
\end{split}
\end{align*}
Taking the conditional expectation given ${\cal F}_k'$ gives  
\begin{align} \label{eq:bound_x_iter}
\begin{split}
\EE[ \| x^{k+1} - x^\star \|^2 | {\cal F}_k'] & \leq \| x^k - x^\star \|^2 - 2 \alpha_k \pscal{ \grd \ell(x^k) }{ x^k - x^\star } + \alpha_k^2 \EE[ \| h_f^k \|^2 | {\cal F}_k'] \\
& \quad -2 \alpha_k \pscal{ \Bgrd f(x^k, y^{k+1}) - \grd \ell(x^k) + B_k }{ x^k - x^\star } ,
\end{split}
\end{align}
where the inequality follows from \eqref{eq:property:hg1}. 
The strong convexity implies $\pscal{ \grd \ell(x^k) }{ x^k - x^\star } \geq \mu_\ell \| x^k - x^\star \|^2$,
%, together with the basic inequality 
%\beq \label{eq:a_basic_inequality}
%\big| 2 \alpha_k \pscal{ a }{ b } \big| \leq 
%\alpha_k^{7/6} \| a \|^2 + \alpha_k^{5/6}  \| b \|^2,
%\eeq 
we further bound the r.h.s.~of 
\eqref{eq:bound_x_iter}
via {\small
\begin{align*}
\begin{split}
 & \EE[ \| x^{k+1} - x^\star \|^2 | {\cal F}_k']  \notag \\
 &  \leq  \big(1 - 2 \alpha_k \mu_\ell \big) \| x^k - x^\star \|^2 -2 \alpha_k \pscal{ \Bgrd f(x^k, y^{k+1}) - \grd \ell(x^k) + B_k }{ x^k - x^\star } + \alpha_k^2 \EE[ \| h_f^k \|^2 | {\cal F}_k'] \\
&  \leq \big(1 - \alpha_k \mu_\ell \big) \| x^k - x^\star \|^2 + \frac{ \alpha_k }{ \mu_\ell } \| \Bgrd f(x^k, y^{k+1}) - \grd \ell(x^k) + B_k \|^2  + \alpha_k^2  \EE[ \| h_f^k \|^2 | {\cal F}_k'] \\
& \leq \big(1 - \alpha_k \mu_\ell \big)  \cdot\| x^k - x^\star \|^2 + ( 2 \alpha_k / \mu_\ell ) \cdot \big\{ L^2 \| y^{k+1} - y^\star(x^k) \|^2 + b_k^2 \bigr \}  + \alpha_k^2 \cdot \EE[ \| h_f^k \|^2 | {\cal F}_k'],
\end{split}
\end{align*}}%
where the last  inequality is from Lemma \ref{lem:lips}.
Using \eqref{eq:hfbd} and taking total expectation:  
\begin{align*} \notag
% \label{eq:bound_x_iter2}
\begin{split}
\Delta_x^{k+1} & \leq \big [  1 - \alpha_k \mu_\ell \big ]  \Delta_x^k  + \big [ 2 \alpha_k / \mu_\ell \big ]  L^2   \Delta_y^{k+1} + 2 \alpha_k b_{k} ^2 / \mu_\ell + \alpha_k^2 \big[ \tilde\sigma_f^2 + 3 b_0^2 + 3 L^2 \Delta_y^{k+1} \big] \\
& \leq \big [  1 - \alpha_k \mu_\ell \big ] \cdot  \Delta_x^k  + \big [ 2 \alpha_k / \mu_\ell + 3 \alpha_k^2 \big ] \cdot  L^2 \cdot   \Delta_y^{k+1} + \alpha_k^2 \big[ 2 \tilde{\rm c}_b / \mu_\ell + \tilde\sigma_f^2 + 3 b_0^2 \big],
\end{split}
\end{align*}%
where we have used $b_k^2 \leq \tilde{c}_b \alpha_k$.
% where the second inequality follows from the fact that $\alpha_k \leq 1$, and the last inequality uses \eqref{eq:bias_decay}. 
% Our choice of $\alpha_k$ 
% such that
% \beqq  
% 4 L_f^2 (2+\sigma_f^2) \alpha_k \leq \mu_\ell / 2 ,
% \eeqq
% implies that 
% $1 - \alpha_k \mu_\ell + 4 L_f^2 (2+\sigma_f^2) \alpha_k^2  \leq 1 - \alpha_k \cdot \mu_\ell / 2$.
% Thus, by \eqref{eq:bound_x_iter2},
% \beqq
% \begin{split}
% \Delta_x^{k+1} & \leq \big( 1 - \alpha_k \mu_\ell / 2 \big) \cdot \Delta_x^k + 2 L^2 (3+\sigma_f^2) \cdot \alpha_k  \Delta_y^{k+1} / \mu_\ell + \alpha_k^2  \cdot ( \bar{\sigma}_f^2 + 4 \tilde {\rm c}_b ).
% \end{split}
% \eeqq
Solving the recursion above leads to
\begin{align*} 
\begin{split}
\Delta_x^{k+1} & \leq G_{0:k}^{(2)} \Delta_x^0 +  \sum_{j=0}^k \Big\{ \big[ \frac{2 \tilde{\rm c}_b }{ \mu_\ell } + \tilde\sigma_f^2 + 3 b_0^2 \big] \alpha_j^2 G_{j+1:k}^{(2)} + \big[ \frac{ 2L^2 }{\mu_\ell} + 3 \alpha_0 L^2 \big] \alpha_j G_{j+1:k}^{(2)} \Delta_y^{j+1} \Big\} \\
& \leq G_{0:k}^{(2)} \Delta_x^0 + \frac{2}{\mu_\ell} \big[ \frac{2 \tilde{\rm c}_b }{ \mu_\ell } +  \tilde\sigma_f^2 + 3 b_0^2 \big] \cdot \alpha_k + \Big( \frac{ 2L^2 }{\mu_\ell} + 3 \alpha_0 L^2 \Big) \sum_{j=0}^k \alpha_j  G_{j+1:k}^{(2)} \Delta_y^{j+1} .
\end{split}
\end{align*}
The last inequality follows from applying 
\Cref{lem:aux2} with $q = 2$ and $a = \mu_{\ell}$. 
% Again, we obtain an \emph{intermediate} bound for $\Delta_x^{k}$ involving a coupling term with $\Delta_y^{k}$. 
% That said, from \eqref{eq:yrecur}, we have a closed form upper bound for the latter. This allows us to derive a bound for $\Delta_x^k$ as in the desired lemma. Therefore, 
% We conclude the proof of this lemma. 

\subsection{Bounding $\Delta_x^k$ by coupling with $\Delta_y^k$}\label{sec:pfthm31} \vspace{-.2cm}
Using \eqref{eq:yrecur}, we observe that
\begin{align} \label{eq:plug_in_bounds}
\begin{split}
\textstyle \sum_{j=0}^k \alpha_j  G_{j+1:k}^{(2)} \Delta_y^{j+1} & 
\textstyle \leq \sum_{j=0}^k \alpha_j G_{j+1:k}^{(2)} \Big\{ G_{0:j}^{(1)} \Delta_y^0 + {\rm C}_y^{(1)} \beta_j \Big\} .
\end{split}
\end{align}
We bound each term on the right-hand side of \eqref{eq:plug_in_bounds}.
For the first term, as $\mu_\ell \alpha_i \leq \mu_g \beta_i / 8$ [cf.~\eqref{eq:stepsize1}], applying \Cref{lem:aux3} with $a = \mu_\ell$, $b = \mu_g / 4$, $\gamma_i = \alpha_i$, $\rho_i = \beta_i$ gives
\beq \label{eq:term31} \textstyle
\sum_{j=0}^k \alpha_j  G_{j+1:k}^{(2)} G_{0:j}^{(1)} \Delta_y^0 \leq \frac{1}{\mu_\ell} G_{0:k}^{(2)} \Delta_y^0.
\eeq
Recall that $\beta_j \leq {\rm c}_1 \cdot \alpha_j^{2/3}$. Applying Lemma~\ref{lem:aux2} with $q = 5/3$, $a = \mu_{\ell}$ yields
\beq \label{eq:term32} \textstyle
\sum_{j=0}^k \alpha_j \beta_j G_{j+1:k}^{(2)} \leq {\rm c}_1 \sum_{j=0}^k \alpha_j^{5/3} G_{j+1:k}^{(2)} \leq {\rm c}_1 \frac{2}{\mu_\ell} \, \alpha_k^{2/3},
\eeq
We obtain a bound on the optimality gap as {\footnotesize
\begin{align*} \notag
\begin{split}
\Delta_x^{k+1} & \leq G_{0:k}^{(2)} \Big\{ \Delta_x^0 + \Big[ \frac{2 L^2}{\mu_\ell^2} + \frac{3 \alpha_0 L^2}{\mu_\ell} \Big] \Delta_y^0 \Big\} + \frac{2}{\mu_\ell} \Big[ \frac{2 \tilde{\rm c}_b} { \mu_\ell } + \tilde{\sigma}_f^2 + 3 b_0^2 \Big] \alpha_k \\
& + \frac{2 {\rm c}_1}{ \mu_\ell } \Big[ \frac{2L^2}{\mu_\ell} + 3 \alpha_0 L^2 \Big] {\rm C}_y^{(1)} \alpha_k^{\frac{2}{3}}.
\end{split}
\end{align*}}To simplify the notation, we define the constants {\small
\begin{align*}\notag
\begin{split}
& {\rm C}_x^{(0)} = \Delta_x^0 + \Big[ \frac{2 L^2}{\mu_\ell^2} + \frac{3 \alpha_0 L^2}{\mu_\ell} \Big] \Delta_y^0, {\rm C}_x^{(1)} = \frac{2}{\mu_\ell} \Big[ \frac{2 \tilde{\rm c}_b} { \mu_\ell } + \tilde{\sigma}_f^2 + 3 b_0^2 \Big] + \frac{2 {\rm c}_1}{ \mu_\ell } \Big[ \frac{2L^2}{\mu_\ell} + 3 \alpha_0 L^2 \Big] {\rm C}_y^{(1)}
\end{split}
\end{align*}}Then, as long as $\alpha_k < 1/\mu_\ell$ and we use the step size parameters in \eqref{eq:kalpha}, we have
\beqq \notag
\Delta_x^{k+1} \leq G_{0:k}^{(2)} {\rm C}_x^{(0)} + {\rm C}_x^{(1)} \alpha_k^{2/3} =  {\cal O} \Big( \Big[ \frac{L^2}{\mu_\ell^2 \mu_g^2} + \frac{ L^2 L_y^2 }{ \mu_\ell^4 } \Big] \frac{\tilde\sigma_f^2}{k^{2/3}} + \frac{L^2}{ \mu_\ell^2 \mu_g } \frac{\sigma_g^2}{k^{2/3}}  \Big),
\eeqq
and we recall that $\tilde{\sigma}_f^2 = \sigma_f^2 + 3 \sup_{x \in X} \| \grd \ell(x) \|^2$.
%where the last asymptotic relation holds . The proof is concluded.
% Thus $\Delta_x^{k+1} = {\cal O}(1/k^{2/3})$.

% Finally, to obtain the convergence  rate for the tracking error, we recall from \eqref{eq:yrecur} that
% \beq  \label{eq:yrecur0}
% \begin{split}
% \Delta_y^{k+1} & \leq G_{0:k}^{(1)} \, \Delta_y^0 + {\rm C}_y^{(1)} \beta_k + {\rm C}_y^{(2)} \sum_{j=0}^k \beta_j^2 G_{j+1:k}^{(1)} \Delta_x^{j-1}.
% \end{split}
% \eeq
% By \eqref{eq:define_operators} and direct  computation, 
% we have 
% \beq \label{eq:bound_stepsize}
%  G_{0:k}^{(1)} = \prod_{ i =0}^k ( 1 - \beta_i \mu_g / 4) \leq \prod_{ i =0}^k \exp( - \beta_i \mu_g / 4 ) = \exp \biggl ( - \mu_g / 4  \cdot \sum_{i=0}^k \beta_i \bigg )  = {\cal O} (\beta_k),
% \eeq 
% where the last equality follows from the fact that $\sum_{i=0}^k \beta_i  = \Theta(k^{1/3})$ and thus $G_{0:k}^{(1)}$ decays faster than sublinear. 
% In addition, by \eqref{eq:opt_gap_final}, 
% there exists a constant ${\rm c}'$ such that  
% \beqq
%   {\rm C}_y^{(2)} \sum_{j=0}^k \beta_j^2 G_{j+1:k}^{(1)} \Delta_x^{j-1} \leq {\rm c}' \cdot  {\rm C}_y^{(2)} \sum_{j=0}^k \beta_j^2 \alpha_j^{2/3} G_{j+1:k}^{(1)} \leq 
% {\rm c}' \alpha_0^{2/3} {\rm C}_y^{(2)} \sum_{j=0}^k \beta_j^2 G_{j+1:k}^{(1)} \leq \frac{8 {\rm c}' \alpha_0^{2/3} {\rm C}_y^{(2)}}{\mu_g} \beta_k, 
% \eeqq
% where the last inequality is due to \Cref{lem:aux2}. Combining this inequality and   \eqref{eq:yrecur0} and \eqref{eq:bound_stepsize}, we conclude that 
% \beqq 
%   \Delta_y^{k+1} = {\cal O}(\beta_k) = {\cal O}( k^{-2/3} ).
% \eeqq
% Therefore, we conclude the proof of Theorem \ref{th:sc:uc}. 

		\section{Omitted Proofs of \Cref{th:wc:c} and \Cref{cor:c:c}}\label{app:proof:weakly}
		% For the results discussed in this section, we assume fixed step sizes with $\alpha_k \equiv \alpha$ and $\beta_k \equiv \beta$.   

\subsection{Proof of \Cref{lem:recur_lem}} \label{sec:pfrecur_n}
We observe that summing the first and the second inequalities in \eqref{eq:recurnew} from $k=0$ to $k=K-1$ gives:
\begin{align} 
&\textstyle {\rm c}_0 \sum_{k=1}^K \Theta^k \leq \Omega^0 + {\rm c}_1  \sum_{k=1}^K \Upsilon^k + {\rm c}_2 \cdot K. \label{eq:recur_n_theta}\\
&\textstyle {\rm d}_0 \sum_{k=1}^K \Upsilon^k \leq \Upsilon^1 + {\rm d}_1  \sum_{k=1}^K \Theta^k + {\rm d}_2 \cdot K.   \label{eq:recur_n_upsilon}
\end{align}
Substituting \eqref{eq:recur_n_theta} into \eqref{eq:recur_n_upsilon} gives
\beqq 
\textstyle {\rm d}_0 \sum_{k=1}^K \Upsilon^k \leq \Upsilon^1 + {\rm d}_2 \cdot K + \frac{ {\rm d}_1 }{ {\rm c}_0 } \Big[ \Omega^0 + {\rm c}_1 \sum_{k=1}^K \Upsilon^k + {\rm c}_2 \cdot K \Big] . 
\eeqq
Therefore, if ${\rm d}_0 - {\rm d}_1 \frac{ {\rm c}_1 }{ {\rm c}_0 } > 0$, a simple computation yields the second inequality in \eqref{eq:recur_concl}. 
Similarly, we substitute \eqref{eq:recur_n_upsilon} into \eqref{eq:recur_n_theta} to yield
\beqq 
\textstyle {\rm c}_0 \sum_{k=1}^K \Theta^k \leq \Omega^0 + {\rm c}_2 \cdot K + \frac{ {\rm c}_1 }{ {\rm d}_0 } \Big[ \Upsilon^1 + {\rm d}_1 \sum_{k=1}^K \Theta^k + {\rm d}_2 \cdot K \Big] . 
\eeqq
Under ${\rm c}_0 - {\rm c}_1 \frac{ {\rm d}_1 }{ {\rm d}_0 } > 0$, simple computation yields the first inequality in \eqref{eq:recur_concl}.  

\subsection{Proof of Lemma \ref{lemma:x:descent_n}} \label{sec:pfxdescent}
Recall that we defined $\mbox{\rm OPT}^{k} := \EE [ \ell(x^{k}) - \ell(x^\star)] $ for each $k\geq 0$. 
To begin with, we  have the following descent estimate 
\begin{align}\label{eq:descent:l}
\ell(x^{k+1})\le \ell(x^k) + \langle \nabla \ell(x^k), x^{k+1}-x^k\rangle + (L_f/2) \|x^{k+1}-x^k\|^2.
\end{align}
The optimality condition of step \eqref{eq:x:update} leads to the following bound{\footnotesize
\begin{align}
\langle \nabla \ell(x^k), x^{k+1}-x^k\rangle
% & = \langle \nabla \ell(x^k) -\Bgrd f(x^k,y^{k+1}) - B_k, x^{k+1}-x^k\rangle + \langle B_k + \Bgrd f(x^k,y^{k+1}) - h^k_f, x^{k+1}-x^k\rangle  +  \langle h^k_f , x^{k+1}-x^k\rangle  \nonumber\\
 & \le \langle \nabla \ell(x^k) - \Bgrd f(x^k,y^{k+1}) - B_k, x^{k+1}-x^k\rangle \nonumber \\
 & \quad + \langle B_k + \Bgrd f(x^k,y^{k+1}) - h^k_f, x^{k+1}-x^k\rangle   -\frac{1}{\alpha}\| x^{k+1}-x^k\|^2  \nonumber,
\end{align}}where we obtained the inequality by adding and subtracting $B_k + \Bgrd f(x^k,y^{k+1}) - h^k_f$.
Then, taking the conditional expectation on $\mathcal{F}'_k$,  for any $c,d > 0$, we obtain 
\begin{align*}
& \mathbb{E}[\langle \nabla \ell(x^k), x^{k+1}-x^k\rangle | \mathcal{F}'_k]  \\
& \leq \EE \bigl [ \|\nabla \ell(x^k) -\Bgrd f(x^k,y^{k+1})  - B_k \| \cdot \| x^{k+1}  - x^k \| \big| \mathcal{F}'_k  \bigr ] \notag \\
&\quad + \EE \bigl [  \| B_k + \Bgrd f(x^k,y^{k+1})  - h_f^k \| \| x^{k+1} - x^k \|   \big| \mathcal{F}'_k \bigr ] - \frac{1}{\alpha}\mathbb{E}[\|x^{k+1}-x^k \|^2| \mathcal{F}'_k] \\
& \le \frac{1}{2c}\mathbb{E}[\|\nabla \ell(x^k) -\Bgrd f(x^k,y^{k+1})  - B_k \|^2| \mathcal{F}'_k] + \frac{c}{2}\mathbb{E}[\|x^{k+1}-x^k\|^2|\mathcal{F}'_k] \nonumber\\
& \quad + \frac{\sigma^2_f}{2 d} + \frac{d}{2}\mathbb{E}[\|x^{k+1}-x^k\|^2|\mathcal{F}'_k]  - \frac{1}{\alpha}\mathbb{E}[\|x^{k+1}-x^k \|^2| \mathcal{F}'_k]   ,
\end{align*}
where the second  inequality follows from the Young's inequality and \Cref{ass:stoc}.
Simplifying  the terms above leads to
\begin{align*}
& \mathbb{E}[\langle \nabla \ell(x^k), x^{k+1}-x^k\rangle | \mathcal{F}'_k]   \\
& \le  \frac{1}{2c}\mathbb{E}[\|\nabla \ell(x^k) -\Bgrd f(x^k,y^{k+1})  - B_k \|^2| \mathcal{F}'_k]   + \frac{\sigma^2_f}{2d} +\Bigl ( \frac{c+d} {2} - \frac{1}{ \alpha} \Bigr  )  \cdot   \mathbb{E}[\|x^{k+1}-x^k\|^2|\mathcal{F}'_k].
\end{align*}
Setting $d=c=\frac{1}{2\alpha}$, plugging the above to \eqref{eq:descent:l}, and taking the full expectation:
\begin{align} 
\mbox{OPT}^{k+1}& \le \mbox{OPT}^k - \left(\frac{1}{2\alpha} - \frac{L_f}{2}\right) \cdot  \mathbb{E}[\|x^{k+1}-x^k\|^2] + \alpha \Delta^{k+1} + \alpha {\sigma^2_f}, \label{eq:bound_opt1}
\end{align}
where we have denoted $\Delta^{k+1}$ as follows
\begin{align}
\Delta^{k+1} & := \EE[\|\Bgrd f (x^k;y^{k+1}) - \nabla \ell(x^k) - B_k \|^2] \stackrel{\eqref{eq:lip:f:bar}}  \leq 2 L^2 \EE[\|y^{k+1}-y^\star(x^k)\|^2] + 2 b_k^2, \notag
\end{align}
where the last inequality follows from Lemma \ref{lem:lips} and \eqref{eq:property:hg1} in \Cref{ass:stoc}.  
Next, following from the standard SGD analysis [cf.~\eqref{eq:yfirst}] and using $\beta \leq \mu_g / ( L_g^2(1+\sigma_g^2) )$, we have  
\begin{align}
&\mathbb{E}[\|y^{k+1}- y^{\star}(x^k)\|^2|{\cal F}_k] \le (1- \mu_g \beta)\mathbb{E}[\|y^{k}- y^{\star}(x^k)\|^2|{\cal F}_k] + \beta^2 \sigma_g^2\notag\\
& \le (1+c) (1- \mu_g \beta) \mathbb{E}[\|y^{k}- y^{\star}(x^{k-1})\|^2|{\cal F}_k] + (1+1/c) \mathbb{E}[\|y^{\star}(x^k)- y^{\star}(x^{k-1})\|^2|{\cal F}_k]+ \beta^2 \sigma_g^2\nonumber\\
& \le \big (1- \mu_g  \beta / 2 \bigr )\mathbb{E}[\|y^{k}- y^{\star}(x^{k-1})\|^2|{\cal F}_k] + \Big (\frac{2}{\mu_g  \beta}-1\Big ) \cdot \mathbb{E}[\|y^{\star}(x^k)- y^{\star}(x^{k-1})\|^2|{\cal F}_k]+ \beta^2 \sigma_g^2, \notag\\
& \le \big (1- \mu_g  \beta / 2 \bigr )\mathbb{E}[\|y^{k}- y^{\star}(x^{k-1})\|^2|{\cal F}_k] + \Big (\frac{2}{\mu_g  \beta}-1\Big ) {L_y^2} \cdot \mathbb{E}[\|x^k- x^{k-1}\|^2|{\cal F}_k]+ \beta^2 \sigma_g^2, \notag
\end{align}
where the last inequality is due to the Lipschitz continuity property \eqref{eq:lip:f:bar} and $\mu_g \beta<1$. Furthermore, we have picked $c = \mu_g \beta \cdot [ 2(1-\mu_g \beta) ]^{-1} $, so that 
\beqq \notag
(1+c)(1-\mu_g  \beta) = 1-\mu_g  \beta /2 , \qquad  1 / c  + 1= 2 / (\mu_g \beta) - 1.
\eeqq
Taking a full expectation on both sides leads to the desired result.

\subsection{Proof of \Cref{lem:moreau}} \label{sec:pfmoreau}
For simplicity, we let  $\hx^{k+1}$ and $\hx$  denote $\hx(x^{k+1})$ and   $\hx(x)$, respectively. 
For any $x \in X$, letting $x_1 = \hat x$ and $x_2 = x$ in 
 \eqref{eq:weakly:convex}, we get
\beqq
\ell(\hx) \ge \ell(x) + \langle \nabla \ell(x), \hx-x \rangle +\frac{\mu_\ell}{2}\|\hx-x\|^2.  
\eeqq
Moreover, by the definition of $\hat x$, for any $x \in X$, we have 
\begin{align}
& \ell(x) + \frac{\rho}{2}\|x-x\|^2- \Big  [ \ell(\hx) +\frac{\rho}{2}\|\hx-x\|^2\Big ]  =\ell(x) -\Big  [ \ell(\hx) +\frac{\rho}{2}\|\hx-x\|^2\Big ] \ge 0 .
& \end{align}
Adding the two inequalities  above, we obtain
\begin{align}\label{eq:descent:M}
-\frac{\mu_\ell+\rho}{2} \cdot \|\hx-x\|^2 \ge  \langle \nabla \ell(x), \hx-x \rangle.
\end{align}
Note that we choose  $\rho$ such  that $
\rho + \mu_\ell > 0.$ 
To proceed, 
combining the definitions of the Moreau envelop and $\hat x $ in \eqref{eq:moreaum}, for $x ^{k+1}$, we have  %{\red[double checked the proof, actually needs one additional term.]}
\begin{align}
\Phi_{1/\rho}(x^{k+1}) & \stackrel{\eqref{eq:moreaum}}= \ell(\hx^{k+1}) + \frac{\rho}{2} \cdot \|x^{k+1}-\hx^{k+1}\|^2 \le \ell(\hx^k)+\frac{\rho}{2}  \cdot  \|x^{k+1}-\hx^k\|^2\nonumber\\
& \le \ell(\hx^k)+\frac{\rho}{2}\cdot \|x^{k}-\hx^k\|^2+ \frac{\rho}{2}  \cdot  \|x^{k+1}-x^k\|^2+ \rho\alpha  \cdot   \langle \hx^k-x^k, h^k_f\rangle \nonumber\\
& \quad\quad  + \alpha \rho \cdot \langle h^k_f, x^k - x^{k+1}\rangle + \rho \|x^{k+1}-x^k\|^2  \nonumber\\
& \stackrel{\eqref{eq:moreaum}} = \Phi_{1/\rho}(x^{k}) + {\frac{5 \rho}{2}}\cdot \|x^{k+1}-x^k\|^2+ \rho\alpha\langle \hx^k-x^k, h^k_f\rangle  +  {\alpha^2 \rho \|h^k_f\|^2} \label{eq:Phi:descent},
\end{align}
where the first
equality and the first inequality follow from the optimality of $\hat x^{k+1} = \hat x (x^{k+1})$,
and the second term is from the optimality condition in \eqref{eq:x:update}. 
For any $x^{\star}$ that is a global optimal solution for the original problem $\min_{x\in  X}\ell(x)$, we must have 
\beq \notag
\Phi_{1/\rho}(x^\star) = \min_{x \in X} \Big\{  \ell(x) + \frac{\rho}{2}\|x-x^{\star}\|^2 \Big\}  = \ell(x^{\star}),
\eeq
where the last equality holds because  
\begin{align}
	\Phi_{1/\rho}(x^\star)  = \min_{x \in X} \bigl \{  \ell(x) + \frac{\rho}{2}\|x-x^{\star}\|^2\bigr \}  &\leq  \ell(x^\star ) + \frac{\rho}{2}\|x ^\star -x^{\star}\|^2 = \ell(x^\star ),    \\ 
\Phi_{1/\rho}(z)  = \min_{x\in X} \big \{  \ell(x) + \frac{\rho}{2}\|x-z\|^2 \bigr \} &  \ge \min_{x\in X} \ell(x)  =\ell(x^\star),\; \forall~z\in X.
\end{align}
Taking expectation of $\langle \hx^k-x^k, h^k_f\rangle$ while conditioning on $\mathcal{F}_k'$, we have:
\begin{align} 
&\mathbb{E}[\langle \hx^k-x^k, h^k_f\rangle|\mathcal{F}_k'] \nonumber\\
& = \mathbb{E}[\langle \hx^k - x^k, h^k_f-\Bgrd f(x^k,y^{k+1}) + \Bgrd f(x^k,y^{k+1}) - {\nabla} \ell(x^k) +  {\nabla} \ell(x^k)\rangle|\mathcal{F}_k'] \nonumber\\
& = \langle \hx^k-x^k, B_k\rangle + \mathbb{E}[\langle \hx^k-x^k, \Bgrd f(x^k,y^{k+1}) - {\nabla} \ell(x^k)\rangle + \langle \hx^k-x^k,  {\nabla} \ell(x^k)\rangle|\mathcal{F}_k'], \label {eq:some_bound11}
\end{align}
where the second equality follows from \eqref{eq:property:hg1} in \Cref{ass:stoc}.
By Young's inequality, for any $c > 0$, we have\vspace{-.1cm}
\beqq \label{eq:some_bound111}
\langle \hx^k-x^k, B_k\rangle   \leq  \frac{c}{4}\|\hx^k-x^k\|^2 + \frac{1}{c} b_k^2 ,  \vspace{-0.4cm}
\eeqq
\beqq \notag
\mathbb{E}[\langle \hx^k-x^k, \Bgrd f(x^k,y^{k+1}) - {\nabla} \ell(x^k)\rangle|\mathcal{F}_k'] \leq \frac{1}{c} \|\Bgrd f(x^k,y^{k+1}) - {\nabla} \ell(x^k)\|^2 + \frac{c}{4}\|\hx^k-x^k\|^2 , 
\eeqq
where we also use \eqref{eq:property:hg1} in deriving \eqref{eq:some_bound111}. 
Combining \eqref{eq:descent:M}, \eqref{eq:some_bound11}, \eqref{eq:some_bound111}, and setting $c = ( \rho + \mu_{\ell}) / 2$,  we obtain that 
\begin{align}
&\mathbb{E}[\langle \hx^k-x^k, h^k_f\rangle|\mathcal{F}_k'] \\
& \le \frac{c}{2}\|\hx^k-x^k\|^2 + \frac{1}{c} b_k^2 +\frac{1}{c}\mathbb{E}[\|\Bgrd f(x^k,y^{k+1}) - {\nabla} \ell(x^k)\|^2] - \frac{\rho+\mu_\ell}{2}\|\hx^k-x^k\|^2\nonumber\\
&  = \frac{2}{\rho+\mu_\ell} \cdot \mathbb{E}[\|\Bgrd f(x^k,y^{k+1}) - {\nabla} \ell(x^k)\|^2]  - \frac{(\rho+\mu_\ell)}{4} \cdot \|\hx^k-x^k\|^2 +  \frac{2}{\rho +\mu_\ell} \cdot  b_k^2   \nonumber\\
&   \le  \frac{2L^2}{\rho+\mu_\ell} \cdot \mathbb{E}[\|y^{k+1}- y^\star(x^k)\|^2]  - \frac{(\rho+\mu_\ell)}{4} \cdot \|\hx^k-x^k\|^2 +  \frac{2}{\rho +\mu_\ell} \cdot b_k^2, \nonumber
\end{align}
where the last step  follows from the first inequality of  Lemma \ref{lem:lips}.
Plugging the above into \eqref{eq:Phi:descent}, and taking a full expectation, we obtain
\begin{align}
& \EE [\Phi_{1/\rho}(x^{k+1})] - \EE[\Phi_{1/\rho}(x^{k})]\nonumber\\
&  \le \frac{5\rho}{2} \EE [\|x^{k+1}-x^k\|^2]+ \frac{2\rho\alpha L^2}{\rho+\mu_\ell} \Delta^{k+1}_y - \frac{(\rho+ \mu_\ell)\rho\alpha}{4} \EE [\|\hx^k-x^k\|^2] +\frac{2\rho \alpha b_k^2}{\rho+\mu_\ell} + \alpha^2 \rho \EE[ \|h^k_f\|^2] \nonumber \\
&  \le \frac{5\rho}{2} \EE [\|x^{k+1}-x^k\|^2]+ \frac{2\rho\alpha L^2}{\rho+\mu_\ell} \Delta^{k+1}_y - \frac{(\rho+ \mu_\ell)\rho\alpha}{4} \EE [\|\hx^k-x^k\|^2] + \frac{2\rho \alpha b_k^2}{\rho+\mu_\ell}  \nonumber\\
&  \quad + \alpha^2 \rho ( \tilde{\sigma}_f^2 + 3 b_k^2 + 3 L^2 \Delta_y^{k+1} ) \nonumber\\
&  \leq \frac{5\rho}{2} \EE [\|x^{k+1}-x^k\|^2] + \Big[ \frac{2\alpha \rho L^2}{\rho+\mu_\ell} + 3 \alpha^2 \rho L^2 \Big] \Delta^{k+1}_y - \frac{(\rho+ \mu_\ell)\rho\alpha}{4} \EE [\|\hx^k-x^k\|^2] \nonumber\\
& \quad + \Big[ \frac{2\rho}{\rho+\mu_\ell} + \rho ( \tilde{\sigma}_f^2 + 3 b_0^2 ) \Big] \alpha^2 , \nonumber
\end{align}
where  %we recall that  the tracking error $\Delta_y^{k+1} $  is defined as $\mathbb{E}[\|y^{k+1}- y^\star(x^k)\|^2]$ and 
the last inequality is due to the assumption $b_k^2 \leq \alpha$. 

{\blue \subsection{Proof of \Cref{cor:c:c}} \label{sec:pfcor}
Our proof departs from that of \Cref{th:wc:c} through manipulating the descent estimate \eqref{eq:descent:l} in an alternative way. The key is to observe the following three-point inequality \cite{beck2017first}:
\beq \label{eq:threepoint}
\langle h_f^k , x^{k+1} - x^\star \rangle \leq \frac{1}{2 \alpha} \Big\{ \| x^\star - x^k \|^2 - \| x^\star - x^{k+1} \|^2 - \| x^k - x^{k+1} \|^2 \Big\},
\eeq 
where $x^\star$ is an optimal solution to \eqref{eq:bilevel}.
Observe that
\beqq
\langle \grd \ell(x^k) , x^{k+1} - x^k \rangle = \langle \grd \ell(x^k) - h_f^k , x^{k+1} -x^\star \rangle + \langle h_f^k , x^{k+1} - x^\star \rangle + \langle \grd \ell(x^k) , x^\star - x^k \rangle.
\eeqq
Notice that due to the convexity of $\ell(x)$, we have $\langle \grd \ell(x^k) , x^\star - x^k \rangle \leq -{\rm OPT}^k$. Furthermore,
\begin{align*}
\begin{split}
& \langle \grd \ell(x^k) - h_f^k , x^{k+1} -x^\star \rangle\\
& = \langle \grd \ell(x^k) - h_f^k + B_k + \Bgrd f(x^k, y^{k+1} ) - B_k - \Bgrd f(x^k, y^{k+1} )  , x^{k+1} -x^\star \rangle \\
& \leq D_x \big\{ b_k + L \| y^{k+1} - y^\star(x^k) \| \big\} + \langle B_k + \Bgrd f(x^k, y^{k+1} ) - h_f^k , x^{k+1} - x^k + x^k - x^\star \rangle. 
\end{split}
\end{align*}
We notice that $\EE[ \langle B_k + \Bgrd f(x^k, y^{k+1} ) - h_f^k , x^k - x^\star \rangle | {\cal F}_k'] = 0$. Thus, taking the total expectation on both sides and applying Young's inequality on the last inner product lead to 
\begin{align*}
\begin{split}
& \EE[ \langle \grd \ell(x^k) - h_f^k , x^{k+1} -x^\star ] \leq D_x \big\{ b_k + L \EE[ \| y^{k+1} - y^\star(x^k) \| ] \big\} + \frac{\alpha}{2} \sigma_f^2 + \frac{1}{2\alpha} \EE[ \| x^{k+1} - x^k \|^2 ] 
\end{split}
\end{align*}
Substituting the above observations into \eqref{eq:descent:l} and using the three-point inequality \eqref{eq:threepoint} give
\begin{align*}
\begin{split}
\EE[ \ell(x^{k+1}) - \ell(x^k) ] & \leq D_x \big\{ b_k + L \EE[ \| y^{k+1} - y^\star(x^k) \| ] \big\} + \frac{1}{2\alpha} \Big\{ \| x^\star - x^k \|^2 - \| x^\star - x^{k+1} \|^2 \Big\} \\
& \quad + \frac{\alpha}{2} \sigma_f^2 - {\rm OPT}^k + \frac{L_f}{2} \EE[ \| x^{k+1} - x^k \|^2 ]. 
\end{split}
\end{align*}
Summing up both sides from $k=0$ to $k=\Kmax-1$ and dividing by $\Kmax$ gives
\begin{align*}
\begin{split}
\frac{1}{\Kmax}\sum_{k=1}^{\Kmax} {\rm OPT}^k & \leq D_x b_0 + \frac{\alpha \sigma_f^2}{2} + \frac{\|x^\star - x^0 \|^2 }{2\alpha \Kmax} + \frac{D_x L}{K}\sum_{k=1}^{\Kmax}\EE[\| y^{k} - y^\star(x^{k-1}) \|] \\
& \quad + \frac{L_f}{2\Kmax} \sum_{k=1}^{\Kmax} \EE[ \| x^k - x^{k-1} \|^2 ].
\end{split}
\end{align*}
Applying Cauchy-Schwartz inequality and \Cref{lem:recur_lem}, \ref{lemma:x:descent_n} with $\alpha = {\cal O}( \Kmax^{-3/4} )$, $\beta = {\cal O}(\Kmax^{-1/2})$ as in \eqref{eq:stepsize_cvx} show that $\frac{1}{\Kmax}\sum_{k=1}^{\Kmax}\EE[\| y^{k} - y^\star(x^{k-1}) \|] \leq \sqrt{\frac{1}{\Kmax}\sum_{k=1}^{\Kmax} \Delta_y^k} = {\cal O}(\Kmax^{-1/4})$; cf.~\eqref{eq:main_ref_ncvx}. The proof is concluded.}

		\vspace{-0.2cm}
\ifonlineapp
\section{Justifications to \Cref{ass:bdd_reward}--\Cref{ass:concentrability}} \label{app:just}
In the following, we list these assumptions and provide explanations for when the assumptions are satisfied.
	 
\begin{itemize} 
\item (\Cref{ass:bdd_reward})   The reward function is uniformly bounded by a constant $\overline r$. That is, $|r(s,a)| \leq \overline{r}$ for all $(s,a) \in S \times A$. 
% 			The action space is finite, i.e., $|A| < \infty$.  
 
 This assumption merely states that the reward functions are uniformly bounded. This is a standard assumption used in MDP and reinforcement learning community. See, e.g., Chapter 2.2 of \cite{szepesvari2010algorithms}. 
 In practice, the reward functions are usually hand-crafted by the problem solver. They often encode the scores earned by the agent in each step,  or whether some desired goal is reached.   
 
\item (\Cref{ass:linear_assumption})  The feature map $\phi \colon S \times A \rightarrow \RR^d$ satisfies $\| \phi(s,a) \|_2 \leq 1$ for all $(s,a) \in S \times A$. 
			The action-value function associated with each policy is a linear function of $\phi$. 
That is, for  any policy $\pi \in X$,  there exists $\theta^{\star} (\pi)  \in \RR^d$ such that $Q^{\pi} (\cdot, \cdot)  = \phi (\cdot, \cdot) ^\top \theta^\star (\pi)  = Q_{ \theta^\star( \pi ) } (\cdot, \cdot) $.
		
This assumption assumes that 	the action-value function $	Q^{\pi}$ is a linear function in a known feature mapping $\phi$ and $\phi$ is bounded. Such an assumption is standard in the literature on  reinforcement learning with linear function approximation. See, e.g., Chapter 3.2 of \cite{szepesvari2010algorithms}. In this line of research, it is oftentime postulated that $V^{\pi}(\cdot )$ or $Q^{\pi} (\cdot, \cdot) $ are linear functions of a known feature mapping.

 As for the feature mapping $\phi$, it is usually constructed based on domain knowledge. Some of the common examples include polynomial functions, on $[0,1]$ radial basis function, and random features, which are all bounded. Here we assume that $\sup_{(s,a) \in S\times A} \|\phi(s, a )\|_2  $ is bounded by  one for simplicity, which can be replaced by any fixed parameter. 

Moreover, a concrete mathematical model that satisfies such a model is known as the linear MDP (see \cite{jin2020provably}), which assumes that both the reward function and the Markov transition kernel are linear in the given feature mapping $\phi$. 
Specifically, it is assumed that there exist $\lambda \in \RR^{d}$ and $\mu \colon S \rightarrow \RR^d$ such that 
\begin{align}\label{eq:linear_mdp}
r(s,a) = \phi(s,a)^\top \lambda, \qquad P(s' \,|\, s,a) = \phi(s,a)^\top \mu(s') \qquad  \forall (s, a, s') \in S \times A\times S. 
\end{align}
Such a model includes the finite tabular MDP as a special case with $\phi (s, a) $ being the canonical vector ${\bm e}_{s,a} $  in $\RR^{S\times A} $.
Under  the linear MDP assumption, for any policy $\pi$, the value functions $Q^{\pi}$ and $V^{\pi}$  exist and satisfy 
\begin{align*}
V^{\pi}(s)  & = \sum_{a \in A} \pi(a \, | \, s) Q^{\pi}(s,a) \\
 Q^{\pi}(s,a) & = r(s,a) + \gamma \cdot \sum_{s' \in S } P(s' \,|\, s,a) \cdot V^{\pi}(s') = \phi(s,a)^\top  \biggl (   \underbrace{\lambda + \sum_{\in S } \mu(s') \cdot V^{\pi}(s') } _{\theta^\star (\pi)  } \bigg) . 
\end{align*}
Thus \eqref{eq:linear_mdp} serves as a sufficient condition for the assumption.

\item (\Cref{ass:stationary}) For each policy $\pi \in X$, the induced Markov chain $P^{\pi}$ admits a  unique stationary distribution $\mu^{\pi}$ for all $\pi \in X$. Moreover, there exists $\mu_\phi > 0$ such that 
\[ 
\EE_{s \sim \mu^{\pi}, a \sim \pi(\cdot|s) }  [\phi(s,a)  \phi^\top(s,a)] \succeq \mu_{\phi} ^2 \cdot I_d,~\forall~\pi \in X.
\]
 
The assumption that the Markov chain $P^{\pi}$ induced by any policy $\pi$ has a unique stationary distribution $\mu^{\pi}$ is a common assumption made in the literature on policy gradient. A sufficient condition ensures this  is  that all deterministic (stationary) policies visit all states eventually with probability one, i.e., the MDP is \emph{unichain} (see Section 4.2.4 of \cite{szepesvari2010algorithms}).  

Furthermore, for asymptotic convergence analysis, classical RL literature often assumes that 
$\EE_{s \sim \mu^{\pi}, a \sim \pi(\cdot|s) }  [\phi(s,a)  \phi(s,a) ^\top ]$ is invertible (see Section 4.4.2 of \cite{szepesvari2010algorithms}; page 70). Here we additionally assumes that such a matrix is well-conditioned in the sense the smallest eigenvalue is lower bounded for nonasymptotic analysis. Such an assumption is also required for establishing statistical rates in linear regression. 

In the tabular setting, a sufficient condition that justifying such an assumption is that the transition model is sufficient stochastic such that every policy induces $\pi$ induces a stationary distribution $\mu^{\pi} $ over $S$ such that the mass of $\mu^{\pi} $ on each state $s$ is lower bounded by $\mu_{\phi} > 0$. 
		 
\item  (\Cref{ass:concentrability}) For any $(s,a) \in S \times A$ and any $\pi \in X$, let $ \varrho(s,a, \pi) $ be a probability measure over $S$,  defined by 
\begin{align} \textstyle
[\varrho(s,a, \pi) ] (s' ) = (1 - \gamma)^{-1} \sum_{t\geq 0} \gamma ^t \cdot \PP(s_t  = s'), \qquad \forall s' \in S.  \label{eq:define_sa_visitation}  
\end{align} 
That is, $\varrho(s,a, \pi) $ is the visitation measure induced by the Markov chain starting from  $(s_0, a_0 ) = (s,a)$ and follows $\pi$ afterwards. For any $\pi^\star$, there exists $C_{\rho} > 0$ such that 
\[
\EE_{s'\sim \rho^{\pi^\star} } \bigg[  \biggl |  \frac{  \varrho (s,a,\pi) }{ \rho^{\star}} (s')    \bigg| ^2  \bigg]  \leq C_{\rho} ^2, \quad \forall~(s,a) \in S \times A,  \; \pi \in X.
\]
This assumption postulates that the distribution shift between the visitation measure induced by any policy $\pi$ and that induced by the optimal policy $\pi^*$ is bounded. Here the distribution shift is defined by the second-order moment of the density ratio. 
	 
Such  an assumption is commonly made in reinforcement learning literature with various forms, which are referred to \emph{concentrability coefficients} in general. 
It is conjectured in \cite{chen2019information} that such an assumption is necessary for theoretical analysis. Moreover, our version is slightly weaker than that in \cite{chen2019information}, which essentially assumes the $\ell_{\infty}$-norm of the  density ratio between $\varrho (s,a,\pi)$ and $\rho^\star$ is upper bounded. 
Moreover, a sufficient condition of \Cref{ass:concentrability} is that the initial distribution $\rho_0$ is lower bounded everywhere over $S\times A$. Such a condition also appears in existing work, e.g., \cite{agarwal2019optimality}.  Note that $\rho^* \geq (1-\gamma)\cdot \rho_0$.  Thus when the probability mass function of $\rho_0$ is lower bounded by $c_0$, \Cref{ass:concentrability} is satisfied with $C_{\rho} = (1-\gamma)^{-1} \cdot c_0^{-1}$. 
\end{itemize} 	
\fi 

\section{Proof of \Cref{cor:rl}}\label{app:rldet}
Hereafter, we let $\langle \cdot, \cdot \rangle$ and $\| \cdot \|$ denote the inner product and $\ell_1$-norm on  $\RR^{|A|}$, respectively. 
For any two policies $\pi_1$ and $\pi_2$, for any $s \in S$, $\| \pi_1(\cdot |s) - \pi_2(\cdot | s) \|_1$ is the total variation  distance between $\pi_1 (\cdot |s)$ and $\pi_2 (\cdot |s)$. 
% In addition, $\| \cdot \|_\infty$ is the $\ell_{\infty}$-norm 
% on $\RR^{|A|}$ which is dual to $\| \cdot \|_1$.
% Based on $\langle \cdot, \cdot \rangle$, $\| \cdot \|_1$, and $\| \cdot \|_\infty$ on $\RR^{|A|}$, 
% any probability distribution $\rho$ on $S$ induces an inner product, primal, and dual norms on the set of functions on $S\times A$. 
For any $f, f' \colon S \times A \rightarrow \RR $, define the following norms: 
\begin{align*}
\textstyle
% \langle f, f'  \rangle_\rho = \sum_{s \in S} \langle f(s, \cdot), f' ( s, \cdot ) \rangle \rho(s), \quad 
\| f  \|_{\rho,1} = \big[  \sum_{s \in S} \|f(s, \cdot ) \|_1^2 \rho(s) \big]  ^{1/2}, \quad 
\| f  \|_{\rho,\infty} = \big[  \sum_{s \in S} \|f(s, \cdot ) \|_\infty^2 \rho(s) \big]  ^{1/2}
\end{align*}
% Moreover, we define the norm
% % the dual norm of $\| \cdot \|_{\rho,1}$ as 
% $\| f \|_{\rho , \infty} = \sqrt{ \sum_{s \in S} \|f(s, \cdot ) \|_\infty^2 \rho(s) }$.
The following result can be derived from the H\"older's inequality: 
\begin{align} \textstyle
 \big | \langle f ,f'  \rangle_{\rho}  \bigr |  \leq  \sum_{s \in S} \bigl |  
 \langle  f (s, \cdot ),  f'(s, \cdot ) \rangle 
 \bigr | \rho(s)  
%  \leq \sum_{s \in S}  \| f (s, \cdot ) \|_1 \|  f'(s, \cdot ) \|_\infty \rho(s) 
 \leq \| f\|_{\rho,1} \| f' \|_{\rho, \infty} \label{eq:rho_holder}.
\end{align} 
%where the last inequality follows from the Cauchy-Schwarz inequality. 
Lastly, it can be shown that $\| \pi \|_{\rho,1} = 1, \| \pi \|_{\rho,\infty} \leq 1$.
% In addition, we have
% \[ 
% \| \pi \|_{\rho} = \biggl [  \sum_{s \in S} \| \pi (\cdot |s)  \|^2 d\rho(s)\biggr ]^{1/2} = 1, \qquad \| \pi \|_{\rho, *} = \biggl [  \sum_{s \in S} \| \pi (\cdot |s)  \|_* ^2 d\rho(s)\biggr ]^{1/2} \leq 1. 
% \] 

% Recall that $X $ denotes the set of feasible policies. 
% Also recall that 
% any probability distribution $\rho$ over $S$ induces an inner product between a value function $Q$ and a policy $\pi \in X$, and a distance between   any two policies $\pi_1, \pi_2 \in X$, which are  given by 
% \begin{align}
% \langle Q, \pi \rangle_\rho =  \int_S \langle Q(s, \cdot), \pi(\cdot | s) \rangle d\rho(s), \quad  \| \pi_1 - \pi_2  \|_{\rho } =  \bigg [  \int_S \|\pi_1(\cdot | s) - \pi_2(\cdot | s)\| ^2 d\rho(s) \bigg ] ^{1/2 } , \label{eq:policy_norms}
% \end{align}
% where $\langle  \cdot, \cdot \rangle  $ is inner product on $\RR^{|A|}$ and $\| \cdot \|$ is the $\ell_1$ norm on $\RR^{|A|}$.
% Recall that we let 
%  $\| \cdot \|_{*}$  denote the $\ell_{\infty}$-norm on $\RR^{|A|}$.
% Then, the inner product $\langle  \cdot, \cdot \rangle_{\rho}$ and the norm $\| \cdot \|_{\rho}$ in \eqref{eq:policy_norms} induce a distance  between two  action-value functions $Q_1$ and $Q_2$,
%   which is given by 
% \begin{align}
% \| Q _1 - Q_2 \|_{\rho, *} = \biggl [  \int_{S}  \| Q_1(s, \cdot ) - Q_2 (s, \cdot) \|_{*}^2 d \rho (s) \bigg ] ^{1/2 }. \label{eq:value_dualnorm} 
% \end{align}
Under \Cref{ass:linear_assumption}, 
$\theta^\star (\pi)$ is the solution to the inner problem with 
$Q^{\pi} (\cdot, \cdot) = \phi(\cdot, \cdot)^\top \theta^{\star} (\pi)$. 
Below we first show that $\theta^{\star} (\pi)$ and $Q^{\pi}$ are Lipschitz continuous maps with respect to $\| \cdot \|_{\rho^{\star},1}$, where $\rho^{\star}$ is the visitation measure of an optimal policy $\pi^\star$. 

\begin{Lemma}
\label{lem:rl_lip}
Under \Cref{ass:bdd_reward}--\ref{ass:concentrability},
for any two policies $ \pi_1, \pi_2 \in X$,  
\begin{align}
\| Q^{\pi_1}- Q^{\pi_2} \|_{\rho^\star, \infty} & \leq ( 1- \gamma)^{-2} \cdot \overline {r} \cdot C_{\rho} \cdot \| \pi_1 - \pi_2 \|_{\rho^\star, 1}, \notag \\
    \bigl \| \theta^{\star} (\pi_1) - \theta^{\star} (\pi_2) \bigr \|_2 & \leq  ( 1- \gamma)^{-2} \cdot \overline {r} \cdot C_{\rho}/  \mu_{\phi} \cdot \| \pi_1 - \pi_2 \|_{\rho^\star , 1},
\end{align}
where $\overline r$ is an upper bound on the reward function, $\mu_{\phi}$ is specified in \Cref{ass:stationary}, and  $C_{\rho} $ is defined in \Cref{ass:concentrability}. 
\end{Lemma} 
\vspace{-0.2cm}
The proof of the above lemma is relegated to \S\ref{app:lips_RL}. 

In the sequel, we first derive coupled inequalities on the non-negative sequences ${\rm OPT}^k \eqdef \EE[\ell(\pi^k) - \ell(\pi^\star)]$, $\Delta_Q^k \eqdef \EE[ \| \theta^k - \theta^\star( \pi^{k-1} ) \|^2 ]$, $\EE[ \| \pi^k - \pi^{k+1} \|_{\rho^\star,1}^2 ]$, then we apply \Cref{lem:recur_lem} to derive the convergence rates of \acm. 
Using the performance difference lemma %\cite[Lemma 6.1]{kakade2002approximately} 
[cf.~\eqref{eq:key:1}], we obtain the following
%we have 
%\beqq
%     \ell(\pi^{k}) - \ell(\pi^\star) =  -  (1-\gamma)^{-1}\langle Q^{\pi^k}, \pi^k - \pi^\star \rangle_{\rho^\star}. 
%    %  \;  \ell(\pi^{k+1}) - \ell(\pi^\star) & = -  (1-\gamma)^{-1}\langle Q^{\pi^{k+1}}, \pi^{k+1} - \pi^\star \rangle_{\rho^\star}\nonumber.
%\eeqq
%By direct computation,  we obtain
\begin{align}
    \ell(\pi^{k+1})-\ell(\pi^k) 
     & =  - (1-\gamma)^{-1}\langle Q^{\pi^{k+1}}, \pi^{k+1} - \pi^\star \rangle_{\rho^\star} + (1-\gamma)^{-1} \langle Q^{\pi^k}, \pi^k - \pi^\star \rangle_{\rho^\star}\nonumber\\
    % & \quad  =  (1-\gamma)^{-1}\Bigl [\langle - Q^{\pi^{k+1}}, \pi^{k+1} - \pi^\star \rangle_{\rho^\star} +  \langle Q^{\pi^k}, \pi^k - \pi^{k+1} \rangle_{\rho^\star} +  \langle Q^{\pi^k}, \pi^{k+1} - \pi^\star \rangle_{\rho^\star}\Bigr ] \nonumber\\
    & =  (1-\gamma)^{-1}\langle - Q^{\pi^k}, \pi^{k+1} - \pi^{k} \rangle_{\rho^\star} +   (1-\gamma)^{-1}\langle Q^{\pi^{k}}-Q^{\pi^{k+1} }, \pi^{k+1} - \pi^\star \rangle_{\rho^\star} .\label{eq:policy_bd_dy}
    \end{align}
    Applying the inequality \eqref{eq:rho_holder}, we further have 
    \begin{align}
       & \langle Q^{\pi^{k}}-Q^{\pi^{k+1} }, \pi^{k+1} - \pi^\star \rangle_{\rho^\star}   \leq \| Q^{\pi^{k}}-Q^{\pi^{k+1} } \|_{\rho^\star, \infty }  \| \pi ^{k+1} - \pi^\star \|_{\rho^\star,1 } \label{eq:policy_bd_dy4} \\
        & ~~\qquad \qquad 
        % \leq 2 \| Q^{\pi^{k}}-Q^{\pi^{k+1} } \|_{\rho^\star,\infty} 
          \leq 2 L_Q \| \pi^k - \pi^{k+1} \|_{\rho^\star,1} \leq \frac{1-\gamma}{4 \alpha}\|{\pi^{k+1}}-{\pi^k}\|^2_{\rho^\star,1} + \frac{4 L^2_Q \alpha}{1-\gamma}, \label{eq:policy_bd_dy2}
    \end{align}
    where $L_Q := (1-\gamma)^{-2} \overline{r} \cdot C_\rho$.
    The above inequality follows from $\| \pi^k - \pi^{k+1} \|_{\rho^\star,1} \leq 2  $ for any $\pi_1, \pi_2 \in X$ and applying \Cref{lem:rl_lip}.  
    Then, combining \eqref{eq:policy_bd_dy}, \eqref{eq:policy_bd_dy2} leads to
    \begin{align} 
     \ell(\pi^{k+1})-\ell(\pi^k) \le  \frac{-1}{1-\gamma}\langle Q^{\pi^k}, \pi^{k+1} - \pi^{k} \rangle_{\rho^\star} +  \frac{1}{4\alpha}\|{\pi^{k+1}}-{\pi^k}\|^2_{\rho^\star,1} +   4 \frac{\alpha L^2_Q}{(1-\gamma)^2} . \label{eq:policy_bd_dy3}
\end{align}
 
Let us bound the first term in the right-hand side of \eqref{eq:policy_bd_dy3}. To proceed, note that the policy update \eqref{zweq:ppo_restate_n} can be implemented for each state individually as below:
\begin{align}
\pi^{k+1} (\cdot | s ) = \hspace{-0.3cm} \argmin_{ \nu :\!~ \sum_{a} \nu(a) = 1, \nu(a) \geq 0  } \Bigl \{ - ( 1- \gamma )^{-1} \cdot \langle Q_{\theta^{k+1}} (s, \cdot ), \nu  \rangle   + 1/ \alpha_k \cdot D_{\psi } \big ( \nu,  \pi^k (\cdot | s)    \bigr  )  \Big \},
\end{align}
for all $s \in S$. Observe that we can modify ${\rho^{\pi^k}}$ in  \eqref{zweq:ppo_restate_n} to ${\rho^{\star}}$ without changing the optimal solution for this subproblem. Specifically, \eqref{zweq:ppo_restate_n} can be written as 
\begin{align} 
\pi^{k+1}   = \argmin_{\pi \in X} \Bigl\{-(1-\gamma)^{-1}\langle Q_{\theta^{k+1}}, \pi-\pi^k \rangle_{\rho^\star } + \frac{1}{\alpha} \bar{D}_{\psi, \rho^\star}(\pi, \pi^k)\Bigr\}. \label{eq:restate_ppo_agin}
\end{align}
We have 
\beqq \notag
- (1-\gamma)^{-1} \langle Q^{\pi^k}, \pi^{k+1}-\pi^k\rangle_{\rho^*} = (1-\gamma)^{-1} \big[ \langle Q_{\theta^{k+1}} - Q^{\pi^k}, \pi^{k+1}-\pi^k\rangle_{\rho^{\star}} - \langle Q_{\theta^{k+1}}, \pi^{k+1}-\pi^k\rangle_{\rho^{\star}} \big].
\eeqq
Furthermore, from \eqref{eq:restate_ppo_agin}, we obtain
\begin{align}
& \frac{\langle Q_{\theta^{k+1}}, \pi^{k+1}-\pi^k\rangle_{\rho^{\star}}}{1-\gamma} \geq \frac{1}{\alpha} \sum_{s \in S} \langle \nabla D_{\psi}(\pi^{k+1}(\cdot | s), \pi^k(\cdot | s)), \pi^{k+1}(\cdot | s) - \pi^k(\cdot | s)\rangle \rho^\star(s) ,
 \label{eq:app_innerprod}
\end{align}
where the inequality follows from the optimality condition of the mirror descent step.
Meanwhile, the 
$1$-strong convexity of ${D}_{\psi}(\cdot,\cdot)$ 
implies that 
\begin{align}
\bigl \langle \nabla D_{\psi}\bigl (\pi^{k+1}(\cdot | s), \pi^k(\cdot | s)\bigr ), \pi^{k+1}(\cdot | s) - \pi^k(\cdot | s) \bigr \rangle 
% & = \langle \nabla \psi \bigl ( \pi^{k+1}(\cdot | s) \bigr ) -\nabla \psi \bigl ( \pi^{k}(\cdot | s) \bigr ) , 
% \pi^{k+1}(\cdot | s) - \pi^k(\cdot | s) \bigr \rangle \notag \\
&
\ge  \| \pi^{k+1}(\cdot | s) - \pi^k(\cdot | s) \|^2. \label{eq:apply_kl_convexity}
\end{align}
Thus, combining \eqref{eq:app_innerprod} and \eqref{eq:apply_kl_convexity}, and applying 
  Young's inequality, we further have 
\begin{align}
 & - (1-\gamma)^{-1} \langle Q^{\pi^k}, \pi^{k+1}-\pi^k\rangle_{\rho^*} \nonumber\\ 
 & \quad  \leq   \frac{1 }{4\alpha}  \|\pi^{k+1}-\pi^k\|_{\rho^\star,1}^2 + \alpha (1-\gamma)^{-2} \cdot  \|Q^{\pi^k} - Q_{\theta^{k+1}} \|_{\rho^{\star},\infty}^2 - \frac{1}{ \alpha}\|\pi^{k+1}-\pi^k\|_{\rho^{\star},1}^2\nonumber\\
   & \quad   =   \alpha (1-\gamma)^{-2} \cdot \|Q^{\pi^k} - Q_{\theta^{k+1}} \|_{\rho^{\star},\infty}^2- \frac{3}{4\alpha}\|\pi^{k+1}-\pi^k\|_{\rho^\star,1}^2 . \label{eq:meaningless1}
\end{align}
% Notice that this step is analogous to using the property \eqref{eq:lip:f:bar} from \Cref{lem:lips}.
% Furthermore, recall that we define $\Delta^{k+1}_Q : =\mathbb{E}[\| \theta^{k+1} - \theta^{\star} (\pi^k) \|_2^2]$. It remains to relate $\|Q^{\pi^k} - Q_{\theta^{k+1}} \|^2_{\rho^{\star},\infty}$ to $\Delta_Q^{k+1}$.
By direct computation and using $\| \phi(s,a ) \| \leq 1$ [cf.~\Cref{ass:linear_assumption}], we have 
\begin{align}
\| Q^{\pi^k} - Q_{\theta^{k+1}} \|_{\rho^\star , \infty} ^2 & = \sum_{s \in S} \Bigl \{ \max_{a\in A} \bigl  |  \phi(s,a) ^\top [ \theta^\star (\pi^k)  - \theta^{k+1} ]  \bigr | \Bigr \}^2  \rho^\star (s)  \notag \\
& \leq \sum_{s \in S}  \max_{ a \in A}  \bigl \{\|   \phi(s,a) \|^2 \bigr \} \|  \theta^\star (\pi^k)  - \theta^{k+1} \|^2 \rho^\star (s)  \leq \|  \theta^\star (\pi^k)  - \theta^{k+1} \|^2. \label{eq:meaningless2}
\end{align}
Combining \eqref{eq:policy_bd_dy3}, \eqref{eq:meaningless1}, and \eqref{eq:meaningless2},
we obtain 
\[
    \ell(\pi^{k+1})-\ell(\pi^k)  \le      \alpha (1-\gamma)^{-2} \|  \theta^\star (\pi^k)  - \theta^{k+1} \|^2  -  \frac{ 1}{2\alpha}\|{\pi^{k+1}}-{\pi^k}\|^2_{\rho^\star,1} + 4   (1-\gamma)^{-2}L^2_Q \alpha .
\]
Taking full expectation leads to
\beq \label{eq:optdiff_d1}
    \mbox{OPT}^{k+1} - \mbox{OPT}^{k} \le   \alpha (1-\gamma)^{-2}   \Delta^{k+1}_Q - \frac{1}{2\alpha} \EE[ \|\pi^{k+1}-\pi^k\|_{\rho^*,1}^2] +  4  (1-\gamma)^{-2} L^2_Q \alpha.
\eeq

% {\noindent \bf Step 2 (Bounding $\Delta_Q^{k}$).}
% Next let us briefly look at the inner problem. 
Next, we consider the convergence of $\Delta_Q^k$. 
Let $\mathcal{F}_k = \sigma \{ \theta^0, \pi^0, \ldots, \theta^k, \pi^k \}$ be the $\sigma$-algebra generated by the first $k+1$ actor and critic updates.  
Under \Cref{ass:linear_assumption}, we can write the conditional expectation of $h_g^k$ as   
\begin{align}
\EE [ h_g^k | \cF_k ] & =  \EE_{ \mu^{\pi^k} } \bigl [ \{ Q_{\theta^k}(s, a) - r(s, a) - \gamma Q_{\theta^k}(s', a') \} \phi(s, a) | \cF_k \bigr] \notag \\
& =  \EE_{ \mu^{\pi^k} } \bigl[ \phi (s, a) \{ \phi(s,a)  - \gamma \phi (s', a') \}^\top   \bigr ] [\theta^{k} - \theta^\star (\pi^k) ],
\end{align} 
where $\EE_{\mu^{\pi^k}}[\cdot]$ denotes the expectation taken with $s \sim \mu^{\pi^k}$, $a \sim \pi^k(\cdot|s)$, $s' \sim P(\cdot|s,a)$, $a' \sim \pi^k(\cdot|s')$.
Under \Cref{ass:linear_assumption} and \ref{ass:stationary}, Lemma 3 of \cite{bhandari2018finite} shows that $\EE[h_g^k|\cF_k]$ is a semigradient of the MSBE function $\| Q_{\theta^k} - Q_{\theta^\star(\pi^k) } \|_{\mu^{\pi^k} \otimes \pi^k }^2$. Particularly, we obtain 
\beqq
   \EE [ h_g^k | \cF_k ]^\top [ \theta^k - \theta^\star (\pi^k) ]    
    \geq ( 1- \gamma) \| Q_{\theta^k} - Q_{\theta^\star(\pi^k) } \|_{\mu^{\pi^k} \otimes \pi^k } ^2  \geq \mu_{\sf td} \| \theta^k - \theta^{\star } (\pi^k) \|_2^2,  \label{eq:russo_lem1} 
\eeqq
where we have let $\mu_{\sf td} = (1 - \gamma) \mu_\phi ^2$.
Moreover, Lemma 5 of \cite{bhandari2018finite} demonstrates that the second order moment $\EE[ \| h_g^k \|_2^2 | \cF_k ]$ is bounded as
\beqq
   \EE [ \| h_g^k \|_2^2  | \cF_k ] \leq  8 \| Q_{\theta^k} - Q_{\theta^\star(\pi^k) } \|_{\mu^{\pi^k} \otimes \pi^k } ^2  + \sigma_{{\sf td}}^2 \leq 8 \| \theta^k - \theta^{\star } (\pi^k) \|_2^2 + \sigma_{{\sf td}}^2 , \label{eq:russo_lem2} 
\eeqq
where $\sigma_{{\sf td}}^2 = 4 \overline r^2 (1- \gamma)^{-2} $. 
Combining \eqref{eq:russo_lem1}, \eqref{eq:russo_lem2} and recalling $\beta \leq \mu_{\sf td} / 8$, it holds 
\begin{align}
\mathbb{E}[\|\theta^{k+1}-\theta^{\star}(\pi^k)\|_2^2|{\cal F}_k]&  =  \|\theta^{k}-\theta^{\star}(\pi^k)\|_2^2 - 2 \beta \EE [ h_g^k | \cF_k ]^\top [ \theta^k - \theta^\star (\pi^k) ] + \beta^2  \EE [ \| h_g^k \|_2^2  | \cF_k ] \notag \\
& \leq \bigl ( 1 - 2 \mu_{\sf td} \beta  + 8 \beta ^2 ) \cdot  \|\theta^{k}-\theta^{\star}(\pi^k)\|_2^2  + \beta ^2 \cdot \sigma_{\sf td}^2  \notag \\
& \leq (1 - \mu_{\sf td} \beta ) \cdot  \|\theta^{k}-\theta^{\star}(\pi^k)\|_2^2  + \beta ^2 \cdot \sigma_{\sf td}^2 . \label{eq:td_bounds1}
\end{align}
By Young's inequality and Lemma \ref{lem:rl_lip},
we further have 
\begin{align} 
& \mathbb{E}[\|\theta^{k+1}-\theta^{\star}(\pi^k)\|_2^2|{\cal F}_k]  \nonumber\\
&\quad  \le (1+c) (1- \mu_{ {\sf td}} \beta)   \|\theta^{k}-\theta^{\star}(\pi^{k-1})\|_2^2  + (1+1/c)  \|\theta^{\star} (\pi^{k}) -\theta^{\star}(\pi^{k-1})\|_2^2  + \beta^2 \sigma_{\sf {td}}^2 \nonumber\\
%&\quad  \le (1- \mu_{ {\sf td}} \beta /2 ) \cdot \|\theta^{k}-\theta^{\star}(\pi^{k-1})\|_2^2 + \Big (\frac{2}{\mu_{ {\sf td}}  \beta }-1 \Bigr ) \cdot \|\theta^{\star} (\pi^k) -\theta^{\star}(\pi^{k-1})\|_2^2 + \beta^2 \cdot \sigma_{\sf {td}}^2 \nonumber\\
&\quad  \le   (1-\mu_{ {\sf td}} \beta /2 )  \|\theta^{k}-\theta^{\star}(\pi^{k-1})\|_2^2 + \Big (\frac{2}{\mu_{\sf {td}}   \beta }-1 \Bigr ) \overline L_Q   \| \pi^k - \pi^{k+1} \|^2_{\rho^\star,1}  + \beta^2 \sigma_{ {\sf td}}^2, \label{eq:td_bounds2}
\end{align}
where we have chosen $c>0$ such that $(1+ c) (1- \mu_{\sf td} \beta) = 1 - \mu_{\sf td} \beta /2$, which implies that   
$1/ c + 1  = 2/ (\mu_{\sf td} \beta ) -1 >0$ [cf.~\eqref{eq:alpha:convex:new}];
The last inequality comes from \Cref{lem:rl_lip} with the constant $\overline L_Q = (1- \gamma)^{-4} \cdot \overline r \cdot C_{\rho}^2 \cdot \mu_{\phi}^{-2}$. 

From \eqref{eq:optdiff_d1}, \eqref{eq:td_bounds2}, we identify that condition \eqref{eq:recurnew} of \Cref{lem:recur_lem} holds with:
\begin{align*}
\begin{split}
	& \Omega^k = {\rm OPT}^k, \; \Theta^k = \EE[ \| \pi^k - \pi^{k-1} \|_{\rho^*,1}^2], \; {\rm c}_0 = \frac{1}{2\alpha}, \; {\rm c}_1 = \frac{\alpha}{(1-\gamma)^2}, \; {\rm c}_2 = \frac{4 L^2_Q \alpha}{(1-\gamma)^2}  , \\
	& \Upsilon^k = \EE[ \| \theta^k - \theta^\star(\pi^{k-1}) \|^2 ], \; {\rm d}_0 = \mu_{ {\sf td}} \beta /2, \; {\rm d}_1 = \Big (\frac{2}{\mu_{\sf {td}}   \beta }-1 \Bigr ) \overline L_Q>0, \; {\rm d}_2 = \beta^2 \cdot \sigma_{ {\sf td}}^2.
\end{split}
\end{align*} %This concludes our first step. 
Selecting the step sizes as in 
\eqref{eq:alpha:convex:new}, one can verify that 
$\frac{\alpha}{\beta}<\frac{\mu_{\sf td}(1-\gamma)}{16 \sqrt{\bar{L}_Q}}$.
This ensures \vspace{-.1cm}
\begin{align}
{\rm c}_0 - {\rm c}_1 {\rm d}_1 ( {\rm d}_0)^{-1} >{1}/{(4\alpha)}, \;\;  {\rm d}_0 - {\rm c}_1 {\rm d}_1 ({\rm c}_0)^{-1} >{\mu_{\sf td}\beta}/{4}.
\end{align}
Applying \Cref{lem:recur_lem}, we obtain for any $K \geq 1$ that
\begin{align*}
\frac{1}{K}\sum_{k=1}^{K}\mathbb{E}[\|\pi^k-\pi^{k+1}\|^2_{\rho^\star,1}]&\le \frac{{\rm OPT}^0 \cdot  4 \alpha + \frac{8 \alpha^2 (1-\gamma)^{-2}}{\mu_{\sf td}\beta}(\Delta^0_Q + \beta^2 \sigma^2_{\sf td})}{K} + \frac{8 \alpha^2 (4 {L}^2_Q  + \beta \sigma^2_{\sf td}/\mu_{\sf td})}{(1-\gamma)^2}\\
\frac{1}{K}\sum_{k=1}^{K}\mathbb{E}[\Delta^{k+1}_Q] &\le  \frac{\mathbb{E}[\Delta^0_Q] + \beta^2 \sigma^2_{\sf td} + \frac{4 \alpha}{\mu_{\sf td} \beta} \bar{L}_Q {\rm OPT}^0}{\mu_{\sf td} \beta K/4} + \frac{\beta^2 \sigma^2_{\sf td} + \frac{16 \alpha^2}{\mu_{\sf td}\beta} (1-\gamma)^{-2}{L}^2_Q \alpha^2}{\mu_{\sf td}\beta/4}.
\end{align*}
Particularly, plugging in $\alpha \asymp \Kmax^{-3/4}$, $\beta \asymp \Kmax^{-1/2}$ shows that the convergence rates are $\Kmax^{-1} \sum_{k=1}^{\Kmax}\mathbb{E}[\|\pi^k-\pi^{k+1}\|^2_{\rho^\star,1}] = \mathcal{O}(\Kmax^{-3/2})$, $\Kmax^{-1} \sum_{k=1}^{\Kmax}\mathbb{E}[\Delta^{k+1}_Q] = \mathcal{O}(\Kmax^{-1/2})$.
 
Our last step is to analyze the convergence rate of the objective value ${\rm OPT}^k$.
% Recall that $\pi^{k+1} $ is the solution to the mirror descent step \eqref{eq:restate_ppo_agin}.
To this end, we observe the following three-point inequality \cite{beck2017first}
\begin{align}
\frac{-\langle Q_{\theta^{k+1}}, \pi^{k+1} - \pi^\star \rangle_{\rho^\star}}{1-\gamma} \le \frac{1}{\alpha}\Bigl [ \bar{D}_{\psi,\rho^\star}( \pi^\star, \pi^k)- \bar{D}_{\psi,\rho^\star}( \pi^\star, \pi^{k+1})- \bar{D}_{\psi,\rho^\star}(\pi^{k+1},\pi^k)\Bigr ] .\label{eq:mirror:property} 
\end{align}
Meanwhile, by the inequalities \eqref{eq:rho_holder}, \eqref{eq:policy_bd_dy}, \eqref{eq:policy_bd_dy4},  we have{\small
\begin{align}
   & \ell(\pi^{k+1})-\ell(\pi^k)  =  - \frac{1}{1-\gamma} \langle   Q^{\pi^{k}}, \pi^{k+1} - \pi^{k} \rangle_{\rho^\star} +  \frac{1}{1-\gamma}\langle Q^{\pi^{k}}-Q^{\pi^{k+1}}, \pi^{k+1} - \pi^\star \rangle_{\rho^\star} \nonumber\\
     & \le \frac{1}{1-\gamma} \big[ \langle - Q^{\pi^k}, \pi^{k+1} - \pi^{\star} \rangle_{\rho^\star} - \langle Q^{\pi^k}, \pi^{\star} - \pi^{k} \rangle_{\rho^\star} + \| {\pi^{k+1}}-{\pi^\star }\|_{\rho^{\star},1} \| Q^{\pi^k} - Q^{\pi^{k+1} } \|_{\rho^\star, \infty} \big] \nonumber\\
    & \leq \frac{1}{1-\gamma} \big[ \langle -Q^{\pi^k}, \pi^{k+1} - \pi^{\star} \rangle_{\rho^\star} - \langle Q^{\pi^k}, \pi^{\star} - \pi^{k} \rangle_{\rho^\star} +   2 L_Q \| \pi^{k+1} - \pi^k \|_{\rho^\star,1} \big] , \notag
    % \label{eq:rl_final1}
    \end{align}
}where the last inequality follows from \Cref{lem:rl_lip}. Now, with the performance difference lemma $\ell(\pi^*) - \ell(\pi^k) = (1-\gamma)^{-1}\langle -Q^{\pi^k}, \pi^{\star} - \pi^{k} \rangle_{\rho^\star}$, the above simplifies to 
\beqq \notag
\ell(\pi^{k+1}) - \ell(\pi^\star) \leq (1-\gamma)^{-1} \big[ - \langle Q^{\pi^k}, \pi^{k+1} - \pi^{\star} \rangle_{\rho^\star} +   2 L_Q \| \pi^{k+1} - \pi^k \|_{\rho^\star,1} \big]
\eeqq
With 
$\langle Q^{\pi^k}, \pi^{k+1} - \pi^{\star} \rangle_{\rho^\star} = \langle Q^{\pi^k} - Q_{\theta^{k+1}}, \pi^{k+1} - \pi^{\star} \rangle_{\rho^\star} + \langle Q_{\theta^{k+1}}, \pi^{k+1} - \pi^{\star} \rangle_{\rho^\star}$ and applying the three-point inequality \eqref{eq:mirror:property}, we have {\small
    \begin{align} 
      \ell(\pi^{k+1})-\ell(\pi^\star) 
    % &  \leq  (1-\gamma)^{-1} \big[ - \langle Q^{\pi^k}-Q_{\theta^{k+1}}, \pi^{k+1} - \pi^{\star} \rangle_{\rho^\star} - \langle Q_{\theta^{k+1}}, \pi^{k+1} - \pi^{\star} \rangle_{\rho^\star} - \langle Q^{\pi^k}, \pi^{\star} - \pi^{k} \rangle_{\rho^\star} \big] \notag \\
    % & \qquad + 2 (1-\gamma)^{-1} L_Q \cdot \| \pi^{k+1} - \pi^k \|_{\rho^\star,1} \notag \\
    & \leq \frac{2}{1-\gamma} \big[ L_Q \| \pi^{k+1} - \pi^k \|_{\rho^\star,1} + \| Q^{\pi^k} - Q_{\theta^{k+1}} \|_{\rho^\star , \infty} \big] \notag \\
    &\quad + \frac{1}{\alpha}\Bigl [ \bar{D}_{\psi,\rho^\star}( \pi^\star, \pi^k)- \bar{D}_{\psi,\rho^\star}( \pi^\star, \pi^{k+1})- \bar{D}_{\psi,\rho^\star}(\pi^{k+1},\pi^k)\Bigr ].\notag \\
    & \hspace{-2cm} \leq  \frac{2}{1-\gamma} \big[ L_Q \| \pi^{k+1} - \pi^k \|_{\rho^\star,1} + \| \theta^{\star}
        (\pi^k) - \theta^{k+1} \|_{2} \big] + \frac{1}{\alpha}\bigl [ \bar{D}_{\psi,\rho^\star}( \pi^\star, \pi^k)- \bar{D}_{\psi,\rho^\star}( \pi^\star, \pi^{k+1}) \bigr ] \notag
    \end{align}}%
    % We can further write the above as
    % \begin{align}
    %     \ell(\pi^{k+1}) - \ell(\pi^\star) & \leq  2     (1-\gamma)^{-1} L_Q \cdot \| \pi^{k+1} - \pi^k \|_{\rho^\star,1}   +2  (1-\gamma)^{-1} \cdot  \| \theta^{\star}
    %     (\pi^k) - \theta^{k+1} \|_{2} \notag \\
    % &\quad + \frac{1}{\alpha}\bigl [ \bar{D}_{\psi,\rho^\star}( \pi^\star, \pi^k)- \bar{D}_{\psi,\rho^\star}( \pi^\star, \pi^{k+1}) \bigr ]  . \label{eq:rl_final3}
    % \end{align}
    where the last inequality uses \eqref{eq:meaningless2} and the fact that $\bar{D} _{\psi, \rho^\star}$ is non-negative. 
    Finally, taking the full expectation on both sides of the inequality, we obtain 
\begin{align}
\mbox{OPT}^{k+1} &   \le 2  (1-\gamma)^{-1}  \EE[ \| \theta^{\star}
        (\pi^k) - \theta^{k+1} \|_{2} ] + 2  (1-\gamma)^{-1} L_Q  \EE[ \| {\pi^{k+1}}-{\pi^k}\|_{\rho^{\star},1} ]  \nonumber\\
& \quad +\frac{1}{\alpha}\EE \big [\bar{D}_{\psi,\rho^\star}( \pi^\star, \pi^k)- \bar{D}_{\psi,\rho^\star}( \pi^\star, \pi^{k+1}) \bigr ] .
\end{align}
Summing up both sides from $k=0$ to $k=\Kmax-1$ and dividing by $\Kmax$ yields
\begin{align}
\begin{split}
\frac{1}{\Kmax} \sum_{k=1}^{\Kmax} \mbox{OPT}^{k} & \leq \frac{1}{\alpha \Kmax} \big\{ \bar{D}_{\psi,\rho^\star}( \pi^\star, \pi^0)- \bar{D}_{\psi,\rho^\star}( \pi^\star, \pi^{\Kmax}) \big\} \\
& \hspace{-1cm} + \frac{2  }{(1-\gamma)} \frac{1}{\Kmax} \sum_{k=1}^{\Kmax} \big\{ \EE[\| \theta^{\star}
        (\pi^{k-1}) - \theta^{k} \|_{2} ] + L_Q \cdot \EE[ \| {\pi^{k}}-{\pi^{k-1}} \|_{\rho^{\star},1} ]  \big\}.
\end{split}
\end{align}
Using Cauchy-Schwarz's inequality, it can be easily seen that the right-hand side is ${\cal O}(\Kmax^{-1/4})$. This concludes the proof of the theorem.

\vspace{-0.2cm}
\subsection{Proof of \Cref{lem:rl_lip}} \label{app:lips_RL}
We first bound $|Q^{\pi_1} (s,a) - Q^{\pi_2} (s,a) |$. %for any $(s,a) \in S\times A$.
By the Bellman equation \eqref{zweq:bellman} and the performance difference lemma \eqref{eq:key:1}, we have 
\begin{align}
Q^{\pi_1} (s,a) - Q^{\pi_2} (s,a) 
%& = \bigl [r(s,a) + \sum_{s' \in S} P(s'| s,a) V^{\pi_1} (s') \bigr ]     - \bigl [r(s,a) + \sum_{s' \in S } P(s'| s,a) V^{\pi_2} (s')  \bigr ]\notag \\
& =  \sum_{s' \in S} P(s' | s,a) \cdot \bigl [ V^{\pi_1}(s') - V^{\pi_2} (s')  \bigr  ] \notag \\
& \hspace{-1.5cm} = (1 - \gamma)^{-1}  \sum_{s'\in S} P(s' | s,a) \cdot \EE_{\tilde s \sim \tilde \varrho(s',\pi_1) } \bigl [\big  \langle  Q^{\pi_2} (\tilde s, \cdot ) , \pi_1(\cdot | \tilde s) - \pi_2(\cdot | \tilde s) \big \rangle  \bigr ], \label{eq:rl_lip1} 
\end{align}
where $\tilde \varrho(s', \pi_1)$ is the visitation measure obtained by the Markov chain induced by $\pi_1$ with the initial state fixed to $s'$. 
Recall the definition of the visitation measure  $\varrho (s,a, \pi)$ in \eqref{eq:define_sa_visitation}. 
We rewrite \eqref{eq:rl_lip1} as 
\begin{align}
    Q^{\pi_1} (s,a) - Q^{\pi_2} (s,a) = (1 - \gamma)^{-1 } \cdot \EE_{\tilde s \sim   \varrho(s , a,\pi_1) } \bigl [\big  \langle  Q^{\pi_2} (\tilde s, \cdot ) , \pi_1(\cdot | \tilde s) - \pi_2(\cdot | \tilde s) \big \rangle  \bigr ]. \label{eq:rl_lip2}
\end{align}
Moreover, notice that $\sup_{(s,a) \in S\times A } |Q^{\pi}(s,a) | \leq (1- \gamma)^{-1} \cdot \overline {r}$ under \Cref{ass:bdd_reward}.
Then, applying H\"older's inequality to \eqref{eq:rl_lip2}, we obtain 
\begin{align}
   \bigl | Q^{\pi_1} (s,a) - Q^{\pi_2} (s,a)  \bigr |  & \leq ( 1- \gamma )^{-1} \cdot \EE_{\tilde s \sim \tilde \varrho(s',\pi_1) } \bigl [    \|  Q^{\pi_2} (\tilde s, \cdot ) \|_\infty  \cdot \| \pi_1(\cdot | \tilde s) - \pi_2(\cdot | \tilde s) \|_1   \bigr ] \notag \\
   & \hspace{-1.75cm} \leq (1 -\gamma)^{-2} \cdot \overline{r} \cdot \EE_{\tilde s \sim \rho^\star } \biggl [ \frac{\varrho(s,a, \pi)}{\rho^\star} (\tilde s) \cdot   \| \pi_1(\cdot | \tilde s) - \pi_2(\cdot | \tilde s) \|_1 \bigg] \notag \\
   & \hspace{-1.75cm} \leq  (1 -\gamma)^{-2} \cdot \overline{r} \cdot \biggl \{    \EE_{\tilde s \sim \rho^\star } \biggl [ {\big| \frac{\varrho(s,a, \pi)}{\rho^\star} (\tilde s) \big|^2}  \bigg] \EE_{\tilde s \sim \rho^\star }  \bigl [ \| \pi_1(\cdot | \tilde s) - \pi_2(\cdot | \tilde s) \|_1^2 \bigr ] \biggr \} ^{1/2}  \notag \\
   & \hspace{-1.75cm} \leq ( 1- \gamma)^{-2} \cdot \overline {r} \cdot C_{\rho} \cdot \| \pi_1 - \pi_2 \|_{\rho^\star,1}, \label{eq:rl_lip3}
\end{align}
where the second inequality is from the boundedness of $Q^{\pi}$, the third one is the Cauchy-Schwarz inequality, and the last one is from \Cref{ass:concentrability}. 
Finally, we have 
\begin{align*}
\| Q^{\pi_1}- Q^{\pi_2} \|_{\rho^\star, \infty} \leq ( 1- \gamma)^{-2} \cdot \overline {r} \cdot C_{\rho} \cdot \| \pi_1 - \pi_2 \|_{\rho^\star,1}.
\end{align*}
It remains to bound $\| \theta^\star (\pi_1) - \theta^\star (\pi_2) \|^2$. 
Under \Cref{ass:linear_assumption}, we have 
\begin{align}
    \| Q^{\pi_1} - Q^{\pi_2} \| _{\mu^{\pi^\star}\otimes \pi^\star} ^2 &  = \EE_{s\sim \mu^{\pi^\star}, a\sim \pi^\star (\cdot | s)} \bigl \{ \bigl [ Q^{\pi_1} (s,a) - Q^{\pi_2} (s,a) \bigr ]^2 \bigr \} \notag \\
    &  = \EE_{s\sim \mu^{\pi^\star}, a\sim \pi^\star (\cdot | s)} \Bigl (    \bigl \{  \phi (s,a)^\top [ \theta^{\star}(\pi_1)  - \theta^{\star}(\pi_2) ]  \bigr \} ^2 \Bigr ) \notag \\
    & = [ \theta^{\star}(\pi_1)  - \theta^{\star}(\pi_2) ] ^\top \Sigma_{\pi^\star } [ \theta^{\star}(\pi_1)  - \theta^{\star}(\pi_2) ].
\end{align}
Then, combining  \Cref{ass:stationary} and \eqref{eq:rl_lip3}, we have 
\begin{align}
    \mu_{\phi}^2 \|  \theta^{\star}(\pi_1)  - \theta^{\star}(\pi_2) \|^2   \leq \| Q^{\pi_1} - Q^{\pi_2} \| _{\mu^{\pi^\star}\otimes \pi^\star} ^2 \leq ( 1- \gamma)^{-4} \cdot {\overline {r}^2} \cdot C_{\rho} ^2 \cdot \| \pi_1 - \pi_2 \|_{\rho^\star,1}^2,
\end{align}
which yields the second inequality in Lemma \ref{lem:rl_lip}. We conclude the proof.
		
 		% \section{Other Technical Results}

\section{Auxiliary Lemmas} The proofs for the lemmas below can be found in the online appendix \cite{Hong-TTSA-2020}.
\begin{Lemma} \label{lem:aux1} \cite[Lemma 12]{kaledin2020finite}
Let {\blue $a>0$}, $\{ \gamma_j \}_{j \geq 0}$ be a non-increasing, non-negative sequence such that $\gamma_0 < 1/a$, it holds for any $k \geq 0$ that
\beqq
\sum_{j=0}^k \gamma_j \prod_{\ell=j+1}^k (1 - \gamma_\ell a) \leq \frac{1}{a}.
\eeqq
\end{Lemma}

\begin{Lemma} \label{lem:aux2}
Fix a real number $1 < q \leq 2$.
Let {\blue $a>0$}, $\{ \gamma_j \}_{j \geq 0}$ be a non-increasing, non-negative sequence such that $\gamma_0 < 1/(2a)$. Suppose that
$\frac{\gamma_{\ell-1}}{ \gamma_{\ell} } \leq 1 + \frac{a}{2(q-1)} \gamma_\ell$.
Then, it holds for any $k \geq 0$ that
\beqq
\sum_{j=0}^k \gamma_j^q \prod_{\ell=j+1}^k (1 - \gamma_\ell a) \leq \frac{2}{a} \gamma_k^{q-1}.
\eeqq
\end{Lemma}

\begin{Lemma} \label{lem:aux3}
Fix the real numbers $a, b > 0$.
Let $\{ \gamma_j \}_{j \geq 0}, \{ \rho_j \}_{j \geq 0}$ be nonincreasing, non-negative sequences such that $2 a \gamma_j \leq b \rho_j$ for all $j$, {\blue and $\rho_0 < 1/b$}. 
Then, it holds that
\beqq
\sum_{j=0}^k \gamma_j \prod_{\ell=j+1}^k (1 - \gamma_\ell a) \prod_{i=0}^j (1 - \rho_i b ) \leq \frac{1}{a} \prod_{\ell=0}^k (1 - \gamma_\ell a), \; \forall\; k\ge 0.
\eeqq
\end{Lemma}

\ifonlineapp
\section{Technical Results Omitted from the Main Paper}

\subsection{Proof of \Cref{lem:aux2}}
To derive this result, we observe that 
\begin{align*}
\begin{split}
\sum_{j=0}^k \gamma_j^q \prod_{\ell=j+1}^k (1 - \gamma_\ell a) & \leq \gamma_k^{q-1} \sum_{j=0}^k \gamma_j \frac{\gamma_j^{q-1}}{\gamma_k^{q-1}} \prod_{\ell=j+1}^k (1 - \gamma_\ell a) \\
& = \gamma_k^{q-1} \sum_{j=0}^k \gamma_j \prod_{\ell=j+1}^k \left( \frac{\gamma_{\ell-1}}{\gamma_{\ell}} \right)^{q-1}  (1 - \gamma_\ell a).
\end{split}
\end{align*}
Furthermore, from the conditions on $\gamma_\ell$,
\begin{align*}
\begin{split}
\left( \frac{\gamma_{\ell-1}}{\gamma_{\ell}} \right)^{q-1}  (1 - \gamma_\ell a) \leq \big( 1 + \frac{a}{2(q-1)} \gamma_\ell \big)^{q-1} (1 - \gamma_\ell a) & \leq 1 - \frac{a}{2} \gamma_\ell.
\end{split}
\end{align*}
Therefore,
\begin{align*}
\begin{split}
\sum_{j=0}^k \gamma_j^q \prod_{\ell=j+1}^k (1 - \gamma_\ell a) & \leq \gamma_k^{q-1} \sum_{j=0}^k \gamma_j \prod_{\ell=j+1}^k (1 - \frac{a}{2} \gamma_\ell ) \leq \frac{2}{a} \gamma_k^{q-1}.
\end{split}
\end{align*}
This concludes the proof.

\subsection{Proof of \Cref{lem:aux3}}
First of all, the condition $2 a \gamma_j \leq b \rho_j$ implies
\beqq \notag 
\frac{1-\rho_i b}{1 - \gamma_i a} \leq 1 - \rho_i b / 2,~\forall~i \geq 0.
\eeqq 
As such, we observe 
$\prod_{i=0}^j \frac{ 1 - \rho_i b }{ 1 - \gamma_i a } \leq \prod_{i=0}^j (1 - \rho_i b/2)$ and subsequently, {\small
\begin{align*}
\sum_{j=0}^k \gamma_j \prod_{\ell=j+1}^k (1 - \gamma_\ell a) \prod_{i=0}^j (1 - \rho_i b )
% =\Big[\prod_{\ell=0}^k (1 - \gamma_\ell a) \Big] \sum_{j=0}^k \gamma_j \prod_{i=0}^j \frac{1 - \rho_i b}{1 - \gamma_i a} 
\leq \Big[ \prod_{\ell=0}^k (1 - \gamma_\ell a)\Big] \sum_{j=0}^k \gamma_j \prod_{i=0}^j (1 - \rho_i b / 2).
\end{align*}}Furthermore, for any $j=0,...,k$, it holds
\beqq
\rho_j \prod_{i=0}^{j-1} (1 - \rho_i b / 2) = \frac{2}{b} \Big[ \prod_{i=0}^{j-1} (1 - \rho_i b / 2) - \prod_{i=0}^{j} (1 - \rho_i b / 2) \Big],
\eeqq
where we have taken the convention $\prod_{i=0}^{-1} (1-\rho_i b/2) = 1$.
We obtain that{\small
\begin{align*}
\begin{split}
\sum_{j=0}^k \gamma_j \prod_{i=0}^j (1 - \rho_i b / 2) & \leq \frac{b}{2a} \sum_{j=0}^k \rho_j \prod_{i=0}^{j-1} (1 - \rho_i b / 2) \\
& = \frac{1}{a} \sum_{j=0}^k \Big[ \prod_{i=0}^{j-1} (1 - \rho_i b / 2) - \prod_{i=0}^{j} (1 - \rho_i b / 2) \Big] \leq \frac{1}{a},
\end{split}
\end{align*}}where the last inequality follows from the bound $1 - \prod_{i=0}^k (1 - \rho_i \mu_g b/2 ) \leq 1$. 
Combining with the above inequality yields the desired results.

% \subsection{Computing Biased Samples of $\Bgrd f(x^k,y^{k+1})$} 
\subsection{Proof of \Cref{lem:hessinv_main}}
\label{sec:bias}

\begin{proof}
Since the samples are drawn independently, the expected value of $h_f^k$ is
\beq \label{eq:exp_hfk} \textstyle
\EE[ h_f^k ] = \grd_x f(x,y) - \grd_{xy}^2 g(x,y) \EE\big [ \frac{\tmax \, \mu_g}{L_g (\mu_g^2 + \sigma_{gxy}^2)} \prod_{i=1}^{\sf p} \big (  I - \frac{\mu_g}{L_g (\mu_g^2 + \sigma_{gxy}^2)} \grd_{yy}^2 g(x,y ; \xi_i^{(2)} ) \big)   \big ]   \grd_y f(x,y).
\eeq
We have 
\begin{align*}
\begin{split}
& \| \Bgrd f(x,y) - \EE[ h_f^k ] \| \\
& \textstyle = \left\| \grd_{xy}^2 g(x,y) \big\{  \EE\big [ \frac{\tmax \mu_g}{L_g (\mu_g^2 + \sigma_{gxy}^2)} \prod_{i=1}^{\sf p} \big(  I - \frac{\mu_g}{L_g (\mu_g^2 + \sigma_{gxy}^2)} \grd_{yy}^2 g(x,y ; \xi_i^{(2)} )  \big)  \big ]   - \bigl [ \grd_{yy} g(x,y) \bigr ] ^{-1} \big\} \grd_y f(x,y) \right\| \\
& \textstyle \leq C_{gxy} C_{fy} \left\| \EE\big [ \frac{\tmax \mu_g}{L_g (\mu_g^2 + \sigma_{gxy}^2)} \prod_{i=1}^{\sf p} \big(  I - \frac{\mu_g}{L_g (\mu_g^2 + \sigma_{gxy}^2)} \grd_{yy}^2 g(x,y ; \xi_i^{(2)} )  \big) \big ]   - \bigl [ \grd_{yy} g(x,y) \bigr ] ^{-1} \right\|,
\end{split}
\end{align*}
where the last inequality follows from \Cref{ass:f}-3 and \Cref{ass:g}-5.

Applying \cite[Lemma 3.2]{ghadimi2018approximation}, the latter norm can be bounded by  $\frac{1}{\mu_g}\big(1 - \frac{\mu_g^2}{L_g (\mu_g^2+\sigma_{gxy}^2)} \big)^{\tmax}$. This concludes the proof for the first part of the lemma.

It remains to bound the variance of $h_f^k$. 
We first let 
\[
  H_{yy} = \frac{\tmax \, \mu_g}{L_g (\mu_g^2 + \sigma_{gxy}^2)} \prod_{i=1}^{\sf p} \big(  I - \frac{\mu_g}{L_g (\mu_g^2 + \sigma_{gxy}^2)} \grd_{yy}^2 g(x,y ; \xi_i^{(2)} )  \big).
\] 
To estimate the variance of $h_f^k$, using \eqref{eq:exp_hfk}, we observe that 
\begin{align*}
\begin{split}
& \EE[ \| h_f^k - \EE[ h_f^k ] \|^2 ] = \EE[ \| \grd_x f(x,y;\xi^{(1)} ) - \grd_x f(x,y) \|^2 ] \\
& \quad \quad + \EE\Big[ \Big\| \grd_{xy}^2 g(x,y; \xi_0^{(2)}) H_{yy} \grd_y f(x,y;\xi^{(1)}) - \grd_{xy}^2 g(x,y) \EE \big[ H_{yy} \big] \grd_y f(x,y) \Big\|^2 \Big] . 
\end{split}
\end{align*}
The first term on the right hand side can be bounded by $\sigma_{fx}^2$. Furthermore
\begin{align*}
\begin{split}
& \grd_{xy}^2 g(x,y; \xi_0^{(2)}) H_{yy} \grd_y f(x,y;\xi^{(1)}) - \grd_{xy}^2 g(x,y) \EE \big[ H_{yy} \big] \grd_y f(x,y) \\
& = \bigl \{ \grd_{xy}^2 g(x,y; \xi_0^{(2)}) - \grd_{xy}^2 g(x,y) \big \} H_{yy} \grd_y f(x,y;\xi^{(1)}) \\
& \quad + \grd_{xy}^2 g(x,y) \{ H_{yy} - \EE[ H_{yy} ] \} \grd_y f(x,y; \xi^{(1)}) \\
& \quad + \grd_{xy}^2 g(x,y) \EE[ H_{yy} ] \bigl \{ \grd_y f(x,y; \xi^{(1)}) - \grd_y f(x,y) \bigr \}. 
\end{split}
\end{align*}
We also observe
\beqq \notag
\EE[ \| \grd_y f(x,y;\xi^{(1)}) \|^2 ] = \EE[ \| \grd_y f(x,y;\xi^{(1)}) - \grd_y f(x,y) \|^2 ] + \EE[ \| \grd_y f(x,y) \|^2 ] \leq \sigma_{fy}^2 + C_y^2 .
\eeqq
Using $(a+b+c)^2 \leq 3(a^2 + b^2 +c ^2)$ and the Cauchy-Schwarz's inequality, we have 
\begin{align*}
\begin{split}
& \EE\Big[ \Big\| \grd_{xy}^2 g(x,y; \xi_0^{(2)}) H_{yy} \grd_y f(x,y;\xi^{(1)}) - \grd_{xy}^2 g(x,y) \EE \big[ H_{yy} \big] \grd_y f(x,y) \Big\|^2 \Big] \\
& \leq 3 \Big\{ (\sigma_{fy}^2 + C_y^2) \{ \sigma_{gxy}^2   \EE[ \| H_{yy} \|^2 ] + C_{gxy}^2    \EE[ \| H_{yy} - \EE[ H_{yy} ] \|^2 ] \} + \sigma_{fy}^2 C_{gxy}^2 \cdot  \| \EE[H_{yy}] \|^2 \Big\} . 
\end{split}
\end{align*}
Next, we observe that 
\beq 
\EE[ \| H_{yy} \|^2 ] = \frac{\mu_g^2}{L_g^2 (\mu_g^2 + \sigma_{gxy}^2)^2} \sum_{p=0}^{\tmax-1} \EE\left[ \left\| \prod_{i=1}^p \big (  I - \frac{\mu_g}{L_g (\mu_g^2 + \sigma_{gxy}^2)} \grd_{yy}^2 g(x,y ; \xi_i^{(2)} ) \big) \right\|^2 \right] 
\eeq 
Observe that the product of random matrices satisfies the conditions in \Cref{lem:product} with $\mu \equiv \frac{ \mu_g^2 }{L_g (\mu_g^2+\sigma_{gxy}^2)}$, $\sigma^2 \equiv \frac{\mu_g^2 \sigma_{gxy}^2}{L_g^2(\mu_g^2+\sigma_{gxy}^2)^2}$. Under the condition $L_g \geq 1$, it can be seen that 
\[
(1-\mu)^2 + \sigma^2 \leq 1 - {\mu_g^2} / ({L_g ( \mu_g^2 + \sigma_{gxy}^2 )}) < 1.
\]
Applying \Cref{lem:product} shows that  
\[
  \EE\left[ \left\| \prod_{i=1}^p \big (  I - \frac{\mu_g}{L_g (\mu_g^2 + \sigma_{gxy}^2)} \grd_{yy}^2 g(x,y ; \xi_i^{(2)} ) \big) \right\|^2 \right] \leq d_1 \left( 1 - \frac{\mu_g^2}{L_g ( \mu_g^2 + \sigma_{gxy}^2 )} \right)^p.
\]
Subsequently, 
\[
  \EE[ \| H_{yy} \|^2 ] \leq \frac{d_1}{L_g (\mu_g^2 + \sigma_{gxy}^2)} .
\]
% a slight modification of the proof of Lemma 3.2 in  \cite{ghadimi2018approximation} gives
% \beqq
% \| \EE[ H_{yy} ] \|^2 \leq (\EE[ \| H_{yy} \| ])^2 \leq \mu_g^{-2}.
% \eeqq
\footnote{Notice that a slight modification of the proof of \cite[Lemma 3.2]{ghadimi2018approximation} yields $(\EE[ \| H_{yy} \| ])^2 \leq \mu_g^{-2}$. However, the latter lemma requires $I - (1/L_g) \grd_{yy}^2 g(x,y, \xi_i^{(2)})$ to be a PSD matrix almost surely, which is not required in our analysis.}Furthermore, it is easy to derive that $\| \EE[ H_yy ] \| \leq 1 / \mu_g$. 
Together, the above gives the following estimate on the variance:
\beqq
\EE[ \| h_f^k - \EE[ h_f^k ] \|^2 ] \leq \sigma_{fx}^2 + \Big\{  (\sigma_{fy}^2 + C_y^2)  \{ \sigma_{gxy}^2 + 2 C_{gxy}^2 \} + \sigma_{fy}^2 C_{gxy}^2 \Big\} \max\big\{ \frac{3}{\mu_g^2}, \frac{3d_1}{L_g ( \mu_g^2 + \sigma_{gxy}^2)} \big\}.
\eeqq
This concludes the proof for the second part of the lemma. 
\end{proof}
% \vspace{-0.2cm}

We observe the following lemma on the product of (possibly non-PSD) matrices, which is inspired by \cite{durmus2021tight, huang2021matrix}:
\begin{Lemma}\label{lem:product}
Let $Z_i, i=0,1,...$ be a sequence of random matrices defined recursively as $Z_i = Y_i Z_{i-1}$, $i \geq 1$, with $Z_0 = I \in \RR^{d \times d}$, and $Y_i, i=0,1,...$ are independent, symmetric, random matrices satisfying $\| \EE[Y_i] \| \leq 1 - \mu$ and $\EE[ \| Y_i - \EE[Y_i] \|_2^2 ] \leq \sigma^2$. If $(1-\mu)^2+\sigma^2 < 1$, then for any $t \geq 0$, it holds
\beq
\EE[ \| Z_t \|^2 ] \leq \EE[ \| Z_t \|_2^2 ] \leq d \left( (1-\mu)^2 + \sigma^2 \right)^t ,
\eeq
where $\| X \|_2$ denotes the Schatten-2 norm of the matrix $X$.
\end{Lemma}
\begin{proof} 
We note from the norm equivalence between spectral norm and Schatten-2 norm which yields $\| Z_t \| \leq \| Z_t \|_2$ and thus $\EE[ \|Z_t\|^2 ] \leq \EE[ \| Z_t \|_2^2 ]$. 
For any $i \geq 1$, we observe that 
\[
  Z_i = \underbrace{(Y_i - \EE[Y_i]) Z_{i-1} }_{= A_i} + \underbrace{ \EE[Y_i] Z_{i-1} }_{= B_i}.
\] 
Notice that as $\EE[ A_i | B_i ] = 0$, applying \cite[Proposition 4.3]{huang2021matrix} yields
\beq \label{eq:prod_substitute}
  \EE[ \| Z_t \|_2^2 ] \leq \EE[ \| A_t \|_2^2 ] + \EE[ \| B_t \|_2^2 ].
\eeq
Furthermore, using the fact that $Y_i$s are independent random matrices and H\"{o}lder's inequality for matrices, we observe that
\[ 
  \EE[ \| A_t \|_2^2 ] \leq \EE[ \| Y_t - \EE[Y_t] \|_2^2 \| Z_{t-1} \|_2^2 ] = \EE[ \| Y_t - \EE[Y_t] \|_2^2 ] \EE[ \| Z_{t-1} \|_2^2 ] \leq \sigma^2 \EE[ \| Z_{t-1} \|_2^2 ]
\]
and using \cite[(4.1)]{huang2021matrix},
\[
  \EE[ \| B_t \|_2^2 ] \leq \| \EE[Y_t] \|^2 \EE[ \| Z_{t-1} \|_2^2 ] \leq (1- \mu)^2 \EE[ \| Z_{t-1} \|_2^2 ]
\]
Substituting the above into \eqref{eq:prod_substitute} yields 
$\EE[ \| Z_t \|_2^2 ] \leq ( (1-\mu)^2 + \sigma^2 ) \EE[ \| Z_{t-1} \|_2^2 ]$. Repeating the same arguments for $t$ times and using $\| I \|_2^2 = d$ yields the upper bound. 
\end{proof}

% Finally, let us discuss the consequence of \Cref{lem:hessinv} on the sample complexity of the {\sf TTSA} algorithm under the bias assumptions in \Cref{th:sc:uc}, \ref{th:wc:c}, \Cref{cor:c:c}. 
% In particular, \eqref{eq:bias_hfk} shows that 
% \beqq
% \EE[ h_f^k ] = \Bgrd f( x^k,y^{k+1} ) + B_k \quad \text{with} \quad \| B_k \| \leq b_k = {\cal O}\big( (1 - \mu_g / L_g)^t \big).
% \eeqq
% Recall that the theorems require $b_k \leq {\rm c}_b k^{-a}$ for some $a>0$. To satisfy this, one only needs $t = \Theta( \log k )$. The  $k$-th iteration of {\sf TTSA} requires $1 + \Theta( \log k )$ queries of stochastic oracles. 
\fi

		% \bibliographystyle{abbrvnat}
		% \bibliography{ref}
		
		\ifonlineapp
		{
		\linespread{0.9}\small 
		\bibliographystyle{plain}
		\bibliography{ref-bi}
		}
		\else
		{
		\linespread{0.9}\small 
		\bibliographystyle{ims}
		\bibliography{ref-bi}
		}
		\fi

		%\linespread{1.1}
		%\normalsize
		
		%-----------------------------------------------------------------------------
		%\vspace{0.4cm}
		
	\end{document}